\documentclass[11pt]{amsart}

\usepackage{amsmath,mathtools}
\usepackage{amsmath, amsthm, amssymb, amsfonts, enumerate}
\usepackage{bbm}
\usepackage{graphicx}
\usepackage{float}
\usepackage{natbib}
\usepackage[dvipsnames]{xcolor}
\usepackage{ulem}
\usepackage[colorlinks=true,linkcolor=blue,citecolor=blue,urlcolor=blue]{hyperref}
\usepackage{dsfont}
\usepackage{color}
\usepackage{geometry}
\usepackage{todonotes}
\usepackage{epstopdf}
\usepackage{tikz}
\usetikzlibrary{patterns}
\usetikzlibrary{patterns.meta}
\usetikzlibrary{arrows.meta}
\usepackage{pgfplots}
\usepackage{caption}
\usepgfplotslibrary{fillbetween}
\graphicspath{ {./Figures/} }

\newcommand{\stkout}[1]{\ifmmode\text{\sout{\ensuremath{#1}}}\else\sout{#1}\fi}

\newtheorem{theorem}{Theorem}[section]
\newtheorem{remark}[theorem]{Remark}
\newtheorem{assumption}[theorem]{Assumption}
\newtheorem{lemma}[theorem]{Lemma}
\newtheorem{proposition}[theorem]{Proposition}
\newtheorem{corollary}[theorem]{Corollary}
\newtheorem{definition}[theorem]{Definition}

\newcommand{\Exp}{\operatornamewithlimits{Exp}}

\DeclareMathOperator*{\argmax}{arg\,max}

\definecolor{darkgreen}{rgb}{0.00,0.35,0.00}

\def\E{\mathsf{E}}

\renewcommand{\hat}{\widehat}
\renewcommand{\tilde}{\widetilde}

\title[Stationary Ergodic Constrained Mean-field Games]{Stationary Ergodic Mean-field Games of\\Constrained Singular Controls}

\author[Ferrari]{Giorgio Ferrari}
\author[Lempa]{Jukka Lempa}
\author[Rodosthenous]{Neofytos Rodosthenous}
\author[Saarinen]{Harto Saarinen}

\address{G.~Ferrari: Center for Mathematical Economics (IMW), Bielefeld University, Universit\"atsstrasse 25, 33615, Bielefeld, Germany}
\email{\href{mailto:giorgio.ferrari@uni-bielefeld.de}{giorgio.ferrari@uni-bielefeld.de}}
\address{J.~Lempa: Department of Mathematics and Statistics, University of Turku, FI-20014 Turun yliopisto, Finland}
\email{\href{mailto:jumile@utu.fi}{jumile@utu.fi}}
\address{N.~Rodosthenous: Department of Mathematics, University College London, 25 Gordon Street, London WC1H 0AY, United Kingdom}
\email{\href{mailto:n.rodosthenous@ucl.ac.uk}{n.rodosthenous@ucl.ac.uk}}
\address{H.~Saarinen: Department of Economics, Turku School of Economics, University of Turku, FI-20014 Turun yliopisto, Finland}
\email{\href{mailto:hoasaa@utu.fi}{hoasaa@utu.fi}}

\date{\today}

\numberwithin{equation}{section}

\begin{document}

% \begin{abstract}
% We study a stationary ergodic mean-field game of scalar type in which the representative player intervenes on a one-dimensional diffusion by two-sided singular controls that can be exerted only at the jump times of an independent Poisson process. 
% For a fixed mean-field parameter, we firstsolve the associated novel constrained ergodic singular control problem and show that the optimal policy is of threshold type, prescribing that the diffusion is reset only at Poisson arrival times to the nearest boundary point of a suitable interval.
% A key contribution is the  explicit probabilistic derivation of the stationary distribution induced by Poissonian singular controls of threshold type, which is of independent interest. Such stationary distribution is then used to construct the mean-field fixed-point map. We prove existence and uniqueness of the stationary mean-field equilibrium by combining the monotonicity of the optimal thresholds, their continuity in the mean-field parameter, and the explicit stationary density. 
% The model provides a tractable framework for a case study of  partially reversible investment in which adjustments are singular and subject to exogenous intervention constraints.
% \end{abstract}

\begin{abstract}
We study a stationary ergodic mean-field game of scalar type where the representative player controls a diffusion through two-sided singular controls exerted only at the jump times of an independent Poisson process. For each mean-field parameter, we solve the novel constrained ergodic singular control problem, proving that the optimal policy is of threshold type: the diffusion is reset at Poisson arrival times to the nearest boundary point of an interval. Our explicit probabilistic derivation of the stationary distributions induced by Poissonian threshold singular controls is a key contribution of independent interest. This distribution is used to construct the mean-field fixed-point map. We prove existence and uniqueness of the stationary mean-field equilibrium by combining the stationary density with the monotonicity and continuity of the optimal thresholds in the mean-field parameter. The model provides a tractable framework for a case study of partially reversible investment with singular adjustments under exogenous intervention constraints.
\end{abstract}

\maketitle



\section{Introduction}

Singular stochastic control is a natural mathematical formulation for decision problems in which an agent can adjust the state of a system by cumulative, possibly irreversible or partially reversible, interventions. In operations research and mathematical finance, such controls arise in inventory and capacity management, dividend and reserve problems, investment under costly reversibility, and liquidity management. In many of these applications, however, the controller is unable to intervene continuously. 
Orders may be processed only at inspection times, trades may be possible only when liquidity permits, dividends may be paid only at specific dates, and managerial attention or regulatory permissions may occur at random times. A convenient and analytically tractable way to capture this feature is to restrict interventions to the jump times of an exogenous Poisson process.

This paper studies a stationary mean-field game in which this intervention constraint is combined with two-sided singular control and an ergodic criterion. 
The representative player controls a one-dimensional diffusion by upward and downward bounded-variation interventions, which can be implemented only when an independent Poisson process jumps. For a given mean-field parameter $\theta$, the agent maximises a long-run average payoff consisting of a running reward and proportional intervention costs or profits. 
The mean-field interaction is of scalar type: at equilibrium, $\theta$ must coincide with a prescribed function of a generalised moment of the stationary distribution of the optimally controlled state. 

\subsection{Our results}
Our analysis is constructive. We first fix the mean-field parameter and solve the representative player's constrained ergodic singular control problem. We show that the optimal policy is of threshold type: whenever the uncontrolled diffusion is observed below a lower threshold at a Poisson time, it is pushed up to that threshold, and whenever it is observed above an upper threshold, it is pushed down to that threshold. The two thresholds solve an explicit nonlinear algebraic system involving the speed measure, scale function, and the fundamental solutions $\psi_\lambda $ and $\varphi_\lambda $ of the uncontrolled diffusion.

The first main mathematical contribution is the derivation of the stationary distribution generated by Poissonian singular control processes of threshold type. 
For %one-sided downward or upward controls, 
the two-sided controlled process, we identify the invariant law through a renewal-cycle representation, the derivation of occupation measures for one-sided downward or upward controlled processes, and resolvent identities for the underlying diffusion, which eventually allow us to obtain a closed-form expression for the stationary distribution. %We then combine the one-sided occupation measures to obtain a closed-form expression for the stationary distribution in the two-sided case. 
For a threshold pair $(a,y)$, the stationary density with respect to the speed measure is proportional to
\[
\begin{cases}
\psi_\lambda (x)/\psi_\lambda (a), & x\leq a,\\
1, & a<x<y,\\
\varphi_\lambda (x)/\varphi_\lambda (y), & x\geq y,
\end{cases}
\]
with an explicit normalising constant. 
In addition, even though not needed for solving the MFG studied in this paper, using the same occupation measures we also obtain the stationary distributions of the one-sided downward and upward controlled processes in Section~\ref{sec:density1}.
To the best of our knowledge, these derivations are novel in the literature, which carries its own independent value, as they can be useful in other such constrained ergodic singular control problems and games.
It also provides a transparent comparison with the continuously reflected case: the mass outside the continuation interval $(a,y)$ is weighted by the fundamental solutions $\psi_\lambda $ and $\varphi_\lambda$ because excursions beyond the thresholds are adjusted only at Poisson arrival times.

The second contribution is the solution of the two-sided constrained ergodic singular control problem. Using the explicit stationary distribution, we compute the long-run average payoff under a Poissonian singular control of threshold type, for a general pair of thresholds, and derive the first-order conditions for optimality. 
We then prove existence and uniqueness of the optimal pair of thresholds. The verification over the general admissible class is obtained through an associated free-boundary problem and a variational inequality structure adapted to the Poissonian intervention constraint. In particular, the marginal value is linked to an auxiliary constrained Dynkin game whose stopping times are restricted to Poisson arrival times. This connection is useful both for the verification argument and for the subsequent mean-field equilibrium analysis.

The third contribution is the existence and uniqueness of the stationary mean-field equilibrium. We prove the monotonicity of the optimal thresholds with respect to the mean-field parameter $\theta$ by exploiting the auxiliary Dynkin game and their continuity by using the algebraic characterisation of the first-order conditions. The explicit stationary density then yields continuity of the induced fixed-point operator $\mathcal T$.
A compact invariant set for $\mathcal T$ is obtained from the monotonicity of the thresholds and stochastic comparison. Existence follows from the Schauder--Tychonoff fixed-point theorem, while uniqueness follows from the monotonicity of the fixed-point map induced by the decreasing dependence of the optimal thresholds on $\theta$.
From a methodological point of view, the paper thus contributes to the analysis of ergodic singular control under random intervention opportunities and to the probabilistic construction of stationary mean-field equilibria.

Our framework accommodates a variety of applications, including mean-field game models of partially reversible investment in which agents face long-run objectives and interact through an aggregate stationary quantity. 
In the case study considered here, the mean-field parameter can be interpreted as a stationary market price or demand index determined by aggregate productivity, as in investment models with isoelastic demand (see, e.g.~\citep{AchdouBueraLasryLionsMoll2014,CalviaFedericoFerrariGozzi2024}). 

\vspace{-5mm}
\subsection{Related literature}

The paper is related first to the literature on ergodic singular control. For one-dimensional diffusions with continuous intervention opportunities, two-sided ergodic singular control problems have been studied, among others, by \citet{JackZervos2006}, \citet{Alvarez2018} and \citet{Kunwai2022}; see also \cite{calvia2025optimal} for a solvable class of multi-dimensional ergodic singular stochastic control problems. In those works, the optimal policy is typically characterised by reflection at endogenously determined thresholds. Our setting differs because reflection is not continuously available: the controller may act only at Poisson arrival times, and therefore the optimally controlled process can overshoot the continuation interval between intervention opportunities. This changes both the optimality system and the invariant distribution.

The random-intervention constraint connects our work to the literature on Poissonian control and stopping. \citet{Wang2001} studies control problems with random intervention times, while \citet{Lempa2014} and \citet{Saarinen2023} analyse discounted Poissonian control problems, including two-sided variants. The closest paper in ergodic control is \citet{LempaSaarinen2021}, which solves a one-sided ergodic Poissonian control problem. 
In this paper, we firstly consider a two-sided constrained ergodic control problem (instead of only one-sided) and subsequently embed this into a stationary mean-field game. 
Our probabilistic arguments also rely on ideas from Poissonian stopping problems and games, including \citet{DupuisWang2002}, \citet{Lempa2012}, \citet{Guo2005}, \citet{Hobson2021}, \citet{LiangSun2019}, \citet{LiangSun2020}, \citet{Lange2020}, \citet{Perez2021}, and \citet{Alvarez2024}. The latter is particularly relevant for first Poisson-sampled hitting times of diffusions.

Poissonian observation and intervention mechanisms also arise in several applied literatures. 
In insurance and risk theory, randomised or periodic observation has been used to model dividend and capital-injection policies; see, for example, \citet{Albrecher2011a}, \citet{Albrecher2011b}, \citet{Albrecher2011c}, \citet{Albrecher2013}, \citet{Avanzi2013}, \citet{Yu2021}, and \citet{Zhao2017}. Exit identities under Poisson observation have been investigated by \citet{Albrecher2016} and \citet{Albrecher2017}. In portfolio choice and liquidity models, random trading opportunities or liquidity shocks have been studied by \citet{RogersZane2002}, \citet{Matsumoto2006}, \citet{Matsumoto2007}, \citet{Pham2008}, \citet{Pham2009}, \citet{Pham2009b}, and \citet{Liang2015b}. 
These contributions provide an applied motivation for the Poissonian constraint; our focus is on its implications for ergodic singular control and mean-field games.

Finally, our work contributes to the literature on stationary mean-field games with singular controls and scalar interaction. Those have been studied in a series of recent contributions, with different specifications of the underlying process, of the equilibrium notion and of the uncertainty structure; see \cite{aid2025stationary,cao2023stationary,cao2022mfgs,cannerozzi2026cooperation,cannerozzi2026stationary,christensen2025two,dianetti2026ergodic, ferrari2025stationary}.
The scalar interaction and the investment interpretation of our paper are in line with \cite{aid2025stationary,cao2023stationary,cao2022mfgs,cannerozzi2026cooperation,dianetti2026ergodic,ferrari2025stationary}, aligning with the investment models considered in \citet{CalviaFedericoFerrariGozzi2024}, for the case of regular controls. The distinctive feature of the present paper is that the singular control is two-sided, the performance criterion is ergodic, and intervention opportunities are exogenously restricted to Poisson arrival times. To the best of our knowledge, stationary mean-field games with Poissonian singular intervention constraints have not been considered before.

\vspace{-3mm}
\subsection{Organisation of the paper}

The rest of the paper is organised as follows. Section \ref{sec:framework} introduces the probabilistic framework, the standing assumptions and constraints, and the considered constrained ergodic mean-field game. 
In Section~\ref{sec:density}, we derive the stationary distributions associated with two-sided Poissonian singular controls of threshold type. 
Then, we solve the representative player's constrained ergodic singular control problem in Section~\ref{sec:OCP}, and verify the optimality of a Poissonian singular control process of threshold type over the general admissible class. 
In Section~\ref{sec:MFG}, we establish the existence and uniqueness of the stationary mean-field equilibrium. 
Finally, Section~\ref{sec:illustration} illustrates the theory with a mean-field game in which agents make partially reversible investments in productivity to maximise profits determined by an equilibrium price arising from a competitive market model with isoelastic demand.
Our theoretical developments conclude in Section~\ref{sec:density1} with the derivation of the stationary distributions associated with one-sided upward and downward Poissonian singular controls of threshold type.

\section{Probabilistic Framework and the Mean-field Game Problem}
\label{sec:framework}

Let $(\Omega,\mathcal{F},\mathbb{F},\mathbb{P})$ be a filtered probability space,
where the filtration $\mathbb{F} := (\mathcal{F}_t)_{t \ge 0}$ satisfies the usual conditions.
We assume that the filtration $\mathbb{F}$ is generated by a one-dimensional
standard Brownian motion $W = (W_t)_{t \ge 0}$ and an independent Poisson process
$N = (N_t)_{t \ge 0}$ with intensity $\lambda > 0$, that is,
\[
\mathcal{F}_t = \sigma\big( W_s, N_s : 0 \le s \le t \big) \vee \mathcal{N},
\qquad t \ge 0,
\]
where $\mathcal{N}$ denotes the $\mathbb{P}$-null sets of $\mathcal{F}$.

We denote by $(T^{(i)})_{i \in \mathbb{N}}$ the jump times of $N$, with the convention
$T^{(0)} = 0$, and write
\[
0 = T^{(0)} < T^{(1)} < \cdots < T^{(n)} < \cdots,
\qquad T^{(\infty)} := \lim_{n \to \infty} T^{(n)} = \infty \quad \mathbb{P}\text{-a.s.}
\]
The inter-arrival times
\[
\tau_i := T^{(i)} - T^{(i-1)}, \qquad i \in \mathbb{N},
\]
are independent and identically distributed exponential random variables with
mean $\frac{1}{\lambda}$. 

Consider a one-dimensional diffusion $X$ taking values in $\mathcal{I} := (\mathfrak{l},\mathfrak{r}) \subseteq \mathbb{R}_+$, 
satisfying
\begin{equation} \label{Eq:UncontrolledDiffusion}
dX_t = b(X_t) \, dt + \sigma(X_t) \, dW_t, \quad X_0 = x \in \mathcal{I},
\end{equation}
where $b: \mathcal{I} \to \mathbb{R}$ and $\sigma: \mathcal{I} \to \mathbb{R}_+$ satisfy the following assumptions.

\begin{assumption}
\label{Ass:X}
The following hold:
\begin{enumerate}[\rm (i)]
\item The functions $b(\cdot) \in C^2(\mathcal{I})$ and $\sigma(\cdot) \in C^2(\mathcal{I})$; 
\item There exists $C>0$, such that
$|b(x)|+|\sigma(x)| \leq C(1+|x|)$, for any $x \in \mathcal{I}$; 
\item There exists $c>0$, such that 
$b'(x) \leq -c$, for any $x \in \mathcal{I}$. Furthermore, for any $T>0$, we have $\E_x[\exp\{\frac{1}{2}\int_{0}^T (\sigma\sigma')^2(X_t) dt\}] < \infty$, for any $x \in \mathcal{I}$; 
\item For any $x \in \mathcal{I}$, $\sigma(x)>0$; 
\item The boundaries $\mathfrak{l}$ and $\mathfrak{r}$ are assumed to be natural for the diffusion $X$.
\end{enumerate}
\end{assumption}

Assumptions \ref{Ass:X} (i), (ii), and (iv) imply that $X$ is the unique strong solution to the SDE \eqref{Eq:UncontrolledDiffusion} and is a regular diffusion. 
Assumption~\ref{Ass:X}.(iii) additionally imposes a uniformly mean-reverting drift, with the strength of mean reversion bounded away from zero, and guarantees the validity of a Novikov condition, that will play a role in the analysis of Section \ref{sec:MFG}.
Finally, Assumption (v) implies that the process cannot reach either boundary in finite time when started from the interior $(\mathfrak{l},\mathfrak{r})$, nor can it be started from either boundary.

\subsection{Probabilistic background on one-dimensional diffusions}
\label{sec:Diffusions}

Denote by $m$ and $S$ the speed measure and scale function of $X$, with
\begin{align} \label{Eq:Sm}
S'(x) := e^{-\mathcal{B}(x)}, \quad m(dx) := m'(x)dx, \quad m'(x) := \frac{2 e^{\mathcal{B}(x)}}{\sigma^2(x)}, 
\quad \mathcal{B}(x) := \int_{c_0}^x \frac{2 b(z)}{\sigma^2(z)} dz, \; c_0 \in \mathcal{I}.
\end{align}
The infinitesimal generator can then be written as
\begin{align} \label{Eq:GeneratorCanonical}
\mathcal{G} := \frac{\sigma^2(x)}{2} \frac{d^2}{dx^2} + b(x)\frac{d}{dx} = \frac{d}{dm} \frac{d}{dS}.
\end{align}
The Green kernel $G_\lambda : \mathcal{I}^2 \to \mathbb{R}_{+}$ is defined by
\begin{equation} \label{Eq:GreenKernel0}
G_\lambda(x,y) := \int_0^\infty e^{-\lambda t} p(t;x,y) \, dt,
\end{equation}
where $p(t;x,y)$ is the transition density of $X$ with respect to the speed measure. 
Let $\psi_\lambda$ (resp., $\varphi_\lambda$) be the unique, up to a multiplicative constant, increasing (resp., decreasing) solution of $\mathcal{G} u = \lambda u$. 
Then since $\mathfrak{l}$ and $\mathfrak{r}$ are natural (cf.~\citep[II.1.10--11]{BorodinSalminen}) 
\begin{equation*}
\begin{aligned}
&\psi_\lambda(\mathfrak{l}+) = 0, \quad && \frac{\psi_\lambda'(\mathfrak{l}+)}{S'(\mathfrak{l+})} = 0, \quad &&\varphi_\lambda(\mathfrak{l}+) = +\infty, \quad && \frac{\varphi_\lambda'(\mathfrak{l}+)}{S'(\mathfrak{l+})} = -\infty, \\
&\psi_\lambda(\mathfrak{r}-) = +\infty, \quad &&\frac{\psi_\lambda'(\mathfrak{r}-)}{S'(\mathfrak{r}-)}  = +\infty, \quad &&\varphi_\lambda(\mathfrak{r}-) = 0, \quad && \frac{\varphi_\lambda'(\mathfrak{r}-)}{S'(\mathfrak{r}-)} = 0,
\end{aligned}
\end{equation*}
with Wronskian $B_\lambda$ independent of $x$, such that 
\begin{equation} \label{Eq:GreenKernel1}
G_\lambda(x,y) =
\begin{cases} 
B_\lambda^{-1} \psi_\lambda(x) \varphi_\lambda(y), & x < y, \\
B_\lambda^{-1} \psi_\lambda(y) \varphi_\lambda(x), & x \ge y,
\end{cases}
\quad \text{with} \quad 
B_\lambda 
:= \frac{\psi_\lambda'(x) \varphi_\lambda(x) - \psi_\lambda(x) \varphi_\lambda'(x)}{S'(x)}.
\end{equation}
Moreover, we define by $\tau_y := \inf \{ t \geq 0 \mid X_t = y \}$ the hitting time of $y \in \mathcal{I}$ and note that
\begin{equation} \label{Eq:hitprob}
h(x; y) := \mathbb{P}_{x}(\tau_{y}<\infty)
=\begin{cases}
\frac{S(x)-S(\mathfrak{l}+)}{S(y)-S(\mathfrak{l}+)}, & x\le y,\\
\frac{S(\mathfrak{r}-)-S(x)}{S(\mathfrak{r}-)-S(y)}, & x \ge y.
\end{cases}
\end{equation}
and denote by $G_0^y$ the Green kernel of $X$ killed at $\tau_y$, which is defined by
\begin{equation*}
G_0^y(x,z) := \int_0^\infty p^y(t;x,z) \, dt,
\end{equation*}
where $p^y(t;x,z)$ is the transition density of $X$ killed upon hitting $y$. 
We have 
\begin{equation} \label{Eq:GreenKernelKilled}
G_0^y(x,z) =  
\frac{\mathbb{E}_x[\int_0^{\tau_y} \mathbf 1 (X_t\in dz) \, dt]}{m(dz)}
= \begin{cases}
\frac{(S(x \wedge z)-S(\mathfrak{l}+))(S(y)-S(x \vee z))}{S(y)-S(\mathfrak{l}+)}, & x,z<y, \\
\frac{(S(x \wedge z)-S(y))(S(\mathfrak{r}-)-S(x \vee z))}{S(\mathfrak{r-})-S(y)}, & x,z>y, \\
0, & %\hspace{-1.8cm} 
(x-y)(z-y)<0.
\end{cases}
\end{equation}

For any fixed $\lambda>0$, denote by $\mathcal{L}^1(\mathcal{I})$ the set of measurable functions $f:\: \mathcal{I} \rightarrow \mathbb{R}_+$ such that
\begin{align*}
\mathbb{E}_x\left[ \int_0^{\infty}e^{-\lambda s}f(X_s) ds \right] &< \infty,
\end{align*}
and define, for any $f \in \mathcal{L}^1(\mathcal{I})$, the resolvent $R_\lambda$ by
\begin{align} \label{Eq:Resolvent0}
(R_\lambda f)(x) &:= \mathbb{E}_x\left[ \int_0^{\infty}e^{-\lambda s}f(X_s) ds \right].
\end{align}
By combining the identities from \citep[II.1.4]{BorodinSalminen} with (\ref{Eq:GreenKernel1}), we observe that the resolvent $R_\lambda$ can be equivalently expressed in terms of $\psi_\lambda$ and $\varphi_\lambda$ in the form
\begin{align} \label{Eq:Resolvent1}
(R_\lambda f)(x) &= B_\lambda^{-1} \left( \varphi_\lambda(x) (\Psi_\lambda f)(x) + \psi_\lambda(x) (\Phi_\lambda f)(x) \right),
\end{align}
where the functionals $\Psi_\lambda$ and $\Phi_\lambda$ are defined by
\begin{align} \label{PsiPhi}
(\Psi_\lambda f)(x) &:= \int_{\mathfrak{l}}^x \psi_\lambda(z) f(z) \, m(dz) 
\quad \text{and} \quad 
(\Phi_\lambda f)(x) := \int_x^{\mathfrak{r}} \varphi_\lambda(z) f(z) \, m(dz) .
\end{align}

% Recall from \citep[Cor. 3.2]{Alvarez2004} that when the boundaries are natural for $X$, we have for any $f,h \in C^2(\mathcal{I})$ that
% \begin{align} 
% \label{Eq:GeneratorFormula1}
% \frac{f'(x)\psi_\lambda(x) - \psi_\lambda '(x)f(x)}{S'(x)} &= \int_{\mathfrak{l}}^x \psi_\lambda(z)((\mathcal{G} - \lambda)f)(z) \, m(dz) ,\\ 
% \label{Eq:GeneratorFormula2}
% \frac{\varphi_\lambda '(x)h(x) - h'(x)\varphi_\lambda(x)}{S'(x)} &= \int_x^{\mathfrak{r}}\varphi_\lambda (z)((\mathcal{G} - \lambda)h)(z) \, m(dz).
% \end{align}
In the subsequent analysis we will make use of the identity 
\begin{equation} \label{Eq:DriftToScale}
\int_x^y b(z) \, m(dz) = \frac{1}{S'(y)}-\frac{1}{S'(x)},
\end{equation}
which follows from differentiating $1/S'(x)$ and using $\mathcal{G}S(x) = 0$, where $\mathcal{G}$ is defined by \eqref{Eq:GeneratorCanonical}. 
Notice also that \eqref{Eq:DriftToScale}, $\mathcal{G}\varphi_\lambda=\lambda\varphi_\lambda$, the definitions of $S'$ and $m'$ from \eqref{Eq:Sm}, imply   
\begin{align*}
\frac{d}{dx} \bigg( x \, \frac{\varphi_{\lambda}'(x)}{S'(x)} - \frac{\varphi_{\lambda}(x)}{S'(x)} \bigg) 
= -m'(x) \, \varphi_{\lambda}(x) \, \big( b(x)- \lambda x \big).
\end{align*}
Then, using the fundamental theorem of calculus, the natural boundary $\mathfrak{r}$, %yields \begin{equation*} y \, \frac{\varphi_{\lambda}'(y)}{S'(y)} - \frac{\varphi_{\lambda}(y)}{S'(y)} = \int_{y}^{\mathfrak{r}} m'(z) \, \varphi_{\lambda}(z) \, \big( b(z)- \lambda z \big) \, dz. \end{equation*} Then using 
the identity $\varphi_\lambda'(x) = - \lambda \, S'(x) \, (\Phi_{\lambda}1)(x)$,
and the fact that $\text{id} \in \mathcal{L}^1(\mathcal{I})$ thanks to Assumption~\ref{Ass:X} (where $\text{id}(x) := x$), we finally conclude that 
\begin{equation} \label{Eq:zyInt}
\lambda \int_y^{\mathfrak{r}} (z-y) \, \varphi_\lambda(z) \, m(dz) 
= \int_y^{\mathfrak{r}} b(z) \, \varphi_\lambda(z) \, m(dz) + \frac{\varphi_\lambda(y)}{S'(y)}. 
\end{equation}
Similarly we have
\begin{equation} \label{Eq:zaInt}
\lambda \int_{\mathfrak{l}}^a (z-a) \, \psi_\lambda(z) \, m(dz) 
= \int_{\mathfrak{l}}^a b(z) \, \psi_\lambda(z) \, m(dz) - \frac{\psi_\lambda(a)}{S'(a)}.
\end{equation}

\subsection{The constrained controlled diffusion dynamics}
\label{sec:controlledmodel}

% The controlled state process $X^\nu$ satisfies the semimartingale dynamics
% \begin{equation*}
% dX_t^\nu = b(X_{t-}^\nu) \, dt + \sigma(X_{t-}^\nu) \, dW_t + d\nu_t
% = b(X_{t-}^\nu) \, dt + \sigma(X_{t-}^\nu) \, dW_t + d\nu_t^+ - d\nu_t^- , \quad t \ge 0.
% \end{equation*}

We now introduce a controlled process $X^\nu$, where a controller (decision maker) can act on the original diffusion $X$ defined by \eqref{Eq:UncontrolledDiffusion} only at the jump times $(T^{(i)})_{i \in \mathbb{N}}$ of the independent Poisson process $N$. Intervention opportunities therefore arrive with rate $\lambda > 0$.

Intuitively, between intervention times, the controlled process follows the original diffusion dynamics:
\begin{equation*}
dX_t^\nu = b(X_t^\nu) \, dt + \sigma(X_t^\nu) \, dW_t, \quad t \in (T^{(i)}, T^{(i+1)}), \quad i \in \{0\} \cup \mathbb{N}.
\end{equation*}
At each $T^{(i)}$, the controller may exert a bounded-variation intervention, resulting in a jump of $X^\nu$ at $T^{(i)}$. The post-intervention value $X^\nu_{T^{(i)}}$ then serves as the initial condition for the diffusion on $(T^{(i)}, T^{(i+1)})$.
To formalise this, we consider the following set of admissible controls. 

\begin{definition}[Poissonian singular control processes]
\label{def:adm}
A control process $\nu : \Omega \times \mathbb{R}_{+} \to \mathbb{R}$ is called a Poissonian singular control process if $t \mapsto \nu_t$ is right-continuous, (locally) of bounded variation, $\nu$ is $\mathbb{F}$-adapted, and admits the representation 
\begin{equation*}
\nu_t = \int_0^t \eta_s dN_s, \quad t \geq 0, 
\quad \text{with} \quad 
\nu_0 = 0,
\end{equation*}
where $\eta$ is a predictable %($\mathbb{F}$-adapted, c\'agl\'ad) 
process. 
The set of Poissonian singular controls is denoted by $\mathcal{V}$.
\end{definition}

Note that any $\nu \in \mathcal{V}$ admits the minimal decomposition $\nu = \nu^+ - \nu^-$, where $\nu^+$ and $\nu^-$ are non-decreasing, right-continuous, $\mathbb{F}$-adapted processes with 
$\nu_0^+ = \nu_0^- = 0$. Moreover, we may write   
\[
\nu_t^+ = \int_0^t \eta_s^+ \, dN_s
\quad \text{and} \quad 
\nu_t^- = \int_0^t \eta_s^- \, dN_s,
\quad t \geq 0,
\]
such that $\eta^+, \eta^-$ are predictable processes satisfying $\eta = \eta^+ - \eta^-$.
The controlled state process $X^\nu$ then satisfies the semimartingale dynamics
\begin{equation}
\begin{split}
\label{Xnu}
dX_t^\nu 
&= b(X_{t-}^\nu)\,dt + \sigma(X_{t-}^\nu)\,dW_t + d\nu_t
= b(X_{t-}^\nu)\,dt + \sigma(X_{t-}^\nu)\,dW_t + d\nu_t^+ - d\nu_t^- \\
&= b(X_{t-}^\nu)\,dt + \sigma(X_{t-}^\nu)\,dW_t + \eta_t^+ \, dN_t - \eta_t^- \, dN_t
, \quad t \ge 0, \quad X^\nu_0 = x \in \mathcal{I}.
\end{split}
\end{equation}

In the rest of the paper, we adopt the following notation: 
$\mathbb{P}_{x}(\,\cdot\,) := \mathbb{P}(\,\cdot\,|X_{0}^{\nu}=x)$ 
and 
$\mathbb{E}_{x}[\,\cdot\,] := \mathbb{E}^{\mathbb{P}}[\,\cdot\, |X^{\nu}_{0}=x]$ for the controlled diffusion process $X^{\nu}$.

%%%%%%% REINTRODUCE IF NEEDED
% \begin{remark}
% Given that inter-arrival times $(\tau_i)_{i \in \mathbb{N}}$ of the Poisson process $N$ are independent $\Exp(\lambda)$ random variables, which are independent of $X$,
% we can compute the expectation of the post-jump state as follows. 
% For any $g \in \mathcal{L}^1(\mathcal{I})$ and inter-arrival time $\tau_i \sim \Exp(\lambda)$ independent of $X$, we have
% \[
% \mathbb{E}_x[ g(X^\nu_{\tau_i}) ] = \lambda (R_{\lambda} g)(x), \quad i \in \mathbb{N}.
% \]
% This relation will be instrumental in analysing the controlled process, as it links the Poisson jump structure directly to the resolvent operator \eqref{Eq:Resolvent0} of the underlying diffusion (see Section \ref{sec:Diffusions} for more details).
% \end{remark}

% We call a control policy  $(\nu_t^-, \nu_t^+)$ admissible if both processes are non-negative, non-decreasing, right-continuous and can be represented as
% \begin{equation*}
%     \nu_t^- = \int_0^{t-} \eta_s^- dN_s, \, \, \quad \nu_t^+= \int_0^{t-} \eta_s^+ dN_s,
% \end{equation*}
% where $\eta^-$ and $\eta^+$ are predictable. The set of admissible controls is denoted by $\mathcal{V}$.

\subsection{The constrained ergodic mean-field game}
\label{sec:ErgodicMFG}

In this section, we introduce the constrained ergodic mean-field game (MFG) for the controlled diffusion $X^\nu$ defined by \eqref{Xnu}, %in Section~\ref{sec:controlledmodel}, 
where interventions are allowed only at the jump times $(T^{(i)})_{i\in\mathbb{N}}$ of the independent Poisson process~$N$. %(see also Definition~\ref{def:adm}).

To begin with, we focus our attention to a subclass of the Poissonian singular control processes $\nu \in \mathcal{V}$ from Definition \ref{def:adm}. 
The class of admissible controls is defined by the set
\begin{equation}
\label{admissible controls} 
\begin{split}
\mathcal{A} := \bigg\{\nu \in \mathcal{V} \,:\, &X^{\nu}_{t} \in \mathcal{I}, \;\forall \, t \geq 0, \; \mathbb{P}\text{-a.s.}, \\
& 
\, \mathbb{E}_x \big[|\nu_{T}| \big]<\infty, \; \forall \, T < \infty, \, \text{and }  \limsup_{T\uparrow\infty}\frac{1}{T}\mathbb{E}_x \big[ |X^{\nu}_{T}| \big] = 0 \bigg\}.
\end{split}
\end{equation}
We then denote by $\pi: \mathcal{I} \times \mathbb{R}_{+} \to \mathbb{R}_+$ the running reward function of a representative player, which depends on the state $x \in \mathcal{I}$ and a (fixed, for the moment) parameter $\theta \in \mathbb{R}_{+}$ that drives the mean-field interaction of the continuum of players (see Definition \ref{def:Equilibrium} later on). 

For any initial state $x \in \mathcal{I}$, admissible control $\nu \in \mathcal{A}$, mean-field parameter $\theta \in \mathbb{R}_{+}$, and constants $q_u > q_d > 0$, we define the {\it long-run average payoff per unit time} (i.e.~the ergodic profit functional) of each representative player by 
\begin{equation} \label{Eq:ObjectiveImproved}
J(x; \nu, \theta) := \limsup_{T \to \infty} \, \frac{1}{T} \, \mathbb{E}_x \bigg[ \int_0^T \pi(X_s^\nu, \theta) \, ds - \int_0^T  q_u \, d\nu_s^+ + \int_0^T  q_d \, d\nu_s^- \bigg].
\end{equation} 
That is, $J(x; \nu, \theta)$ represents the asymptotic average rate at which rewards are earned over an infinite time horizon. 
The first integral term captures cumulative running rewards, while the second and third integrals are the control terms that reflect the impact of interventions: 
increasing the state incurs a proportional cost at rate $q_u$, whereas decreasing it yields a gain at rate $q_d$, with $q_u > q_d > 0$ encoding transaction frictions. 
Typical interpretations include the average profit of an inventory system (buying is costly, selling generates revenue), a cash/reserve model (capital injections versus dividend payouts), or a resource extraction problem (investment costs versus harvesting returns).
Note that $J$ is well defined for $\nu \in \mathcal{A}$, although potentially infinite.

We are now ready to provide the notion of stationary mean-field equilibrium in our setting.  
Consider functions $F:\mathbb{R}_{+} \to \mathbb{R}_{+}$ and $f:\mathcal{I}\to\mathbb{R}_{+}$ and fix a given $\theta \in \mathbb{R}_{+}$. 
Each representative player aims at optimising their individual long-run average payoff functional \eqref{Eq:ObjectiveImproved} over the constrained class $\mathcal{A}$ of admissible controls defined by \eqref{admissible controls}. 
Then, a generalised moment of the stationary distribution (if it exists) of the optimally controlled state process, representing the mean-field interaction of the continuum of players, is identified at equilibrium with $\theta$, which closes the equilibrium feedback loop. 
Hence, the considered stationary mean-field game belongs to the class of mean-field games with scalar interaction, but with the significant novelty that players are allowed to intervene only at the jump times of the independent Poisson process~$N$.

\begin{definition}[Constrained Ergodic MFG Equilibrium] \label{def:Equilibrium}
For an initial state $x \in \mathcal{I}$, a pair $(\nu^{*}(\theta^{*}),\theta^{*}) \in \mathcal{A} \times \mathbb{R}_{+}$ is an \textbf{equilibrium of the constrained ergodic MFG} if
\begin{enumerate}[\rm (1)]
\item \label{def:Equilibrium1} 
$J(x;\nu^{*}(\theta^{*}),\theta^{*}) \geq J(x;\nu,\theta^{*}), \text{ for any } \nu \in \mathcal{A}$;
\item \label{def:Equilibrium2} 
The optimally controlled state process $X^{\nu^{*}(\theta^{*})}$ admits a limiting stationary distribution $\mu^{\theta^{*}}$ and $\theta^{*}$ is such that 
$$
\theta^{*} = F\Big( \int_{\mathcal{I}} f(x) \mu^{\theta^{*}} (dx) \Big).
$$
\end{enumerate}
\end{definition}

The standing assumptions on the structure of the running reward function $\pi$ in our framework are as follows. 

\begin{assumption} \label{Ass:Payoff}
The running reward function $\pi: \mathcal{I} \times \mathbb{R}_{+} \to \mathbb{R}_+$ satisfies the following:
\begin{enumerate}[\rm (i)]
\item $\pi(\cdot,\theta) \in C^{2}(\mathcal{I})$, for any $\theta \in \mathbb{R}_{+}$; 
\item $\pi(\cdot,\theta)$ is non-decreasing and concave on $\mathcal{I}$, for any $\theta \in \mathbb{R}_{+}$; 
\item For any $\theta \in \mathbb{R}_{+}$, the functions
\[
\pi_u(\cdot,\theta) := \pi(\cdot, \theta) + q_u \, b(\cdot) 
\quad \text{and} \quad 
\pi_d(\cdot,\theta) := \pi(\cdot, \theta) + q_d \, b(\cdot), %\quad x \in \mathcal{I}, 
\]
admit unique maximisers $\hat{x}_u = \hat{x}_u(\theta) \in \mathcal{I}$ and $\hat{x}_d = \hat{x}_d(\theta) \in \mathcal{I}$, respectively, such that each $\pi_i(\cdot,\theta)$ is strictly increasing on $(\mathfrak{l}, \hat{x}_i)$ and strictly decreasing on $(\hat{x}_i, \mathfrak{r})$, for $i=u,d$, with 
\[
\lim_{x \to \mathfrak{l}} \pi_d(x,\theta) \ge \lim_{x \to \mathfrak{l}} \pi_u(x,\theta) \geq 0 
\quad \text{and} \quad 
\lim_{x \to \mathfrak{r}} \pi_d(x,\theta) =  \lim_{x \to \mathfrak{r}} \pi_u(x,\theta) =  -\infty,
\]
and additionally 
$\liminf_{x\downarrow\mathfrak l} \frac{\partial}{\partial x} \pi_u(x,\theta) > 0;$
\item For any $\theta \in \mathbb{R}_+$ there exists $\alpha_0(\theta), \alpha_1(\theta) \in \mathbb{R}_+$ such that $$\pi(x,\theta) \leq \alpha_0(\theta) |x|+\alpha_1(\theta), \quad \text{for all } x \in \mathcal{I};$$
\item $\frac{\partial}{\partial x} \frac{\partial}{\partial \theta}\pi(\cdot,\cdot)$ exists and is continuous and $\frac{\partial}{\partial x} \frac{\partial}{\partial \theta}\pi(x,\theta)<0$, for any $(x,\theta)\in\mathcal{I}\times \mathbb{R}_{+}$;
\item There exists $k:\mathcal{I} \to\mathbb{R}_{+}$ such that
\[
\lim_{\theta \downarrow 0} \frac{\partial}{\partial x}\pi(x,\theta) = \infty 
\quad \text{and} \quad 
\lim_{\theta \uparrow \infty} \frac{\partial}{\partial x}\pi(x,\theta) = k(x), 
\;\text{ for any } x \in \mathcal{I}, 
\]
and the functions
$
k_u(\cdot) := k(\cdot) + q_u \, b'(\cdot)$ and 
$k_d(\cdot) := k(\cdot) + q_d \, b'(\cdot)$, 
admit unique zeros $\tilde{x}_u, \tilde{x}_d \in \mathcal{I}$, respectively, such that each $\kappa_i(\cdot)$ is non-negative on $(\mathfrak{l}, \tilde{x}_i)$ and non-positive on $(\tilde{x}_i, \mathfrak{r})$, for $i=u,d$.
\end{enumerate}
\end{assumption}

% Assumption~\ref{Ass:Payoff}.(i)--(iii) imply that, for every fixed $\theta\in\mathbb{R}_{+}$, the intervention-adjusted rewards $\pi_u(\cdot,\theta)$ and $\pi_d(\cdot,\theta)$ are continuous and single-peaked, positive near the left boundary and eventually negative near the right boundary. In particular, each of them has a unique zero on its decreasing branch, which is the qualitative structure needed below to obtain an optimal policy for the constrained control problem in finite threshold-type strategies. The growth condition in Assumption~\ref{Ass:Payoff}.(iv) ensures that the long-run average payoff (\ref{Eq:ObjectiveImproved}) is finite for admissible controls. \bl{Assumption~\ref{Ass:Payoff}.(v) controls the limiting marginal reward at extreme values of the mean-field parameter and (vi) is used to show monotonicity of proposed strategies with respect to $\theta$.}
Assumptions~\ref{Ass:Payoff}.(i)--(iii) impose a natural shape on the reward structure: the running reward is smooth and non-decreasing in the state, while the intervention-adjusted rewards are single-peaked, reflecting that interventions are attractive at relatively low and intermediate states but become unattractive as the state becomes sufficiently large. Assumption~\ref{Ass:Payoff}.(iv) imposes a linear growth bound on the running reward in the state, ensuring that the long-run average payoff is well defined. Assumption~\ref{Ass:Payoff}.(v) specifies that stronger mean-field effects reduce the marginal benefit associated with an increase in the state. Finally, Assumption~\ref{Ass:Payoff}.(vi) characterises the behaviour of marginal rewards in the limiting regimes of the mean-field parameter, ensuring that the interaction between the mean-field environment and individual incentives remains well defined in these regimes.

We next present the standing assumptions on the functions $f,F$ defining the mean-field consistency condition \eqref{def:Equilibrium2} in Definition \ref{def:Equilibrium}.
\begin{assumption}
\label{Ass:Ff}
The functions $f:\mathcal{I}\to\mathbb{R}_{+}$ and $F:\mathbb{R}_{+} \to \mathbb{R}_{+}$ satisfy the following:
\begin{enumerate}[\rm (i)]
    \item $f(\cdot)$ and $F(\cdot)$ are both strictly increasing and continuously differentiable functions;
    \item There exists $\gamma \in(0,1)$ and $C=C(\gamma) > 0$ such that:
    \begin{enumerate}[\rm (a)]
    \item $|f(x)| \leq C (1+|x|^{\gamma})$ for all $x \in \mathcal{I}$ and $|F(m)| \leq C (1+|m|^{\frac{1}{\gamma}})$ for all $m \in \mathbb{R}_{+}$; 
    \item $\big| F(m_1)-F(m_2) \big| \leq C(1+|m_1|+|m_2|)^{\frac{1}{\gamma}-1} |m_1-m_2|$ for all $m_1, m_2 \in \mathbb{R}_{+}$;         \end{enumerate}
    \item $\lim_{x\uparrow \mathfrak{r}} f(x) = \lim_{m\uparrow\infty} F(m) = \infty$.
\end{enumerate}
\end{assumption}

\begin{remark} \label{Rem:Benchmark} 
As a benchmark example satisfying Assumptions \ref{Ass:Payoff} and \ref{Ass:Ff}, we take 
\begin{equation} \label{case}
\pi(x,\theta)=x^{\delta}(\theta^{-(1+\delta)}+\kappa_*),\quad F(m)=m^{1/\delta},\quad f(x)=x^{\delta}, \quad \kappa_* > 0, \quad \delta\in (0,1),
\end{equation}
and a controlled diffusion $X^\nu$ with dynamics given by $b(x)= - \rho x$ and $\sigma(x) = \sigma x$, for constants $\rho, \sigma>0$, which corresponds to a MFG of partially reversible investment: 
$X^\nu$ represents the (controlled) stochastically evolving productivity of a representative company, which depreciates with rate $\rho$ and can be increased (resp., decreased) via the cumulative investment (resp., disinvestment) process $\nu^+$ (resp., $\nu^-$); 
the running reward function~$\pi$ is accrued from the production of a final good, which depends on both the company's productivity level~$x \in \mathbb{R}_{+}$ and the equilibrium stationary price index~$\theta \in \mathbb{R}_{+}$ (see also \cite{cao2022mfgs}, \cite{CalviaFedericoFerrariGozzi2024} in a setting with regular controls, and \citep[Sec.~3]{AchdouBueraLasryLionsMoll2014} for a related real-option problem). 
\end{remark}

In the forthcoming analysis, our search for an ergodic mean-field game equilibrium as in Definition \ref{def:Equilibrium} is split into the following three steps:
\begin{enumerate}
    \item Assuming that, for any given and fixed mean-field parameter $\theta \in \mathbb{R}_{+}$, an optimal policy is of threshold type, we provide the complete characterisation of the stationary distribution $\mu^\theta$ of the state process when it is controlled via a two-sided Poissonian singular control of threshold type (see Section \ref{sec:density}).
    
    \item For given and fixed mean-field parameter $\theta \in \mathbb{R}_{+}$, we then verify the optimality of the candidate threshold-type policy for the ergodic problem in which \eqref{Eq:ObjectiveImproved} is optimised over $\mathcal{A}$. We thus determine the optimal control $\nu^*$, which will depend on $\theta$, i.e.~$\nu^* = \nu^*(\theta)$ (see Section \ref{sec:OCP});
    
    \item Making use of the stationary distribution derived in Step (1) and determined by the optimal controls in Step (2), we finally impose the consistency condition \eqref{def:Equilibrium2} of Definition~\ref{def:Equilibrium}, which is a fixed-point equation in $\theta$, and prove the existence and uniqueness of the mean-field parameter $\theta^{*}$ (see Section \ref{sec:MFG}).
\end{enumerate}
Consequently, we will know by construction that the pair $(\nu^{*}(\theta^{*}),\theta^{*}) \in \mathcal{A} \times \mathbb{R}_{+}$ is an equilibrium of the constrained ergodic MFG. 

Throughout our analysis of the constrained ergodic MFG in Sections~\ref{sec:OCP}--\ref{sec:MFG}, Assumptions \ref{Ass:X}, \ref{Ass:Payoff}, and \ref{Ass:Ff} will be in force and no further reference to them will be made.
The results in Section~\ref{sec:density}, however, hold for more general diffusion processes, as they do not rely on Assumption~\ref{Ass:X}.(iii). 
To emphasise the generality of these results, we therefore state explicitly the assumptions under which each result holds.

% \begin{remark}
% In addition, even though not needed for solving the MFG studied in this paper, using the same occupation measures derived in Section \ref{sec:density} for obtaining the stationary distribution of two-sided controlled processes, we also obtain the stationary distributions of the one-sided downward and upward controlled processes in Section~\ref{sec:density1}.
% Note that, these derivations of the stationary distributions in Sections \ref{sec:density} and \ref{sec:density1} are novel results in the literature which carry their own independent value, since they can be used in other such constrained ergodic singular control problems and games.
% \end{remark}

\begin{remark}
In addition, even though not needed for solving the MFG studied in this paper, using the same occupation measures derived in Section \ref{sec:density} for obtaining the stationary distribution of two-sided controlled processes, we also obtain the stationary distributions of the one-sided downward and upward controlled processes in Section~\ref{sec:density1}.
Note that, these derivations of the stationary distributions in Sections \ref{sec:density} and \ref{sec:density1} are novel results in the literature which carry their own independent value, since they can be used in other such constrained ergodic singular control problems and games.
\end{remark}

\section{Stationary distributions for Poissonian singularly controlled processes of threshold type} \label{sec:density}

We firstly recall that the hitting time $\tau_y$ of $X$ at some $y \in \mathcal{I}$, satisfies for any $r > 0$ (see, e.g.~\citep[p.18]{BorodinSalminen}) the formula
\begin{equation*}
\mathbb{E}_x [ e^{-r \tau_y}] =\begin{cases}
\frac{\psi_r(x)}{\psi_r(y)}, & x \leq y, \\
\frac{\varphi_r(x)}{\varphi_r(y)}, & x > y,
\end{cases}
\end{equation*}
where $\psi_r$ and $\varphi_r$ are the fundamental solutions to $(\mathcal{G}-r)u=0$ (see Section~\ref{sec:Diffusions}).
Then, we further define the first Poisson-sampled hitting times $T_y$ and $T^y$ by
\begin{equation} \label{Ty}
T_y := \inf\{T^{(n)} \mid n \in \mathbb{Z}_+,  X_{T^{(n)}} \geq y\} 
\quad \text{and} \quad
T^y = \inf\{T^{(n)} \mid n \in \mathbb{Z}_+ , X_{T^{(n)}} \leq y\},
\end{equation}
which are the first Poisson jump times at which the process $X$ is at or above, and at or below, the threshold $y$, respectively. 
Then by \citep{Alvarez2024} for any $r>0$ and $\lambda >0$, we have 
\begin{equation*}
    \mathbb{E}_x[e^{-r T_y}] =
    \begin{cases}
        \psi_r(x) \frac{(\Phi_{r+\lambda} 1)(y)}{(\Phi_{r+\lambda} \psi_r)(y)}, & x \leq y, \\
        \varphi_r(x) \frac{(\Psi_{r+\lambda} 1)(y)}{(\Psi_{r+\lambda} \varphi_r)(y)}, & x > y, 
    \end{cases}
\end{equation*}
where $\Psi_{r+\lambda}$ and $\Phi_{r+\lambda}$ are defined by \eqref{PsiPhi}.

Using the above results, we can in principle determine the expected hitting times
\begin{align} \label{ulambda}
u(x;y)
:=\mathbb{E}_x[\tau_y] 
, \quad 
u_{\lambda}^{\downarrow}(x;y) 
:= \mathbb{E}_x[T_y]
\quad \text{and} \quad
u_{\lambda}^{\uparrow}(x;y) 
:= \mathbb{E}_x[T^y],
\end{align}
by observing that  
\begin{equation*}
u(x;y) 
= - \lim_{r \to 0+} \frac{d}{dr} \mathbb{E}_x[e^{-r \tau_y}] 
, \; 
u_{\lambda}^{\downarrow}(x;y) 
= - \lim_{r \to 0+} \frac{d}{dr} \mathbb{E}_x[e^{-r T_y}]
, \;
 u_{\lambda}^{\uparrow}(x;y) 
= - \lim_{r \to 0+} \frac{d}{dr} \mathbb{E}_x[e^{-r T^y}].
\end{equation*}
Although the first representation is well-suited for deriving the stationary density of reflected processes, the latter two are not suitable for deriving the stationary density of threshold-type Poissonian singularly controlled processes.

We therefore begin this section by establishing occupation measures for a diffusion that is controlled either downwards or upwards via Poissonian singular controls at a threshold $y$, i.e.~we focus on one-sided downward or upward controls.
For instance, the case of a one-sided downward control corresponds to a controlled process $X^\downarrow$ which jumps downwards to the point $y$ at jump times $T^{(i)}$ of the independent Poisson process $N$, only if $X^\downarrow_{T^{(i)}} \in (y, \mathfrak{r})$, otherwise $X^\downarrow$ is not controlled.
Namely, we choose the {\it one-sided downward Poissonian singular control process of threshold type} $\nu^\downarrow \in \mathcal{V}$ from Definition~\ref{def:adm}, which is defined by 
\begin{equation} \label{nudown}
\nu_t^\downarrow 
= \int_0^{t} (X^{\downarrow}_{s-}-y)^+ dN_s 
= \sum_{k=1}^{N(t)} (X^{\downarrow}_{T^{(k)}-}-y)^+, 
\quad t \geq 0, 
\quad \text{with} \quad 
\nu^\downarrow_0 = 0,
\end{equation}
where the corresponding downward Poissonian singularly controlled process $X^\downarrow$ at threshold $y$ from \eqref{Xnu} satisfies  
\begin{equation} \label{Xdown}
X^\downarrow_t = x+ \int_0^t b(X^\downarrow_s) ds+ \int_0^t\sigma(X^\downarrow_s) dW_s - \int_0^{t} (X^\downarrow_{s-}-y)^+ dN_s.
\end{equation}
Equivalently, we can also choose the {\it one-sided upward Poissonian singular control process of threshold type} $\nu^\uparrow \in \mathcal{V}$ from Definition~\ref{def:adm}, which is defined by 
\begin{equation} \label{nuup}
\nu_t^\uparrow 
= \int_0^{t} (y-X^{\uparrow}_{s-})^+ dN_s 
= \sum_{k=1}^{N(t)} (y-X^{\uparrow}_{T^{(k)}-})^+, 
\quad t \geq 0, 
\quad \text{with} \quad 
\nu^\uparrow_0 = 0,
\end{equation}
where the corresponding upward Poissonian singularly controlled process $X^\uparrow$ at threshold $y$ from \eqref{Xnu} satisfies  
\begin{equation} \label{Xup}
X^\uparrow_t = x+ \int_0^t b(X^\uparrow_s) ds+ \int_0^t\sigma(X^\uparrow_s) dW_s + \int_0^{t} (y-X^\uparrow_{s-})^+ dN_s.
\end{equation}

For any Borel set $A \subseteq \mathcal{I}$, we then define the occupation measure
\begin{equation} \label{OTu}
u(x;y,A) := \mathbb{E}_x \Big[ \int_0^{\tau_y} \mathbf{1}(X_t \in A) \, dt \Big] =
\begin{cases}
\mathbb{E}_x \big[ \int_0^{\tau_y} \mathbf{1}(X^\downarrow_t \in A) \, dt \big], \quad x \leq y, \\
\mathbb{E}_x \big[ \int_0^{\tau_y} \mathbf{1}(X^\uparrow_t \in A) \, dt \big], \quad x \geq y, 
\end{cases}
\end{equation}
such that $u(x;y,\mathcal{I}) = u(x;y)$ from \eqref{ulambda}, and the occupation measures
\begin{equation} \label{OT}
\begin{split}
V^{\downarrow}(x;y,A) 
&:= \mathbb{E}_x \Big[ \int_0^{T_y} \mathbf{1}(X^\downarrow_t \in A) \, dt \Big] 
= \mathbb{E}_x \Big[ \int_0^{T_y} \mathbf{1}(X_t \in A) \, dt \Big], \\
%\quad \text{and} \quad
V^{\uparrow}(x;y,A) 
&:= \mathbb{E}_x \Big[ \int_0^{T^y} \mathbf{1}(X^\uparrow_t \in A) \, dt \Big]
= \mathbb{E}_x \Big[ \int_0^{T^y} \mathbf{1}(X_t \in A) \, dt \Big],
\end{split}
\end{equation}
such that $V^{\downarrow}(x;y,\mathcal{I}) = u_{\lambda}^{\downarrow}(x;y)$ and $V^{\uparrow}(x;y,\mathcal{I}) = u_{\lambda}^{\uparrow}(x;y)$ from \eqref{ulambda}.

\begin{theorem} \label{Theorem: PoissonOccupation}
Suppose that Assumptions~\ref{Ass:X}.(i),(ii),(iv),(v) hold true, and consider $\lambda > 0$ and $y \in \mathcal{I}$. 
Recall the one-sided downward and upward Poissonian singularly controlled processes $X^\downarrow$ and $X^\uparrow$ at the threshold $y$ satisfying \eqref{Xdown} and \eqref{Xup}, respectively. 
Then, for $x \leq y$,
\begin{equation*} %\label{Eq:V_down_theorem}
V^{\downarrow}(x; y, dz) 
= \frac{m(dz)}{\lambda(\Phi_\lambda 1)(y) + \frac{\varphi_{\lambda}(y)}{S(y)-S(\mathfrak{l}+)}} \cdot
\begin{cases}
h(x \wedge z;y) \big( \varphi_\lambda(y) + \lambda(\Phi_\lambda 1)(y) \big( S(y) - S(x \vee z) \big) \big),&z < y, \\
h(x;y) \, \varphi_\lambda(z), & z \geq y,
\end{cases}
\end{equation*}
and for $x \geq y$
\begin{equation*} %\label{Eq:V_up_theorem}
V^{\uparrow}(x; y, dz) 
= \frac{m(dz)}{\lambda(\Psi_\lambda 1)(y)+\frac{\psi_\lambda(y)}{S(\mathfrak{r}-)-S(y)}} \cdot
\begin{cases}
h(x;y) \, \psi_\lambda(z), & z \leq y, \\\
h(x \vee z;y) \big( \psi_\lambda(y) + \lambda(\Psi_\lambda 1)(y) \big( S(x \wedge z) - S(y) \big) \big), & z > y,
\end{cases}
\end{equation*}
% \begin{equation} \label{Eq:V_down_theorem} 
% V^{\downarrow}(y; y, dx) = \frac{m(dx)}{\lambda(\Phi_\lambda 1)(y) + \frac{\varphi_{\lambda}(y)}{S(y)-S(\mathfrak{l}+)}} \cdot
% \begin{cases}
% \varphi_\lambda(y) h(x;y), & x < y, \\
% \varphi_\lambda(x), & x \geq y.
% \end{cases}
% \end{equation}
% \begin{equation} \label{Eq:V_up_theorem}
% V^{\uparrow}_{\lambda}(y; y, dx) = \frac{m(dx)}{\lambda(\Psi_\lambda 1)(y)+\frac{\psi_\lambda(y)}{S(\mathfrak{r}-)-S(y)}} \cdot
% \begin{cases}
% \psi_\lambda(x), & x \leq y, \\
% \psi_\lambda(y) h(x;y), & x \geq y,
% \end{cases}
% \end{equation}
where $m, S, h, \varphi_\lambda , \Phi_\lambda , \psi_\lambda , \Psi_\lambda $ are defined in Section~\ref{sec:Diffusions}.
\end{theorem}

\begin{proof}
We only prove the expression of $V^{\downarrow}$, since the expression of $V^{\uparrow}$ can be obtained using completely analogous arguments.

By the strong Markov property and the memoryless property of the exponential distribution, we obtain 
\begin{equation} \label{Eq:V_up_markov}
V^{\downarrow}(x; y, A) 
= u(x; y, A) + h(x;y) \, V^{\downarrow}(y; y, A), \quad x \leq y,
\end{equation}
where $u(x; y, A)$ is defined in \eqref{OTu} and $h(x,y)$ in \eqref{Eq:hitprob}.

% \smallskip
% {\it Density of occupation measure $V^{\downarrow}(y; y, A)$.}
We can now condition on the first jump time $U \sim \Exp(\lambda)$ of the Poisson process to get 
\begin{equation} \label{Eq:V_up}
\begin{aligned}
V^{\downarrow}(y; y, A) 
& = \mathbb{E}_y \Big[ \int_0^{U} \mathbf{1}(X^\downarrow_t \in A) \, \mathbf{1}(X^{\downarrow}_U  \geq y) \, dt \Big] + \mathbb{E}_y \Big[ \int_0^{T_y} \mathbf{1}(X^\downarrow_t \in A) \, \mathbf{1}(X^\downarrow_U  < y) \, dt \Big] \\
%& = \mathbb{E}_y \Big[ \int_0^{U} \mathbf{1}(X^\downarrow_t \in A) \, dt \Big] + \mathbb{E}_y \Big[ \int_U^{T_y} \mathbf{1}(X^\downarrow_t \in A) \, \mathbf{1}(X^\downarrow_U  < y) \, dt \Big] \\
& = \mathbb{E}_y \Big[ \int_0^{U} \mathbf{1}(X^\downarrow_t \in A) \, dt \Big] + \mathbb{E}_y \Big[ V^{\downarrow}(X^\downarrow_U; y, A) \, \mathbf{1}(X^\downarrow_U  < y) \Big], 
\end{aligned}
\end{equation}
where the first equality follows from the fact that $T_y = U$ and $T_y > U$ in the first and second expectations, respectively, and the second equality from the strong Markov property at $U$ and the memoryless property of the exponential distribution. 
For the first expectation in the latter expression, we find
\begin{equation*}
\begin{aligned}
& \mathbb{E}_y \Big[ \int_0^{U} \mathbf{1}(X^\downarrow_t \in A) \, dt \Big] 
= \mathbb{E}_y \Big[ \int_0^{\infty} \mathbf{1}(X^\downarrow_t \in A)\, \mathbf{1}(U > t) \, dt \Big] \\
&= \int_0^{\infty} e^{-\lambda t} \,  \mathbb{P}_y(X^\downarrow_t \in A) \, dt 
= \int_A G_{\lambda}(y,x) \, m(dx),
\end{aligned}
\end{equation*}
whereas for the second one we find 
\begin{equation*}
\begin{aligned}
\mathbb{E}_y \Big[ V^{\downarrow}(X^\downarrow_U; y, A) \, \mathbf{1}(X^\downarrow_U  < y) \Big] 
= \mathbb{E}_y \Big[ \lambda \int_0^{\infty} e^{-\lambda t} V^{\downarrow}(X^\downarrow_t; y, A) \, \mathbf{1}(X^\downarrow_t < y) \, dt \Big] 
= \lambda (R_{\lambda} f_y)(y),
\end{aligned}
\end{equation*}
where $f_y(z) := V^{\downarrow}(z; y, A) \, \mathbf{1}(z < y)$. 
Then, using the resolvent representation in \eqref{Eq:Resolvent1} with $x=y$ together with \eqref{Eq:V_up_markov}, we further obtain that
\begin{equation*}
\begin{aligned}
\lambda (R_{\lambda} f_y)(y) 
&= \lambda B_{\lambda}^{-1} \varphi_\lambda(y) \int_{\mathfrak{l}}^y u(z; y, A) \, \psi_\lambda(z) \, m(dz) \\
&\quad + \lambda  B_{\lambda}^{-1} \varphi_\lambda(y) \, V^{\downarrow}(y; y, A) \int_{\mathfrak{l}}^y h(z;y) \, \psi_\lambda(z) \, m(dz).
\end{aligned}
\end{equation*}
Therefore, by substituting \eqref{PsiPhi} and all above expressions into \eqref{Eq:V_up}, we %solve for $V^{\downarrow}(y; y, A)$ to 
get 
\begin{equation*}
V^{\downarrow}(y; y, A) 
= \frac{ \int_A G_{\lambda}(y,x) \, m(dx) + \lambda B_{\lambda}^{-1} \varphi_\lambda(y) \int_{\mathfrak{l}}^y u(z; y, A) \, \psi_\lambda(z) \, m(dz)}{1-\lambda B_{\lambda}^{-1} \varphi_\lambda(y) (\Psi_\lambda  h(\cdot;y))(y)}. 
\end{equation*}
Using the Green kernel representation $u(z; y, A) = \int_A G_0^y(z,x) m(dx)$, we eventually conclude that $V^{\downarrow}(y; y, A)$ admits the density
\begin{equation} \label{Eq:V_up-2}
V^{\downarrow}(y; y, dx) 
= \frac{ G_{\lambda}(y,x) + \lambda B_{\lambda}^{-1} \varphi_\lambda(y)\int_{\mathfrak{l}}^y G_0^y(z,x) \, \psi_\lambda(z) \, m(dz)}{1-\lambda B_{\lambda}^{-1} \varphi_\lambda(y) \,(\Psi_\lambda  h(\cdot;y))(y)} \, m(dx) 
\end{equation}
Now, using \eqref{Eq:hitprob} and \eqref{Eq:GreenKernelKilled} we see that for all $z \in (\mathfrak{l},y)$
\begin{equation} \label{Eq:G_0yzx}
G_0^y(z,x) =
    \begin{cases}
        h(z \wedge x;y)(S(y) - S(z\vee x)), & x < y, \\
        0, & x \geq y.
    \end{cases}
\end{equation}
% Now, using \eqref{Eq:hitprob} and \eqref{Eq:GreenKernelKilled} 
% for $x < y$ we see that $G_0^y(z,x) = h(z \wedge x;y)(S(y) - S(z\vee x))$, whereas for $x \geq  y$ we see that $G_0^y(z,x) = 0$, for all $z \in (\mathfrak{l},y)$. 
Hence, for $x<y$, this implies that by splitting the integral at $x$, we can get 
\begin{align} \label{infG0ly}
\int_{\mathfrak{l}}^{y} G_0^y(z,x) \, \psi_\lambda(z) \, m(dz)
&= (S(y)-S(x)) \,(\Psi_\lambda h(\cdot;y))(x) + h(x;y)\int_x^y(S(y)-S(z)) \, d(\Psi_\lambda1)(z) \notag\\
&= (S(y)-S(x)) \Big( h(x;y) \, (\Psi_\lambda 1)(x) -  \frac{\psi_{\lambda}(x)}{\lambda(S(y)-S(\mathfrak{l}+))} \Big) \notag\\ 
& \quad + h(x;y) \Big( \frac{\psi_{\lambda}(y)-\psi_{\lambda}(x)}{\lambda}-(S(y)-S(x)) (\Psi_{\lambda}1)(x) \Big) \notag\\ 
&= \frac{1}{\lambda} \Big( h(x;y)\psi_{\lambda}(y)-\psi_{\lambda}(x) \Big),
\quad x<y,
\end{align}
where we used Stieltjes integration-by-parts formula in the second equality together with the expressions \eqref{Eq:GeneratorCanonical}, \eqref{Eq:hitprob} and the identity $\frac{d \psi_\lambda}{dS}(x) = \lambda (\Psi_{\lambda}1)(x)$, which yields in particular that  
\begin{align} \label{Eq:h_integralx}
\begin{split}
(\Psi_\lambda h(\cdot;y))(x) 
&= \int_{\mathfrak{l}}^x h(z;y) \, \psi_\lambda(z) \, m(dz) 
= \frac1\lambda \int_{\mathfrak{l}}^x h(z;y) \, \mathcal{G} \psi_\lambda(z) \, m(dz) \\
&= %\frac1\lambda \Big( h(x;y) \, \frac{d \psi_\lambda}{dS}(x) - \frac{\psi_\lambda(x)}{S(y)-S(\mathfrak{l+})} \Big) %\notag\\& = 
h(x;y) \, (\Psi_{\lambda}1)(x) - \frac{\psi_\lambda(x)}{\lambda (S(y)-S(\mathfrak{l+}))}
, \quad x \leq y.
\end{split}
\end{align}
%where the second equality follows from \eqref{Eq:h_integral} at $y=x$. 
Combining the expression \eqref{infG0ly} with the fact that $\int_{\mathfrak{l}}^{y}G_0^y(z,x)\psi_\lambda(z)m(dz) = 0$ for $x \geq y$, thus yields together with the Green-kernel $G_\lambda$ from \eqref{Eq:GreenKernel1}, that the numerator of \eqref{Eq:V_up-2} equals
\begin{equation*}
\begin{cases}
B_{\lambda}^{-1} \varphi_{\lambda}(y) \, \psi_{\lambda}(y) \, h(x;y), & x < y, \\
B_{\lambda}^{-1} \psi_{\lambda}(y) \, \varphi_{\lambda}(x), & x \geq y.
\end{cases}    
\end{equation*}
By %\eqref{Eq:h_integral} 
\eqref{Eq:h_integralx} and \eqref{Eq:hitprob} for $x=y$, and \eqref{Eq:GreenKernel1} together with the identity $\frac{d \varphi_\lambda}{dS}(x) = - \lambda (\Phi_{\lambda}1)(x)$, the denominator of \eqref{Eq:V_up-2} equals
\begin{equation*} %\label{Eq:denominator}
% 1-\lambda B_{\lambda}^{-1} \varphi_\lambda(y) \, (\Psi_\lambda  h(\cdot;y))(y) = 
1-\lambda B_{\lambda}^{-1} \varphi_\lambda(y) \, \Big((\Psi_{\lambda}1)(y) - \frac{\psi_\lambda(y)}{\lambda (S(y)-S(\mathfrak{l+}))} \Big) 
= B_{\lambda}^{-1} \psi_\lambda(y) \Big( \lambda (\Phi_{\lambda}1)(y) -  \frac{\varphi_\lambda(y)}{(S(y)-S(\mathfrak{l+}))} \Big) 
\end{equation*}
Hence the fraction \eqref{Eq:V_up-2} takes the form
\begin{equation} \label{Eq:V_down_theorem_xy} 
V^{\downarrow}(y; y, dx) = \frac{m(dx)}{\lambda(\Phi_\lambda 1)(y) + \frac{\varphi_{\lambda}(y)}{S(y)-S(\mathfrak{l}+)}} \cdot
\begin{cases}
\varphi_\lambda(y) h(x;y), & x < y, \\
\varphi_\lambda(x), & x \geq y.
\end{cases}
\end{equation}

Finally, substituting into \eqref{Eq:V_up_markov} the Green kernel representation $u(x;y,dz) = G_0^y(x,z) \, m(dz)$, where $G_0^y$ is derived in \eqref{Eq:G_0yzx}, together with the expression \eqref{Eq:V_down_theorem_xy} of $V^{\downarrow}(y;y,dz)$, gives 
% $$
% \frac{V^{\downarrow}(x;y,dz)}{m(dz)} 
% = \begin{cases}
% \frac{h(x;y) \, \varphi_\lambda(z)}{\lambda(\Phi_\lambda 1)(y)+\frac{\varphi_\lambda(y)}{S(y)-S(\mathfrak{l}+)}}, & z \geq y, \\
% {\scriptstyle h(x\wedge z;y) (S(y)-S(x\vee z))} + \frac{h(x;y) \, h(z;y) \, \varphi_\lambda(y)}{\lambda(\Phi_\lambda 1)(y)+\frac{\varphi_\lambda(y)}{S(y)-S(\mathfrak{l}+)}}, &z<y. 
% \end{cases} 
% $$
\begin{align}
V^{\downarrow}(x;y,dz)
= \bigg( &\mathbf{1}_{\{z\geq y\}} \frac{h(x;y)\,\varphi_\lambda(z)}{\lambda(\Phi_\lambda 1)(y) + \frac{\varphi_\lambda(y)}{S(y)-S(\mathfrak{l}+)}} \notag\\
&+ \mathbf{1}_{\{z<y\}} \bigg( h(x\wedge z;y)\bigl(S(y)-S(x\vee z)\bigr) + \frac{h(x;y)\,h(z;y)\,\varphi_\lambda(y)}{\lambda (\Phi_\lambda 1)(y) + \frac{\varphi_\lambda(y)}{S(y)-S(\mathfrak{l}+)}} \bigg) \bigg) \, m(dz).
\end{align}
Then, using $h(x;y) \, h(z;y) = h(x \wedge z ; y) \, h(x \vee z; y)$, and the expression $h(x \vee z; y) = 1 - \frac{S(y)-S(x \vee z)}{S(y)-S(\mathfrak{l+})}$ from \eqref{Eq:hitprob} for $x \leq y$, the part related to $\mathbf{1}_{\{z<y\}}$ on the right-hand side becomes
\begin{align*}
&h(x\wedge z;y) \bigg( \bigl(S(y)-S(x\vee z)\bigr) + \frac{\varphi_\lambda(y) - \frac{S(y)-S(x \vee z)}{S(y)-S(\mathfrak{l+})} \,\varphi_\lambda(y)}{\lambda (\Phi_\lambda 1)(y) + \frac{\varphi_\lambda(y)}{S(y)-S(\mathfrak{l}+)}} \bigg) \\
&= h(x\wedge z;y) \bigg( \frac{\varphi_\lambda(y)}{\lambda (\Phi_\lambda 1)(y) + \frac{\varphi_\lambda(y)}{S(y)-S(\mathfrak{l}+)}} + \bigl(S(y)-S(x\vee z)\bigr) \bigg( 1 -  \frac{\frac{\varphi_\lambda(y)}{S(y)-S(\mathfrak{l+})}}{\lambda (\Phi_\lambda 1)(y) + \frac{\varphi_\lambda(y)}{S(y)-S(\mathfrak{l}+)}} \bigg)\bigg) \\
&= \frac{h(x\wedge z;y) \,  \big( \varphi_\lambda(y) + \lambda (\Phi_\lambda 1)(y) \bigl(S(y)-S(x\vee z)\bigr) \big)}{\lambda (\Phi_\lambda 1)(y) + \frac{\varphi_\lambda(y)}{S(y)-S(\mathfrak{l}+)}},
\end{align*}
which completes the proof.
%%%%%%%%%%%%%%%%%%%%%%%%
% [This was not used above / I actually derived a more general formula for $x \leq y$, not only the $x=y$ case:]
% \bl{Since $\mathcal{G}h(z;y) = 0$ for $z \in (\mathfrak{l},y)$, $h(y;y)=1$, and the boundary $\mathfrak{l}$ is natural, integration by parts gives
% \begin{equation} \label{Eq:h_integral}
% (\Psi_\lambda  h)(y) = \int_{\mathfrak{l}}^y h(z;y)\psi_\lambda(z) m(dz) =  (\Phi_{\lambda}1)(y)-\frac{\psi_{\lambda}(y)}{\lambda(S(y)-S(\mathfrak{l+}))}.
% \end{equation}
% }
\end{proof}

Building on the identities for the occupation measures of the one-sided downward and upward Poissonian singularly controlled processes $X^\downarrow$ and $X^\uparrow$ at constant thresholds, which are established in Theorem~\ref{Theorem: PoissonOccupation}, we now proceed to obtaining the stationary distribution of the two-sided Poissonian singularly controlled process $X^\nu$ at the thresholds $a$ and $y$, satisfying $\mathfrak{l}<a<y<\mathfrak{r}$.
That is, if the process $X^\nu_{T^{(i)}}$ is observed at a Poisson jump time $T^{(i)}$ below $a$ (resp., above $y$), then it is immediately dragged up to $a$ (resp., down to $y$) and started anew; otherwise $X^\nu$ is not controlled.
Namely, we choose the {\it two-sided Poissonian singular control process of threshold type} (cf.~Definition~\ref{def:adm}) 
\begin{equation} \label{nuay}
\nu(a,y) := \nu^+(a) - \nu^-(y) \in \mathcal{V}, 
\end{equation} 
where $\nu^+(a)$ takes the form of $\nu^\uparrow$ defined by \eqref{nuup} for threshold $a$ and $\nu^-(y)$ the form of $\nu^\downarrow$ defined by \eqref{nudown} for threshold $y$. 
Then, the corresponding Poissonian singularly controlled process $X^{\nu(a,y)}$ at thresholds $(a,y)$ satisfies \eqref{Xnu} with %$\nu = \nu^+ - \nu^-$ defined by
\begin{equation} \label{nudownup}
\begin{split}
\nu_t^+ = \nu_t^+(a)
&:= \int_0^{t} (a -X^{\nu(a,y)}_{s-})^+ dN_s 
= \sum_{k=1}^{N(t)} (a -X^{\nu(a,y)}_{T^{(k)}-})^+, 
\quad t \geq 0, 
\quad \text{with} \quad 
\nu^+_0(a) = 0, \\
\nu_t^- = \nu_t^-(y)
&:= \int_0^{t} (X^{\nu(a,y)}_{s-} - y)^+ dN_s 
= \sum_{k=1}^{N(t)} (X^{\nu(a,y)}_{T^{(k)}-} - y)^+, 
\quad t \geq 0, 
\quad \text{with} \quad 
\nu^-_0(y) = 0.
\end{split}
\end{equation}

\begin{theorem} \label{Theorem:TwoSidedDensity}
Suppose that Assumptions~\ref{Ass:X}.(i),(ii),(iv),(v) hold true, and consider $\lambda > 0$ and $a, y \in \mathcal{I}$ satisfying $\mathfrak{l}<a<y<\mathfrak{r}$.
Recall the two-sided Poissonian singularly controlled processes $X^{\nu(a,y)}$ at the thresholds $(a,y)$ satisfying \eqref{Xnu} with \eqref{nudownup}. 
Then the stationary distribution of $X^{\nu(a,y)}$ is given by 
\begin{equation*}
\mu_{\lambda}^{(a,y)}(dx) = \frac{m(dx) }{ m((a,y)) + \dfrac{(\Psi_\lambda  1)(a)}{\psi_\lambda(a)}+ \dfrac{(\Phi_\lambda  1)(y)}{\varphi_\lambda(y)}} \cdot
\begin{cases}
\dfrac{\psi_\lambda(x)}{\psi_\lambda(a)}, & x \leq a, \\
1, & a < x < y, \\
\dfrac{\varphi_\lambda(x)}{\varphi_\lambda(y)}, & x \geq y.
\end{cases}
\end{equation*}
\end{theorem}

\begin{proof}
Suppose that the path of the process $X^{\nu}$ is formed by cycles $C_0, C_1, C_2, C_3, \ldots$. 
Each cycle $C_{i}$ starts at $a$ (apart from the initial cycle $C_0$ that starts at $x$ and ends at $a$), runs until the first visit to $y$, and end at the first subsequent visit to $a$. 
Because $a<y$, the interventions \eqref{nudownup} only displace the process towards $[a,y]$: an upward intervention lands at $a$ and a downward one at $y$. 
Hence neither level can be crossed by a jump and both $y$ and $a$ are therefore reached continuously, with no overshoot. 

\smallskip
{\it Large cycles $C_1, C_2, \ldots$.}
We then say that the first part of each cycle $C_i$, $i=1,2,\ldots$, ends when $X^{\nu(a,y)}$ first reaches $y$, and that the second part ends when $X^{\nu(a,y)}$ subsequently reaches $a$. 
Since the cycles begin and end at $a$, the strong Markov property of $X^{\nu}$, together with the memoryless property of the exponential inter-arrival times $\tau_{i}$, implies that $C_{1}, C_{2}, \ldots$ are i.i.d.. 
Moreover $X^{\nu(a,y)}$ remains in $(\mathfrak{l},y]$ during the first part of a cycle and in $[a,\mathfrak{r})$ during the second part. 
Thus, only upward interventions occur during the first part and only downward interventions during the second.

\smallskip
{\it First part of each cycle $C_i$.}
During the first part of each cycle, the process evolves as the uncontrolled diffusion started at $a$ until either it hits $y$ or a Poisson arrival occurs at a time when the process is at or below $a$, in which case it is reset to $a$ and the procedure starts afresh. 
Thus, the first part consists of repeated trials of the uncontrolled diffusion from $a$, each stopped at $\tau_{y} \wedge T^{a}$: 
a trial succeeds at $\tau_{y}$, whereas at $T^{a}$ the process is reset to $a$ and a new trial begins. 
Its expected occupation measure is therefore the expected occupation of a single trial divided by $\mathbb{P}_{a}(\tau_{y}<T^{a})$. 
Equivalently, the first part of each cycle can be viewed as the upward Poissonian singularly controlled process $X^\uparrow$ at threshold $a$ satisfying \eqref{Xup} when started at $a$ and run until it reaches $y$.

\smallskip
{\it Small blocks $D_1,D_2,\ldots$ in the first part of each cycle $C_i$.}
Cut the path of this $X^{\uparrow}$ each time it is controlled upwards to $a$, on the event that $y$ has not been reached (so that we are still in the first part of a cycle). 
These cuts give i.i.d blocks $D_1, D_2, \ldots$, between successive upward interventions. 
Within each $D_j$, the process evolves as the uncontrolled diffusion started at $a$, and hence it reaches $y$ before the next upward intervention precisely when $\tau_y < T^a$, which occurs with probability 
$$
p:=\mathbb{P}_a(\tau_y < T^a).
$$
Thus, if $M$ denotes the index of the first successful trial, the first part consists of $D_1, D_2, \ldots, D_{M-1}$ in full, followed by the block $D_M$ truncated at $\tau_y$. 
To formalise this, we define $\theta_0 := 0$ and $\theta_k$ as the $k$-th intervention time of $X^{\uparrow}$, so that for a Borel set $A \subseteq \mathcal{I}$, the occupation time $O_k(A)$ of $A$ between two successive interventions and the index $S_k$ for the $k$-th block hitting $y$ are defined by 
\begin{equation*}
O_k(A) := \int_{\theta_{k-1}}^{\theta_k} \mathbf{1}(X_t^{\uparrow} \in A) \, dt 
\quad \text{and} \quad 
S_k := \mathbf{1}\Big(\sup_{t \in [\theta_{k-1}, \theta_{k})}  X^{\uparrow}_t \geq y \Big).
\end{equation*}
Since $\theta_1 = T^a$, we have by definition $\mathbb{E}_a[O_1(A)] = V^{\uparrow}(a;a,A)$, which is defined by \eqref{OT}. 

\smallskip
{\it Occupation measure $V^{(1)}(A)$ of the first part of each cycle $C_i$.}
By further defining 
\[
M := \inf \{k \geq 1 \,:\, S_k =1 \} 
\quad \text{and} \quad 
R_M(A) := \int_{\tau_y}^{\theta_M} \mathbf{1}(X_t^{\uparrow} \in A) \, dt ,
\]
we see that the the occupation measure $V^{(1)}(A)$ during the first part of a cycle is
\begin{equation*}
V^{(1)}(A) := \mathbb{E}_a \Big[\sum_{k \leq M} O_k(A) - R_M(A)\Big].
\end{equation*}
Since $X^{\uparrow}$ is reset to $a$ at every intervention time and has the strong Markov property, while the Poisson intervention mechanism is memoryless, the successive blocks $(O_k, S_k)_{k \geq 1}$ are i.i.d copies of $(O_1, S_1)$.
In particular $S_k$ are i.i.d Bernoulli random variables with 
$\mathbb{P}(S_k = 1) =  \mathbb{P}_a(\tau_y < T^a) = p$, and $M$ is thus geometric with parameter $p$ and $\mathbb{E}[M] = 1/p$.
We therefore get that
\begin{align} \label{Eq:OccupationTS}
V^{(1)}(A) 
% &= \mathbb{E} \bigg[ \sum_{k\leq M} O_k(A) - R_M(A) \bigg] \\
&= \sum_{k \leq M} \mathbb{P}(M\geq k) \, \mathbb{E}[O_1(A)] - \mathbb{E}[R_M(A)] 
= \frac{1}{p} \, V^{\uparrow}(a;a,A) - \frac{1}{p} \, \mathbb{E} \big[ \mathbf{1}(S_1=1) \, R_1(A) \big] \notag\\
&= \frac{1}{p} \, V^{\uparrow}(a;a,A) - V^{\uparrow}(y;a,A) 
= \Big( \frac{1}{p} - h(y; a) \Big) \, V^{\uparrow}(a;a,A) - u(y;a,A),
\end{align}
where the first and second equalities follow from the independence of $\{M\geq k\}$ from $O_k(A)$ and the fact that the blocks are i.i.d., 
the third equality from the fact that on $\{S_1=1\}$, the process reaches $y$ before the next intervention, so that by the strong Markov property at $\tau_y$ and the memoryless property of the Poisson intervention mechanism, we get that 
$\mathbb{E} \big[ \mathbf{1}(S_1=1) \, R_1(A) \big] = p \, V^{\uparrow}(y;a,A)$ from \eqref{OT}, 
and finally the fourth equality by $V^{\uparrow}(y;a,A) =  u(y;a,A) + h(y;a) \, V^{\uparrow}(a;a,A)$, where $u(y;a,A)$ is defined by \eqref{OTu}.  

Since the occupation measure $V^{(1)}(A_0)$ of any Borel set $A_0 \subset (y, \mathfrak{r})$ is zero, the right-hand side of equation \eqref{Eq:OccupationTS} must vanish on $A_0$. Combining this with Theorem~\ref{Theorem: PoissonOccupation} and \eqref{Eq:hitprob}, which imply 
\[
V^{\uparrow}(a; a, A_0) 
= \frac{\psi_\lambda(a)}{\lambda(\Psi_\lambda 1)(a)+\frac{\psi_\lambda(a)}{S(\mathfrak{r}-)-S(a)}} \, \int_{A_0} h(z;a) % \frac{S(\mathfrak{r}-)-S(z)}{S(\mathfrak{r}-)-S(a)} 
\, m(dz) , 
\]
and with %Corollary~\ref{Corr:u_lambda}, 
the Green kernel representation and \eqref{Eq:GreenKernelKilled} for $z > y > a$, which imply 
$$
u(y;a,A_0) 
= \int_{A_0} G_0^a(y,z) \, m(dz) 
= (S(y)-S(a))\int_{A_0} h(z;a) \, m(dz), 
$$ 
we eventually get from \eqref{Eq:OccupationTS} that
\begin{align*}
0 
&= \Big( \frac{1}{p} - h(y; a) \Big) \, V^{\uparrow}(a;a,A_0) - u(y;a,A_0) \\
&= \bigg( \Big( \frac{1}{p} - h(y; a) \Big) 
\frac{\psi_\lambda(a)}{\lambda (\Psi_\lambda 1)(a) + \frac{\psi_\lambda(a)}{S(\mathfrak{r}-)-S(a)}} 
- (S(y)-S(a)) \bigg) \int_{A_0} h(x;a) \, m(dx) ,
\end{align*}
which after straightforward calculations, using also \eqref{Eq:hitprob}, yields 
\begin{equation*}
p = \frac{\psi_{\lambda}(a)}{\psi_{\lambda}(a)+\lambda(\Psi_{\lambda}1)(a)(S(y)-S(a))}.
\end{equation*}

Hence, by substituting the value of $p$ back to the equation \eqref{Eq:OccupationTS} implies that the occupation measure has density
\begin{align*}
V^{(1)}(dx) 
&= \frac{\psi_{\lambda}(a) + \lambda(\Psi_{\lambda}1)(a)(S(y)-S(a)) - \psi_{\lambda}(a) \, h(y;a)}{\psi_{\lambda}(a)} \, V^{\uparrow}(a;a,dx) - u(y;a,dx) \\
&= \frac{(S(y)-S(a)) \, \big( \lambda(\Psi_{\lambda}1)(a) +  \frac{\psi_{\lambda}(a)}{S(\mathfrak{r}-)-S(a)} \big)}{\psi_{\lambda}(a)} \, V^{\uparrow}(a;a,dx) - u(y;a,dx),
\end{align*}
where the latter equality follows from \eqref{Eq:hitprob}, namely 
\[
1-h(y;a) 
= 1-\frac{S(\mathfrak{r}-)-S(y)}{S(\mathfrak{r}-)-S(a)}
= \frac{S(y)-S(a)}{S(\mathfrak{r}-)-S(a)}.
\]
In light of Theorem~\ref{Theorem: PoissonOccupation} which gives 
\[
V^{\uparrow}(a;a,dx) = \frac{m(dx)}{\lambda(\Psi_\lambda 1)(a)+\frac{\psi_\lambda(a)}{S(\mathfrak{r}-)-S(a)}} \cdot
\begin{cases}
\psi_\lambda(x), & x \leq a, \\
h(x;a) \, \psi_\lambda(a), & x > a, 
\end{cases}
\]
and the Green kernel representation 
$u(y;a,dx) = G_0^a(y,x) \, m(dx)$ with \eqref{Eq:GreenKernelKilled} which gives 
\begin{align*}
u(y;a,dx) 
% &= \frac{m(dx)}{S(\mathfrak{r-})-S(a)} \cdot \begin{cases}
% 0, & x \leq a, \\
% (S(y \wedge x)-S(a))(S(\mathfrak{r}-)-S(y \vee x)) , & x>a, 
% \end{cases} \\
&= m(dx) \cdot
\begin{cases}
0, & x\le a,\\
h(y;a) \, \bigl(S(x)-S(a)\bigr), & a<x<y,\\
h(x;a) \, \bigl(S(y)-S(a)\bigr), & x \geq y.
\end{cases}
\end{align*}
we conclude that
\begin{equation} \label{Vdx}
V^{(1)}(dx) = \begin{cases}
(S(y)-S(a)) \, \frac{\psi_{\lambda}(x)}{\psi_{\lambda}(a)} \, m(dx), & x \leq a, \\
(S(y)-S(x)) \, m(dx), & a < x < y, \\
0, & x \geq y.
\end{cases}
\end{equation}

\smallskip
{\it Occupation measure $V^{(2)}(A)$ of the second part of each cycle $C_i$.}
The second part of each cycle $C_i$ can be equivalently viewed as the downward Poissonian singularly controlled process $X^\downarrow$ at threshold $y$ satisfying \eqref{Xdown} when started at $y$ and run until it reaches $a$.
By cutting its path each time it is controlled downwards to $y$, on the event that $a$ has not been reached, gives i.i.d.\ blocks, in which the process is the uncontrolled diffusion started at $y$, and hence it reaches $a$ before the next downward intervention $T_y$ with probability 
$$
q:= \mathbb{P}_{y}(\tau_{a}<T_{y}).
$$
Thus, the occupation time of the second part $V^{(2)}(A)$ can be obtained using similar arguments as the ones leading to \eqref{Eq:OccupationTS} for $V^{(1)}(A)$, which yield 
\begin{equation}\label{Eq:OccupationTS2}
V^{(2)}(A) = \left(\frac{1}{q}-h(a;y) \right)V^{\downarrow}(y;y,A)-u_{y}(a;y,A).
\end{equation}
Given that the occupation measure $V^{(2)}(A_0)$ of any Borel set $A_0 \subset (\mathfrak{l},a)$ is zero, we use similar arguments as above, to obtain
\begin{equation*}
q = \frac{\varphi_{\lambda}(y)}{\varphi_{\lambda}(y)+\lambda(\Phi_{\lambda}1)(y)(S(y)-S(a))}.
\end{equation*}
Substituting this back into \eqref{Eq:OccupationTS2} and similar arguments as above gives
\begin{equation*}
V^{(2)}(dx)=
\begin{cases}
0, & x\leq a,\\
(S(x)-S(a))\,m(dx), & a<x<y,\\
(S(y)-S(a))\frac{\varphi_{\lambda}(x)}{\varphi_{\lambda}(y)}\,m(dx), & x\geq y.
\end{cases}
\end{equation*}

\smallskip
{\it Occupation measure $V$ of each cycle $C_i$.}
The occupation measure of each cycle $C_i$ is therefore given by $V = V^{(1)} + V^{(2)}$, so that for every Borel set $A \subseteq \mathcal{I}$, we have 
\begin{equation} \label{Eq:SummedOccupation}
\begin{aligned}
V(A) 
&= \mathbb{E}_{a}\left[ \int_{0}^{C_{i}}\mathbf{1}(X^{\nu(a,y)}_{t} \in A) dt \right] = V^{(1)}(A) + V^{(2)}(A) \\
& = (S(y)-S(a)) \Big(\frac{(\Psi_{\lambda}\mathbf{1}_{A})(a)}{\psi_{\lambda}(a)} + m(A\cap(a,y)) + \frac{(\Phi_{\lambda}\mathbf{1}_{A})(y)}{\varphi_{\lambda}(y)} \Big).
\end{aligned}
\end{equation}

\smallskip
{\it Stationary distribution of $X^{\nu(a,y)}$.}
Following the arguments from renewal theory \citep{Harrison1985, KarlinTaylor1981}, we can express the stationary measure, independently of the initial state $x$, as
\begin{equation*}
\mu_{\lambda}^{(a,y)}(A) 
= \frac{\mathbb{E}_a \big[ \int_0^{C_i} \mathbf{1}(X^{\nu(a,y)}_t \in A) \, dt \big]}{\mathbb{E}_a[C_i]} = \frac{ \frac{V(A)}{S(y)-S(a)}}{\frac{(\Psi_{\lambda}1)(a)}{\psi_{\lambda}(a)} + m((a,y)) + \frac{(\Phi_{\lambda}1)(y)}{\varphi_{\lambda}(y)}}, \quad \text{for Borel sets } A \subset \mathcal{I},
\end{equation*}
where the latter equality follows from \eqref{Eq:SummedOccupation} and by taking $A = \mathcal{I}$ in \eqref{Eq:SummedOccupation} which yields the mean cycle length  
$$
\mathbb{E}_{a}[C_{i}] 
= (S(y)-S(a)) \, \Big( \frac{(\Psi_{\lambda}1)(a)}{\psi_{\lambda}(a)} + m((a,y)) + \frac{(\Phi_{\lambda}1)(y)}{\varphi_{\lambda}(y)} \Big).
$$
Hence, the density $\mu_{\lambda}^{(a,y)}(dx)$ takes the desired form and completes the proof. 
\end{proof}

\section{Constrained Ergodic Singular Control Problem} \label{sec:OCP}

The main aim of this section is to solve the constrained ergodic singular control problem of the representative player, by maximising the long-run average payoff functional $J$ in \eqref{Eq:ObjectiveImproved} over the admissible set $\mathcal{A}$ of controls, defined by \eqref{admissible controls}.
For any given and fixed mean-field parameter $\theta \in \mathbb{R}_{+}$ and initial state $x \in \mathcal{I}$, the objective is to find the maximum long-run average payoff 
\begin{equation} \label{Eq:BetaLambda}
\beta(x;\theta) := \sup_{\nu \in \mathcal{A}} J(x; \nu, \theta).
\end{equation}

\subsection{Restricting to the class of Poissonian singular controls of threshold type} \label{sec:barriers}

In this section, we focus on the class of two-sided Poissonian singular control processes $\nu = \nu(a,y)$ at thresholds $a$ and $y$, with $\mathfrak{l}<a<y<\mathfrak{r}$, so that the corresponding Poissonian singularly controlled process $X^{\nu(a,y)}$ at thresholds $(a,y)$ satisfies \eqref{Xnu} with $\nu(a,y) = \nu^+(a) - \nu^-(y)$ defined by \eqref{nudownup}. 
That is, if the process is observed at a Poisson jump time $T^{(i)}$ below $a$ (resp., above $y$), then it is immediately dragged up to $a$ (resp., down to $y$) and started anew; otherwise the process is not controlled.

The aim of the decision maker in this problem is to thus select an optimal pair of thresholds $(a^*, y^*) = (a^*(\theta), y^*(\theta))$ that maximises 
the associated expected long-run average payoff (cf.~\eqref{Eq:ObjectiveImproved})
\begin{equation} \label{Eq:PayoffAandY}
\begin{aligned}
\kappa(a,y;\theta) 
&:= \limsup_{T \to \infty} \frac{1}{T} \, \mathbb{E}_x \Big[ \int_0^{T} \Big( \pi(X^{\nu(a,y)}_t,\theta) \, dt - q_u \, d\nu_t^+(a) + q_d \, d\nu_t^-(y) \Big) \Big] 
\end{aligned}   
\end{equation}
where $\nu(a,y) = \nu^+(a) - \nu^-(y) \in \mathcal{V}$ takes the form \eqref{nudownup}. 
We henceforth denote the corresponding value of this Poissonian singular control problem of threshold type by 
\begin{equation} \label{Eq:PayoffForAandY}
\kappa^*(\theta) 
:= \kappa(a^*(\theta),y^*(\theta);\theta)
= \max_{(a,y) \in \mathcal{I}^2 \,:\, a<y} \kappa(a,y;\theta) .
\end{equation}
% and its potential by $w(\cdot;\theta)$.}

\begin{theorem} \label{Theorem: Expected long-run average payoff}
Suppose that $\lambda > 0$ and $a, y \in \mathcal{I}$ satisfying $\mathfrak{l}<a<y<\mathfrak{r}$, and recall the two-sided Poissonian singularly controlled processes $X^{\nu(a,y)}$ at the thresholds $(a,y)$ satisfying \eqref{Xnu} with \eqref{nudownup}. 
Then, 
% \begin{equation} \label{Eq:PayoffForAandY}
% \begin{aligned}
% \kappa(a,y;\theta) 
% &= \frac{1}{N(a,y)}\bigg[\int_{a}^{y} \pi(z,\theta) m(dz) +\frac{(\Phi_\lambda  \pi_d(\cdot,\theta))(y)}{\varphi_\lambda(y)}  + \frac{(\Psi_\lambda  \pi_u(\cdot,\theta))(a)}{\psi_\lambda(a)} - \frac{q_u}{S'(a)} + \frac{q_d}{S'(y)} \bigg],\\
% \end{aligned}   
% \end{equation}
% where $N(a,y)$ is given by 
% \begin{equation*}
% N(a,y) := m((a,y)) + \dfrac{(\Psi_\lambda  1)(a)}{\psi_\lambda(a)}+ \dfrac{(\Phi_\lambda  1)(y)}{\varphi_\lambda(y)}.
% \end{equation*}
the thresholds $(a^*,y^*) \equiv (a^*(\theta),y^*(\theta))$ maximising \eqref{Eq:PayoffAandY} satisfy the %first-order 
conditions
\begin{equation} \label{eq:FOC1}
\begin{aligned} 
&\frac{\dfrac{q_d}{S'(y^*)} - \dfrac{q_u}{S'(a^*)} + \int^{y^*}_{a^*} \pi(z,\theta)m(dz)}{m((a^*,y^*))} %\\& 
= \frac{(\Psi_\lambda  \pi_u(\cdot,\theta))(a^*)}{(\Psi_\lambda  1)(a^*)} 
= \frac{(\Phi_\lambda  \pi_d(\cdot,\theta))(y^*)}{(\Phi_\lambda  1)(y^*)},
\end{aligned}
\end{equation}
and the resulting long-run average payoff \eqref{Eq:PayoffForAandY} associated to $(a^*,y^*) \equiv (a^*(\theta),y^*(\theta))$ is given by 
\begin{equation*}
\kappa^*(\theta) = \frac{1}{m((a^*(\theta),y^*(\theta)))} \bigg(\frac{q_d}{S'(y^*(\theta))} - \frac{q_u}{S'(a^*(\theta))} +\int^{y^*(\theta)}_{a^*(\theta)} \pi(z,\theta)m(dz) \bigg).
\end{equation*}
\end{theorem}
\begin{proof}
We prove the result in the following two steps. 

\smallskip
{\it Step 1: An analytical expression for the expected long-run average payoff from \eqref{Eq:PayoffAandY}}. 
We find by the compensation formula \cite[Theorem 4.4]{kyprianou2014}, %the standard ergodic results for diffusion \cite[pp. 36--38]{BorodinSalminen} that
the intensity $\lambda$ of $N$ and the stationary distribution $\mu_{\lambda}^{(a,y)}$ of $X^{\nu(a,y)}$ from  Theorem~\ref{Theorem:TwoSidedDensity}, that
\begin{align*}
\lim_{T \to \infty} \frac{1}{T} \, \mathbb{E}_x \Big[ \int_0^T  d\nu_s^+(a) \Big] 
&= \lim_{T \to \infty} \frac{1}{T} \, \mathbb{E}_x \Big[ \int_0^T (a-X^{\nu(a,y)}_{t-})^+  dN_t  \Big] \\
&= \lim_{T \to \infty} \frac{\lambda}{T} \, \mathbb{E}_x \Big[ \int_0^T (a-X^{\nu}_{t-})^+ dt \Big] 
% = \lambda \, \mathbb{E}_{\mu_{\lambda}^{(a,y)}}[(a-X^{\nu(a,y)})^+] 
= \lambda \int_{\mathfrak{l}}^a (a-z) %(z-a) 
\, \mu_{\lambda}^{(a,y)}(dz), \\
%\end{align*}
%\begin{align*}
\lim_{T \to \infty} \frac{1}{T} \, \mathbb{E}_x \Big[ \int_0^T  d\nu_t^-(y) \Big] 
&= \lim_{T \to \infty} \frac{1}{T} \, \mathbb{E}_x \Big[ \int_0^T (X^{\nu(a,y)}_{t-}-y)^+  dN_t  \Big] \\
&= \lim_{T \to \infty} \frac{\lambda}{T} \, \mathbb{E}_x \Big[ \int_0^T (X^{\nu(a,y)}_{t-}-y)^+  dt  \Big] 
% = \lambda \mathbb{E}_{\mu_{\lambda}^{(a,y)}}[(X^{\nu}-y)^+] 
= \lambda \int_y^{\mathfrak{r}} (z-y) \, \mu_{\lambda}^{(a,y)}(dz).
\end{align*}
%where $\mu_{\lambda}^{(a,y)}$ is the stationary distribution given in Theorem~\ref{Theorem:TwoSidedDensity}.
Thus, the average payoff from \eqref{Eq:PayoffAandY} takes the form
\begin{equation*}
\kappa(a,y;\theta) 
= \int_{\mathfrak{l}}^{\mathfrak{r}} \pi(z,\theta) \, \mu_{\lambda}^{(a,y)}(dz) 
+ \lambda \, q_u \int_{\mathfrak{l}}^a (z-a) \, \mu_{\lambda}^{(a,y)}(dz)
+ \lambda \, q_d \int_y^{\mathfrak{r}} (z-y) \, \mu_{\lambda}^{(a,y)}(dz).
\end{equation*}
Substituting the stationary distribution from Theorem \ref{Theorem:TwoSidedDensity}, we find
\begin{align} \label{kN}
\kappa(a,y;\theta) \, N(a,y) 
&= \int_{a}^{y} \pi(z,\theta) \, m(dz) 
+ \int_{\mathfrak{l}}^{a} \pi(z,\theta) \, \frac{\psi_\lambda(z)}{\psi_\lambda(a)} \, m(dz) 
+ \int_y^{\mathfrak{r}} \pi(z,\theta) \, \frac{\varphi_\lambda (z)}{\varphi_\lambda(y)} \, m(dz) \notag\\ 
&\quad + \lambda \, q_u \int_{\mathfrak{l}}^a (z-a) \, \frac{\psi_\lambda(z)}{\psi_\lambda(a)} \, m(dz) + \lambda \, q_d \int_y^{\mathfrak{r}} (z-y) \, \frac{\varphi_\lambda (z)}{\varphi_\lambda(y)} \, m(dz),
\end{align}
where $N(a,y) := m((a,y)) + \frac{(\Psi_\lambda 1)(a)}{\psi_\lambda(a)}+ \frac{(\Phi_\lambda 1)(y)}{\varphi_\lambda(y)}$.

% {\color{olive}
% Then, notice that by \eqref{Eq:GeneratorCanonical} and the relation
% $\mathcal{G}\varphi_\lambda=\lambda\varphi_\lambda$, the definitions of $S'$ and $m'$ from \eqref{Eq:Sm}, we obtain by 
% differentiation that 
% \begin{align*}
% \frac{d}{dx} \bigg( x \, \frac{\varphi_{\lambda}'(x)}{S'(x)} - \frac{\varphi_{\lambda}(x)}{S'(x)} \bigg) 
% = -m'(x) \, \varphi_{\lambda}(x) \, \big( b(x)- \lambda x \big).
% \end{align*}
% Then, the fundamental theorem of calculus and the natural boundary $\mathfrak{r}$ yields
% \begin{equation*} 
% y \, \frac{\varphi_{\lambda}'(y)}{S'(y)} - \frac{\varphi_{\lambda}(y)}{S'(y)} 
% = \int_{y}^{\mathfrak{r}} m'(z) \, \varphi_{\lambda}(z) \, \big( b(z)- \lambda z \big) \, dz.
% \end{equation*} 
% Then using the identity $\frac{d \varphi_\lambda}{dS}(x) = - \lambda (\Phi_{\lambda}1)(x)$ we finally conclude that
% \begin{equation} \label{Eq:zyInt}
% \lambda \int_y^{\mathfrak{r}} (z-y) \, \varphi_\lambda(z) \, m(dz) 
% = \int_y^{\mathfrak{r}} b(z) \, \varphi_\lambda(z) \, m(dz) + \frac{\varphi_\lambda(y)}{S'(y)}. 
% \end{equation}
% Similarly we have
% \begin{equation} \label{Eq:zaInt}
% \lambda \int_{\mathfrak{l}}^a (z-a) \, \psi_\lambda(z) \, m(dz) 
% = \int_{\mathfrak{l}}^a b(z) \, \psi_\lambda(z) \, m(dz) - \frac{\psi_\lambda(a)}{S'(a)}.
% \end{equation}
% }
Substituting \eqref{Eq:zyInt}--\eqref{Eq:zaInt} into \eqref{kN} and
collecting terms yields
\begin{align*}
\kappa(a,y;\theta) \, N(a,y) 
&= \int_{a}^{y} \pi(z,\theta) m(dz)  + \frac{\int_{\mathfrak{l}}^{a} \pi(z,\theta)  \psi_\lambda(z) m(dz)}{\psi_\lambda(a)} + \frac{\int_y^{\mathfrak{r}} \pi(z,\theta)  \varphi_\lambda (z) m(dz)}{\varphi_\lambda(y)} \\ 
&\hspace{-8mm}+ \frac{q_u}{\psi_\lambda(a)} \Big( \int_{\mathfrak{l}}^a b(z)  \psi_\lambda(z) m(dz) 
-\frac{\psi_\lambda(a)}{S'(a)} \Big) + \frac{q_d}{\varphi_\lambda(y)} \Big( \int_y^{\mathfrak{r}} b(z)  \varphi_\lambda (z) m(dz) +\frac{\varphi_\lambda(y)}{S'(y)} \Big) \\
\label{Eq:Payoff_formula}
&= \int_{a}^{y} \pi(z,\theta) \, m(dz) + \frac{(\Phi_\lambda  \pi_d(\cdot,\theta))(y)}{\varphi_\lambda(y)} + \frac{(\Psi_\lambda  \pi_u(\cdot,\theta))(a)}{\psi_\lambda(a)} + \frac{q_d}{S'(y)} - \frac{q_u}{S'(a)}, 
\end{align*}
which results into takes the form
\begin{equation} \label{Eq:PayoffAandYnew}
\begin{aligned}
\kappa(a,y;\theta) 
&= \frac{\int_{a}^{y} \pi(z,\theta) m(dz) + \frac{(\Psi_\lambda  \pi_u(\cdot,\theta))(a)}{\psi_\lambda(a)} +\frac{(\Phi_\lambda  \pi_d(\cdot,\theta))(y)}{\varphi_\lambda(y)} - \frac{q_u}{S'(a)} + \frac{q_d}{S'(y)}}{m((a,y)) + \frac{(\Psi_\lambda  1)(a)}{\psi_\lambda(a)}+ \frac{(\Phi_\lambda  1)(y)}{\varphi_\lambda(y)}}.
\end{aligned}   
\end{equation}

\smallskip
{\it Conditions of optimality for $(a^*(\theta),y^*(\theta))$.}
To find the first-order conditions for the optimal thresholds $a^*$ and $y^*$, we first differentiate the expression \eqref{Eq:PayoffAandYnew} of $\kappa$ from Step 1 with respect to $a$, to get 
\begin{equation} \label{kPsiPhi}
\kappa(a,y;\theta) \, \Big( -\frac{\psi_\lambda '(a)}{\psi_\lambda ^2(a)} \, (\Psi_\lambda  1)(a) \Big) + N(a,y) \, \frac{\partial}{\partial a} \kappa(a,y;\theta)  
= -\frac{\psi_\lambda '(a)}{\psi_\lambda(a)^2}(\Psi_\lambda  \pi_u(\cdot,\theta))(a) ,
\end{equation}
hence the condition $\frac{\partial}{\partial a} \kappa(a,y;\theta) = 0$ holds true if and only if
\begin{align*}
&(\Psi_\lambda  \pi_u(\cdot,\theta))(a) \Big( m((a,y)) + \frac{(\Psi_\lambda 1)(a)}{\psi_\lambda(a)} + \frac{(\Phi_\lambda 1)(y)}{\varphi_\lambda(y)} \Big) \\
& - (\Psi_\lambda  1)(a) \Big( \int_{a}^{y} \pi(z,\theta) \, m(dz) + \frac{(\Phi_\lambda  \pi_d(\cdot,\theta))(y)}{\varphi_\lambda(y)} + \frac{(\Psi_\lambda \pi_u(\cdot,\theta))(a)}{\psi_\lambda(a)} + \frac{q_d}{S'(y)} - \frac{q_u}{S'(a)} \Big) = 0.
\end{align*}
Straightforward manipulations yield 
\begin{equation} \label{Eq:aFOC}
\begin{aligned}
& \frac{(\Psi_\lambda  \pi_u(\cdot,\theta))(a)}{(\Psi_\lambda  1)(a)} \bigg(1 + \frac{(\Phi_\lambda 1)(y)}{\varphi_\lambda(y)m((a,y))}\bigg) \\ & 
= \frac{(\Phi_\lambda  \pi_d(\cdot,\theta))(y)}{\varphi_\lambda(y)m((a,y))} + \frac{1}{m((a,y))}\bigg( \int_a^y \pi(z,\theta)m(dz)+  \frac{q_d}{S'(y)} - \frac{q_u}{S'(a)}\bigg).
\end{aligned}
\end{equation}
Differentiating with respect to $y$ gives the analogous result
\begin{equation} \label{Eq:yFOC}
\begin{aligned}
&\frac{(\Phi_\lambda  \pi_d(\cdot,\theta))(y)}{(\Phi_\lambda  1)(y)} \bigg( 1 + \frac{(\Psi_\lambda 1)(a)}{\psi_\lambda(a)m((a,y))} \bigg) \\
&= \frac{(\Psi_\lambda  \pi_u(\cdot,\theta))(a)}{\psi_\lambda(a)m((a,y))} + \frac{1}{m((a,y))} \bigg( \int_a^y \pi(z,\theta) m(dz) + \frac{q_d}{S'(y)} - \frac{q_u}{S'(a)}\bigg).
\end{aligned}
\end{equation}
% Subtracting these first-order conditions from each other we get
% \begin{equation*}
% \begin{aligned}
%     & \frac{(\Psi_\lambda  \pi_u(\cdot,\theta))(a)}{(\Psi_\lambda  1)(a)} \bigg(1 + \dfrac{(\Phi_\lambda 1)(y)}{\varphi_\lambda(y)m((a,y))}\bigg)-\frac{(\Phi_\lambda  \pi_d(\cdot,\theta))(y)}{(\Phi_\lambda  1)(y)} \bigg( 1 + \frac{(\Psi_\lambda  1)(a)}{\psi_\lambda(a)m((a,y))} \bigg) 
%     \\
%     & = \frac{(\Phi_\lambda  \pi_d(\cdot,\theta))(y)}{\varphi_\lambda(y)m((a,y))}- \frac{(\Psi_\lambda  \pi_u(\cdot,\theta))(a)}{\psi_\lambda(a)m((a,y))},
% \end{aligned}
% \end{equation*}
Subtracting \eqref{Eq:yFOC} from \eqref{Eq:aFOC} and collecting common terms yields
\begin{equation*}
\bigg(\frac{(\Psi_\lambda  \pi_u(\cdot,\theta))(a)}{(\Psi_\lambda  1)(a)} -  \frac{(\Phi_\lambda  \pi_d(\cdot,\theta))(y)}{(\Phi_\lambda  1)(y)}\bigg)\bigg( 1+ \frac{(\Psi_\lambda 1)(a)}{\psi_\lambda(a)m((a,y))} + \frac{(\Phi_\lambda 1)(y)}{\varphi_\lambda(y)m((a,y))} \bigg) = 0,
\end{equation*}
and since the second factor is strictly positive, it follows that
\begin{equation*}
\frac{(\Psi_\lambda  \pi_u(\cdot,\theta))(a)}{(\Psi_\lambda  1)(a)} =  \frac{(\Phi_\lambda  \pi_d(\cdot,\theta))(y)}{(\Phi_\lambda  1)(y)}.
\end{equation*}
Using this equation in \eqref{Eq:yFOC}--\eqref{Eq:aFOC} then yields that the optimal pair $(a^*(\theta),y^*(\theta))$ must satisfy \eqref{eq:FOC1}. 

\smallskip
{\it Value associated with $(a^*(\theta),y^*(\theta))$.}
Using conditions \eqref{eq:FOC1} together with \eqref{kPsiPhi} and the fact that $\frac{\partial}{\partial a} \kappa(a^*(\theta),y^*(\theta);\theta) = 0$, which yield the desired expression for $\kappa^*(\theta)$ and complete the proof.
\end{proof}

\begin{remark}
Notably, the long-run average payoff $\kappa^*(\theta)$ obtained in Theorem~\ref{Theorem: Expected long-run average payoff}.(iii) has the same qualitative form as that arising in the unconstrained singular control problem (see, e.g., \citep{Alvarez2018}), where the optimal thresholds serve as reflecting boundaries. 
However, the threshold values $(a^*(\theta), y^*(\theta))$ differ between the two settings.
\end{remark}

In light of the conditions \eqref{eq:FOC1} derived in Theorem \ref{Theorem: Expected long-run average payoff}, to establish the existence of (optimal) thresholds $(a^*(\theta),y^*(\theta))$ maximising \eqref{Eq:PayoffForAandY}, we must first study the following functions 
\begin{equation} \label{Hdu} 
H_d(x) := \frac{(\Phi_\lambda  \pi_d(\cdot,\theta))(x)}{(\Phi_\lambda  1)(x)}, \quad H_u(x):= \frac{(\Psi_\lambda  \pi_u(\cdot,\theta))(x)}{(\Psi_\lambda  1)(x)}, 
\quad x \in \mathcal{I}.
\end{equation}
The next lemma presents useful properties of these functions.

\begin{lemma} \label{Lemma: FractionProperties}
Recall the functions $H_d$ and $H_u$ defined by \eqref{Hdu} and the unique global maximisers $\hat{x}_d$ and $\hat{x}_u$ of $\pi_d$ and $\pi_u$, respectively, from Assumption~\ref{Ass:Payoff}. 
Then, 
\begin{enumerate}[\rm (i)]
\item There exists a unique $\hat{z}_d = \argmax \{ H_d\} < \hat{x}_d$ such that $H_d$ is increasing on $(\mathfrak{l}, \hat{z}_d)$ and decreasing on $(\hat{z}_d, \mathfrak{r})$. 
Furthermore, $\hat{z}_d$ is the unique solution of $\pi_d(\hat{z}_d,\theta) = H_d(\hat{z}_d)$;

\item  There exists a unique $\hat{z}_u = \argmax \{ H_u(x) \} > \hat{x}_u$ such that $ H_u$ is increasing on $(\mathfrak{l}, \hat{z}_u)$ and decreasing on $(\hat{z}_u, \mathfrak{r})$. Furthermore, $\hat{z}_u$ is the unique solution of $\pi_u(\hat{z}_u,\theta) = H_u(\hat{z}_u)$;

\item $\lim_{x \to \mathfrak{l}} H_d(x) = \pi_d(\mathfrak{l}+,\theta)$ and $\lim_{x \to \mathfrak{l}} H_u(x) = \pi_u(\mathfrak{l}+,\theta)$;

\item $\lim_{x \to  \mathfrak{r}} H_d(x) = \lim_{x \to \mathfrak{r}} H_u(x) = - \infty.$
\end{enumerate}
\end{lemma}

\begin{proof}
We focus on proving the results concerning $H_d$ as the ones concerning $H_u$ can be proved by using completely analogous arguments.

We first establish few auxiliary properties. 
By Assumption \ref{Ass:Payoff}.(iii) and $q_u > q_d$, we get 
\begin{equation*}
0 \leq \pi_d(\mathfrak{l}+,\theta)-\pi_u(\mathfrak{l}+,\theta) = (q_d-q_u)b(\mathfrak{l+})
\quad \Rightarrow \quad 
b(\mathfrak{l}+) \leq 0.
\end{equation*}
Since $b'(x) \leq -c$ for $x \in \mathcal{I}$ by Assumption \ref{Ass:X}.(iii), we therefore get that $b(x) < 0$ for $x \in \mathcal{I}$. 
This shows
that $S''(x) = -\frac{2b(x)}{\sigma^2(x)}S'(x) > 0$, hence  
\begin{equation} \label{Eq:S'}
S' \text{ is increasing} 
\quad \Rightarrow \quad 
\lambda(\Phi_\lambda1)(x) = - \frac{\varphi_\lambda'(x)}{S'(x)} \geq - \frac{\varphi_\lambda'(x)}{S'(a)}, \quad x \leq a. 
\end{equation}
Also, by Assumption~\ref{Ass:Payoff}.(iii),
$\liminf_{x \downarrow \mathfrak{l}} \frac{\partial}{\partial x} \pi_u(x, \theta) > 0$, which means that there exist $\varepsilon > 0$ and $a \in (\mathfrak{l}, \hat x_d)$ such that $\frac{\partial}{\partial x} \pi_u(x, \theta) \geq \varepsilon$ on $(\mathfrak{l}, a)$. 
Using this together with $q_u > q_d$ and $b' \leq -c$ by Assumption~\ref{Ass:X}.(iii), we obtain
\begin{equation} \label{Eq:pid_-x}
\frac{\partial}{\partial x} \pi_d(x, \theta) =\frac{\partial}{\partial x} \pi_u(x, \theta) + (q_d - q_u) b'(x) \geq \varepsilon, \quad x \in (\mathfrak{l}, a).
\end{equation}

Then, using \eqref{Eq:S'} and the limit
$\varphi_\lambda(\mathfrak{l}+)=\infty$ (as $\mathfrak{l}$ is natural), we obtain by \eqref{Eq:pid_-x} that
\begin{equation*}
\int_x^a \frac{\partial}{\partial z} \pi_d(z, \theta) \, (\Phi_\lambda 1)(z) \, dz 
\geq \frac{\varepsilon}{\lambda S'(a)} (\varphi_\lambda(x) - \varphi_\lambda(a)) 
\rightarrow \infty
\quad \text{ as } x \downarrow \mathfrak{l}.
\end{equation*}
Moreover, we know from differentiation that  
\begin{align*}
& \frac{d}{dx}\big[
(\Phi_\lambda1)(x) \big(H_d(x)-\pi_d(x,\theta)\big)
\big] = - \frac{\partial}{\partial x}\pi_d(x,\theta) \,(\Phi_\lambda1)(x), \quad x \in \mathcal{I}.
\end{align*}
This identity implies firstly, after integrating from $x$ to $a$ and using the above limit, that
\begin{equation} \label{Philinfty}
(\Phi_\lambda1)(x) \big( H_d(x) - \pi_d(x,\theta) \big) \rightarrow \infty \quad \text{ as } x \downarrow \mathfrak{l},
\end{equation}
and secondly, after using $\pi_d(\cdot,\theta)$ is increasing on $(\mathfrak{l},\hat x_d)$ from Assumption~\ref{Ass:Payoff}.(iii), that 
\begin{equation} \label{Phidecr}
x \mapsto (\Phi_\lambda1)(x)(H_d(x) - \pi_d(x,\theta)) \text{ is decreasing on } (\mathfrak{l},\hat x_d).
\end{equation}
Finally, given that $\hat{x}_d$ is the unique global maximiser of $\pi_d(\cdot,\theta)$ from Assumption~\ref{Ass:Payoff}.(iii), we also have from \eqref{Hdu} and \eqref{PsiPhi} that 
\begin{align} \label{Phihatx}
\begin{split}
(\Phi_\lambda1)(\hat{x}_d)(H_d(\hat{x}_d) - \pi_d(\hat{x}_d,\theta)) 
&= (\Phi_\lambda  \pi_d(\cdot,\theta))(\hat{x}_d) 
- (\Phi_\lambda1)(\hat{x}_d) \, \pi_d(\hat{x}_d,\theta)) \\
&= \int_{\hat{x}_d}^{\mathfrak{r}} \varphi_{\lambda}(z) \big(\pi_d(z, \theta)-\pi_d(\hat{x}_d, \theta) \big) m(dz) < 0.
\end{split}
\end{align}

Consequently, \eqref{Philinfty}--\eqref{Phihatx} imply thanks to the intermediate value theorem, that there exists exactly one zero $\hat{z}_d \in (\mathfrak{l}, \hat{x}_d)$ of the map in \eqref{Phidecr}. 
Since $(\Phi_\lambda1)(x) > 0$, we conclude that $x = \hat{z}_d \in (\mathfrak{l}, \hat{x}_d)$ is the unique solution of $H_d(x)=\pi_d(x, \theta)$, such that 
\begin{align} \label{Eq:Hdpi1}
H_d(x) - \pi_d(x, \theta) \begin{cases}
>0, & x \in (\mathfrak{l}, \hat{z}_d), \\
<0, & x \in (\hat{z}_d, \hat{x}_d).
\end{cases}
\end{align}
Moreover, since $\pi_d(\cdot,\theta)$ is decreasing on $(\hat{x}_d,\mathfrak{r})$ and $\lim_{x \to \mathfrak{r}} \pi_d(x,\theta) =  -\infty$ by Assumption~\ref{Ass:Payoff}.(iii), we have 
\begin{equation} \label{Eq:Hdpi2}
H_d(x) 
= \frac{(\Phi_\lambda  \pi_d(\cdot,\theta))(x)}{(\Phi_\lambda  1)(x)} 
< \frac{\pi_d(x,\theta)(\Phi_\lambda 1)(x)}{(\Phi_\lambda 1)(x)} 
= \pi_d(x,\theta),
\quad x \geq \hat{x}_d.
\end{equation}
Combining \eqref{Eq:Hdpi1}--\eqref{Eq:Hdpi2} together with the calculation
\begin{align} \label{eq:Rprime}
& H_d'(x) = \frac{\varphi_\lambda(x) m'(x)}{(\Phi_\lambda  1)(x)} \big( H_d(x) - \pi_d(x,\theta) \big), , \quad x \in \mathcal{I}, 
\end{align}
we conclude that that $H_d$ is increasing on %$(\mathfrak{l}, \hat{x}_d)$. 
$(\mathfrak{l}, \hat{z}_d)$ and decreasing on $(\hat{z}_d, \mathfrak{r})$, and that $\hat{z}_d \in (\mathfrak{l}, \hat{x}_d)$ is the unique solution of $H_d(\cdot)=\pi_d(\cdot,\theta)$ on the whole $\mathcal{I}$, which proves part (i).

Thanks to the monotonicity of $S'$ from \eqref{Eq:S'} and $\lambda(\Phi_\lambda1) = -\varphi_\lambda'/S'$, we have
\begin{equation*}
\varphi_{\lambda}(x)- \varphi_{\lambda}(a) = \lambda \int_x^a S'(z)(\Phi_{\lambda}1)(z)dz \leq \lambda S'(a)(a-\mathfrak{l})(\Phi_{\lambda}1)(x),
\end{equation*}
which implies thanks to $\varphi_{\lambda}(\mathfrak{l+}) = \infty$ (since $\mathfrak{l}$ is natural) that $(\Phi_\lambda1)(x) \to \infty$ as $x \downarrow \mathfrak{l}$. 
Moreover, thanks to the limit in \eqref{Philinfty}, the positivity of $\pi_d(\cdot, \theta)$ near $\mathfrak{l}$ thanks to Assumption~\ref{Ass:Payoff}.(iii), and  
$$
(\Phi_{\lambda} \pi_d(\cdot, \theta))(x) = (\Phi_{\lambda}1)(x)(H_d(x)-\pi_d(x, \theta)) + \pi_d(x, \theta)(\Phi_{\lambda}1)(x),
$$
we further conclude that $(\Phi_\lambda \pi_d(\cdot,\theta))(x) \to \infty$ as $x \downarrow \mathfrak{l}$.
Hence, L'Hopital's rule implies that
\begin{equation*}
\lim_{x \downarrow \mathfrak{l}} H_d(x) = \lim_{x \downarrow \mathfrak{l}} \frac{-\varphi_\lambda(x) \, \pi_d(x, \theta) \, m'(x)}{-\varphi_\lambda(x) \, m'(x)} = \pi_d(\mathfrak{l}+,\theta),
\end{equation*}
which proves part (iii).

Finally, taking the limits on both sides of inequality \eqref{Eq:Hdpi2} also yields by Assumption~\ref{Ass:Payoff}.(iii) that $\lim_{x \to \mathfrak{r}} H_d(x) = -\infty$, proving (iv) and completing the proof. 
\end{proof}

In light of the properties established in Lemma~\ref{Lemma: FractionProperties} for the functions $H_d$ and $H_u$, we are now ready to prove that the system of equations \eqref{eq:FOC1} obtained in Theorem~\ref{Theorem: Expected long-run average payoff} for an optimal pair of thresholds admit a unique solution. 

\begin{theorem} \label{thm:barriers}
Recall the system of equations in the optimality condition \eqref{eq:FOC1} of Theorem~\ref{Theorem: Expected long-run average payoff}.
\begin{enumerate}[\rm (i)]
\item Every solution $(a,y)$ of \eqref{eq:FOC1} with $\mathfrak{l} < a < y <\mathfrak{r}$ satisfies $a < \hat{z}_u$ and $y > \hat{z}_d$, where $\hat{z}_u$ and $\hat{z}_d$ are defined in Lemma~\ref{Lemma: FractionProperties};

\item There exists a unique pair $(a^*,y^*) \equiv (a^*(\theta),y^*(\theta))$ with $\mathfrak{l}<a^*(\theta)<y^*(\theta)<\mathfrak{r}$ solving \eqref{eq:FOC1}.
\end{enumerate}
\end{theorem}
\begin{proof}
We prove the result in the following steps. 

\smallskip
{\it Proof of part }(i).
Using \eqref{Eq:DriftToScale}, the system of equations \eqref{eq:FOC1} can be expressed as 
% \begin{equation} \label{Eq:algebraic-foc}
% \begin{aligned}
% & \frac{1}{m((a^*(\theta), y^*(\theta)))} \bigg[\int^{y^*(\theta)}_{a^*(\theta)} \pi_d(z,\theta)m(dz)   +\frac{q_d-q_u}{S'(a^*(\theta))}\bigg] = \frac{(\Phi_\lambda  \pi_d(\cdot,\theta))(y^*(\theta))}{(\Phi_\lambda  1)(y^*(\theta))}, \\
% & \frac{1}{m((a^*(\theta), y^*(\theta)))} \bigg[\int^{y^*(\theta)}_{a^*(\theta)} \pi_u(z,\theta)m(dz)  + \frac{q_d-q_u}{S'(y^*(\theta))}\bigg] = \frac{(\Psi_\lambda  \pi_u(\cdot,\theta))(a^*(\theta))}{(\Psi_\lambda  1)(a^*(\theta))}.
% \end{aligned}
% \end{equation}
\begin{align} \label{Eq: OptRewritten1}
\begin{split}
\int_a^y \big( \pi_u(z, \theta)-H_u(a) \big) m(dz) - \frac{q_u - q_d}{S'(y)} &= 0 %, \\
= \int_a^y \big( \pi_d(z, \theta)-H_d(y) \big) m(dz) - \frac{q_u - q_d}{S'(a)} %&= 0. 
\end{split}
\end{align}
Assume (aiming for a contradiction) that $a \geq \hat{z}_u$. 
By Lemma~\ref{Lemma: FractionProperties}.(ii), $\hat{z}_u > \hat{x}_u$ and $H_u'(a) \leq 0$, implying that  
\begin{equation} \label{piuHu}
H_u'(a) = \frac{\psi_\lambda(a) \,m'(a)}{(\Psi_\lambda1)(a)} \big(\pi_u(a,\theta)-H_u(a) \big) \leq 0 
\quad \Rightarrow \quad 
\pi_u(a,\theta) \leq H_u(a), 
\quad a \geq \hat{z}_u .
\end{equation}
Then, Assumption~\ref{Ass:Payoff}.(iii) gives
\begin{equation*}
\pi_u(z,\theta)\leq\pi_u(a,\theta)\leq H_u(a), \quad z \in [a,y], \quad a \geq \hat{z}_u > \hat{x}_u.
\end{equation*}
This yields that the integral in the left-hand side of \eqref{Eq: OptRewritten1} is nonpositive, and since $q_u > q_d$ and $S'(y) > 0$, the whole left-hand side is strictly negative, which is a contradiction. Therefore $a < \hat{z}_u$ and using similar arguments we also have $y > \hat{z}_d$.

\smallskip
{\it Proof of part }(ii). We prove this in the following two steps. 

\smallskip
{\it Step 1: Reduce the system \eqref{eq:FOC1} to a one-dimensional equation.}
Define the non-empty set
\begin{equation*}
\mathbf{A} := \left\{ a\in(\mathfrak{l},\hat z_u] : H_u(a)\leq H_d(a\vee\hat z_d) \right\},
\end{equation*}
where non-emptiness results from Lemma~\ref{Lemma: FractionProperties}.(i),(iii) and Assumption~\ref{Ass:Payoff}.(iii) which imply 
\begin{equation} \label{Eq:non_empty_A}
H_u(\mathfrak l+) 
= \pi_u(\mathfrak l+,\theta)
\leq \pi_d(\mathfrak l+,\theta)
= H_d(\mathfrak l+)
< H_d(\hat z_d).
\end{equation} 
Thanks to $H_d$ being strictly decreasing on $[\hat z_d,\mathfrak r)$ from Lemma~\ref{Lemma: FractionProperties}.(i), we get that  
\begin{equation} \label{ya}
\forall \; a\in\mathbf{A}, \quad \exists \; ! \; y_a \in [a\vee\hat z_d,\mathfrak r) \quad \text{such that} \quad H_d(y_a)=H_u(a),
\end{equation}

Combining the facts that $H_u$ is strictly increasing on $(\mathfrak l,\hat z_u)$ and $a\mapsto H_d(a\vee\hat z_d)$ is non-increasing on $(\mathfrak l,\hat z_u]$ thanks to Lemma~\ref{Lemma: FractionProperties}.(i),(ii), we get that $a \mapsto H_u(a)-H_d(a\vee\hat z_d)$ is continuous and strictly increasing on $(\mathfrak l,\hat z_u]$.
This yields together with \eqref{Eq:non_empty_A} that the set $\mathbf{A}$ takes the form 
\begin{equation} \label{newA}
\mathbf{A}=(\mathfrak l,a_{\max}]
\quad \text{for some $a_{\max}\in(\mathfrak l,\hat z_u]$}, 
\end{equation}
which is given by 
\[
a_{\max} = \begin{cases}
\hat z_u, & H_u(\hat z_u)\leq H_d(\hat z_u\vee\hat z_d),\\
\hat a, & H_u(\hat z_u)>H_d(\hat z_u\vee\hat z_d),
\end{cases}
\quad \text{where } \hat a\in(\mathfrak l,\hat z_u) \text{ uniquely solves }
H_u(\hat a)=H_d(\hat a\vee\hat z_d).
\]
Consequently, if $a_{\max}=\hat a\geq\hat z_d$, then
$y_{a_{\max}}=a_{\max}$, while if
$a_{\max}=\hat a<\hat z_d$, then
$y_{a_{\max}}=\hat z_d>a_{\max}$. Finally, if
$a_{\max}=\hat z_u$, then $H_u(\hat z_u)\leq H_d(\hat z_u\vee\hat z_d)$, and, since $H_d$ is strictly decreasing on $[\hat z_d,\mathfrak r)$, we get $y_{a_{\max}} =y_{\hat  z_u} \geq \hat z_u\vee\hat z_d \geq \hat z_u=a_{\max}$. 
Thus, in all cases, $y_{a_{\max}}\geq a_{\max}$.
Moreover, in view of \eqref{newA} and Lemma~\ref{Lemma: FractionProperties}.(i),(ii), we know that $H_u$ is strictly increasing on $\mathbf{A}$ and $H_d$ is strictly decreasing on $[\hat z_d,\mathfrak r)$, implying that 
\begin{equation} \label{yadecr} a\mapsto y_a = (H_d |_{[\hat z_d,\mathfrak r)})^{-1} (H_u(a)) \quad \text{is continuous and strictly decreasing on $\mathbf{A}$}.
\end{equation}
This implies that $a\mapsto a-y_a$ is strictly increasing, and given the already established $a_{\max} - y_{a_{\max}} \leq 0$, we conclude that 
\begin{equation} \label{a<ya}
a<y_a, \quad \forall \; a\in(\mathfrak l,a_{\max}).
\end{equation}

We now define in view of Step 1 and the first equation of \eqref{eq:FOC1} the function
\begin{equation} \label{eq:g}
g(a) := \int_a^{y_a}\pi(z,\theta) \, m(dz) + \frac{q_d}{S'(y_a)} - \frac{q_u}{S'(a)} - m((a,y_a)) \, H_d(y_a), 
\quad a\in\mathbf{A}, 
\end{equation}
and notice that 
\begin{equation}
\text{$(a,y)$ with $\mathfrak l<a<y<\mathfrak r$ satisfies \eqref{eq:FOC1}} 
\quad \Leftrightarrow \quad 
a\in\mathbf{A},
\quad y=y_a,
\quad g(a)=0.
\end{equation}

\smallskip
{\it Step 2: Prove that $g$ from \eqref{eq:g} has exactly one zero $a^*(\theta) \in \mathbf{A}$ and $a^*(\theta) < y_{a^*(\theta)}$.}
Applying equation \eqref{Eq:DriftToScale} and noting $H_d(y_a)=H_u(a)$, the function $g$ admits the equivalent representations
\begin{align} \label{Eq:gI1}
g(a) &:= \int_a^{y_a} \big(\pi_u(z, \theta) - H_u(a)\big) \, m(dz)-\frac{q_u - q_d}{S'(y_a)}, \\ \label{Eq:gI2}
& \, \, \, = \int_a^{y_a}\big(\pi_d(z, \theta)-  H_d(y_a)\big) \, m(dz) - \frac{q_u - q_d}{S'(a)}.
\end{align}
We prove the desired result in the following four steps. 

\smallskip
{\it Step 2.(a): $g(a_{\max}) < 0$.}
If $a_{\max}=\hat{z}_u$, then \eqref{piuHu} and Lemma~\ref{Lemma: FractionProperties}.(ii) gives $\pi_u(z,\theta) \leq H_u(z) \leq H_u(\hat{z}_u)$ for all $z \in (\hat{z}_u, y_{\hat{z}_u})$, so that \eqref{Eq:gI1} yields
$g(\hat{z}_u) \leq - (q_u-q_d) / S'(y_{\hat{z}_u}) < 0$. 
Otherwise, if $a_{\max} = \hat{a}$ and $y_{\hat{a}} = \hat{a} \vee \hat{z}_d$, then the interval $(\hat{a}, y_{\hat{a}})$ is either empty or equal to $(\hat{a}, \hat{z}_d)$. 
In the latter case, Lemma~\ref{Lemma: FractionProperties}.(i) now gives $\pi_d(z,\theta) \leq H_d(z) \leq H_d(\hat{z}_d) = H_d(y_{\hat{a}})$ for all $z \in (\hat{a}, \hat{z}_d)$, so that the equation \eqref{Eq:gI2} yields
\begin{equation*}
g(\hat{a}) \leq -\frac{q_u-q_d}{S'(\hat{a})} < 0.
\end{equation*}
 
\smallskip
{\it Step 2.(b): $g$ is decreasing on $\mathbf{A}$.}
Let $a_1 < a_2$ be two points in $\mathbf{A}$, and set $y_i := y_{a_i}$ and $K_i:=H_u(a_i)=H_d(y_i)$, for $i=1,2$. Since $y_{a_2} \geq a_2$ and $a \mapsto y_a$ is decreasing, we have $a_1 < a_2 \leq y_2 < y_1$. We split the integral in \eqref{Eq:gI1} of $g(a_1)$ over $(a_1,a_2)$, $(a_2,y_2)$ and $(y_2,y_1)$, and use
\begin{equation*}
\int_{y_2}^{y_1}\pi_u(z,\theta)\,m(dz) - \frac{q_u-q_d}{S'(y_1)} = \int_{y_2}^{y_1}\pi_d(z,\theta)\,m(dz) - \frac{q_u-q_d}{S'(y_2)},
\end{equation*}
which follows form \eqref{Eq:DriftToScale} on $(y_2, y_1)$, to obtain
\begin{equation} \label{Eq: g_diff}
\begin{aligned}
g(a_2) - g(a_1) = & - (K_2 - K_1) m((a_2,y_2)) \\
& - \int_{a_1}^{a_2} \big(\pi_u(z, \theta)- K_1 \big) m(dz) - \int_{y_2}^{y_1} \big(\pi_d(z, \theta) - K_1 \big) m(dz).
\end{aligned}
\end{equation}
Since $K_2 > K_1$ and $m((a_2,y_2)) \geq 0$ the first term on the right-hand side of \eqref{Eq: g_diff} is non-positive, 
the second term is strictly negative since $\pi_u(z, \theta) \geq H_u(z) > H_u(a_1) = K_1$ for $z \in (a_1,a_2) \subset (\mathfrak{l}, \hat{z}_u)$ (by Lemma~\ref{Lemma: FractionProperties}.(ii)), 
and the third term is non-positive  since $\pi_d(z, \theta) \geq H_d(z) \geq H_d(y_1) = K_1$ for $z\in(y_2,y_1)\subset(\hat{z}_d,\mathfrak{r})$ (by Lemma~\ref{Lemma: FractionProperties}.(i)). 
Hence $g(a_2) < g(a_1)$, which implies that $g$ is decreasing.
 
\smallskip
{\it Step 2.(c): $g(a) \to \infty$ as $a \downarrow \mathfrak{l}$.}
Define $\bar y\in(\hat z_d,\mathfrak r)$ by $H_d(\bar y)=\pi_u(\mathfrak l+,\theta)$, which exists uniquely by \eqref{Eq:non_empty_A}, Lemma~\ref{Lemma: FractionProperties}.(iii)--(iv), and the strict decrease of $H_d$ on $(\hat z_d,\mathfrak r)$. 
To see this, note that $H_u(a) \to \pi_u(\mathfrak l+,\theta)$ as $a \downarrow \mathfrak{l}$ and thanks to $H_d(y_a) = H_u(a)$ we have $y_a \to \bar y$ as $a \downarrow \mathfrak{l}$. %Define $\bar{y}$ by $H_d(\bar{y}) = \pi_u(\mathfrak{l}+, \theta)$.

By Assumption~\ref{Ass:Payoff}.(iii), there exist $\eta > 0$ and $x_0 \in (\mathfrak{l}, \min \{\hat{x}_u, \bar{y}\})$ such that $\partial_x \pi_u(x, \theta) \geq \eta$ for all $x \in (\mathfrak{l}, x_0)$. Since $y_a \to \bar{y} > x_0$ as $a \downarrow \mathfrak{l}$, we have $a < x_0 < y_a$ for all $a$ sufficiently near to $\mathfrak{l}$. For such $a$ we have
\begin{equation} \label{ga}
\begin{aligned}
g(a) =& \int_a^{x_0} \big(\pi_u(z, \theta) - H_u(a)\big) m(dz) + \int_{x_0}^{y_a} \big(\pi_u(z,\theta) - H_u(a)\big) m(dz) - \frac{q_u - q_d}{S'(y_a)}.
\end{aligned}
\end{equation}
The last two terms remain bounded as $a \downarrow \mathfrak{l}$ thanks to $y_a \to \bar{y} \in \mathcal{I}$ and $H_u(a) \to \pi_u(\mathfrak{l}+, \theta)$. 
Moreover, since $\pi_u(\cdot,\theta)$ is increasing on $a < x_0 < \hat{x}_u$, we get $H_u(a) \leq \pi_u(a,\theta)$.
The mean value theorem with $\partial_x \pi_u(x, \theta) \geq \eta$ thus yields
\begin{equation*}
\begin{aligned}
\int_a^{x_0} \big( \pi_u(z, \theta) - H_u(a) \big) m(dz) 
&\geq \int_a^{x_0} \big(\pi_u(z, \theta) - \pi_u(a, \theta) \big) m(dz) 
\geq \eta \int_a^{x_0} (z-a) \, m(dz).
\end{aligned}
\end{equation*}
To see that the last integral diverges, recall from the proof of Lemma~\ref{Lemma: FractionProperties} that $b < 0$ on $\mathcal{I}$, so that $(1/S')'=b\,m'<0$ and $S'$ is increasing. In particular, $S'(x) \leq S'(x_0)$ for all $x \in (\mathfrak{l}, x_0]$, and given that $\mathfrak{l} \in \mathcal{I}$, we get 
$\int_{\mathfrak l}^{x_0}S'(x) dx
\leq S'(x_0) (x_0-\mathfrak l) < \infty$. 
This implies that $S(\mathfrak l+) = S(x_0) - \int_{\mathfrak l}^{x_0} S'(z) dz > -\infty$ %[I do not see this:] and since $\mathfrak{l} \geq 0$ is finite, we have $S(\mathfrak{l}+) > -\infty$
and $S(z) - S(\mathfrak{l}+) \leq \int_{\mathfrak{l}}^z S'(u) du \leq S'(x_0)(z - \mathfrak{l})$ for $z \in (\mathfrak{l}, x_0]$. 
As $\mathfrak{l}$ is natural, and hence not attainable, Feller's test (see, e.g., \citep[pp.~229--230]{KarlinTaylor1981}) gives $\int_{\mathfrak{l}}^{x_0} \big( S(z) - S(\mathfrak{l}+) \big) m(dz) = \infty$, and therefore by above $\int_{\mathfrak{l}}^{x_0} (z - \mathfrak{l}) m(dz) = \infty$. Then monotone convergence yields
$\lim_{a \downarrow \mathfrak{l}} \int_a^{x_0}(z-a) m(dz)=\int_{\mathfrak{l}}^{x_0}(z-\mathfrak{l}) m(dz) = \infty$, 
and consequently, we conclude from \eqref{ga} that $\lim_{a \downarrow \mathfrak{l}}g(a) = \infty$.
 
\smallskip
{\it Step 2.(d): Conclusion.}
By steps 2.(a)--(c), \eqref{newA} and the intermediate value theorem, we conclude that $g$ has exactly one root $a^*(\theta) \in \mathbf{A}$ satisfying $a^*(\theta) < a_{\max}$, so that for $y^*(\theta) := y_{a^*(\theta)} > a^*(\theta)$ from \eqref{ya} and \eqref{a<ya}, $(a^*(\theta), y^*(\theta))$ is the unique solution of \eqref{eq:FOC1} with~$\mathfrak{l}<a<y<\mathfrak{r}$.
\end{proof}

Combining the results in Theorems~\ref{Theorem: Expected long-run average payoff} and \ref{thm:barriers}, we conclude that there exists a unique pair $(a^*,y^*) \equiv (a^*(\theta),y^*(\theta))$ characterising the optimal Poissonian singular control of threshold type which maximises \eqref{Eq:PayoffAandY} and achieves the value $\kappa^*(\theta)$ in \eqref{Eq:PayoffForAandY}.

\subsection{Optimality over the general class of admissible controls} \label{sec:optimality}

In order the upgrade the optimality of the Poissonian singular control at thresholds $(a^*(\theta),y^*(\theta))$ over the general class of controls defined by $\mathcal{A}$ in \eqref{admissible controls}, we first connect $(a^*(\theta),y^*(\theta))$ and $\kappa^*(\theta)$ to an associated free-boundary problem. 
To that end, we define for any given and fixed mean-field parameter $\theta \in \mathbb{R}_{+}$, the function 
\begin{equation} \label{tilpi}
\widetilde{\pi}_{q}(x, \theta):= \pi(x, \theta) + q \lambda x,
\quad q \in \{q_u, q_d\}, 
\quad x \in \mathcal{I},  
\end{equation}
and let the potential function $w: \mathcal{I} \to \mathbb{R}$ be determined, up to an additive constant, by
\begin{equation} \label{Eq:wprime}
w'(x) =
\begin{cases}
(R_\lambda \widetilde{\pi}_{q_u}(\cdot, \theta))'(x) + A_u(\theta) \psi_\lambda'(x), & x \leq a^*(\theta),\\
S'(x) \big(\frac{q_u}{S'(a^*(\theta))} + \int_{a^*(\theta)}^{x} (\kappa^*(\theta) - \pi(z, \theta)) m(dz) \big), & a^*(\theta) < x < y^*(\theta),\\
(R_\lambda \widetilde{\pi}_{q_d} (\cdot,\theta))'(x) + A_d(\theta) \varphi_\lambda'(x), & x \geq y^*(\theta),
\end{cases}
\end{equation}
where
\begin{equation*}
A_u(\theta) := \frac{q_u - (R_\lambda \widetilde{\pi}_{q_u}(\cdot, \theta))'(a^*(\theta))}{\psi_\lambda'(a^*(\theta))} 
\quad \text{and} \quad
A_d(\theta) := \frac{q_d-(R_\lambda \widetilde{\pi}_{q_d}(\cdot, \theta))'(y^*(\theta))}{\varphi_\lambda'(y^*(\theta))}.
\end{equation*}

\begin{lemma} \label{Lemma: FBP}
Recall the thresholds $(a^*(\theta),y^*(\theta))$ from Theorem~\ref{thm:barriers} and their associated value $\kappa^*(\theta)$ from Theorem~\ref{Theorem: Expected long-run average payoff}.
Then, the function $w$ defined via \eqref{Eq:wprime} satisfies
$w \in C^2(\mathcal{I})$ with $w'(a^*(\theta)) = q_u$ and $w'(y^*(\theta)) = q_d$, 
and %the quadruple $(w(x),\kappa^*(\theta),a^*(\theta),y^*(\theta))$ 
solves the free-boundary problem
\begin{equation} \label{eq:FBP}
\mathcal{G} w(x) + \pi(x, \theta) - \kappa^*(\theta) =
\begin{cases}
\lambda\big(w(x) - w(y^*(\theta)) - q_d (x-y^*(\theta))\big), & x \geq y^*(\theta),\\
0, & a^*(\theta) < x < y^*(\theta), \\
\lambda\big(w(x) - w(a^*(\theta)) - q_u (x-a^*(\theta))\big), & x \leq a^*(\theta).
\end{cases}
\end{equation}
\end{lemma}

\begin{proof}
% Recall from Theorems~\ref{Theorem: Expected long-run average payoff} and~\ref{thm:barriers} that
% \begin{equation} \label{Eq:kappa-star}
% \begin{aligned}
% \kappa^*(\theta)
%     &= \frac{1}{m((a^*(\theta),y^*(\theta)))} \bigg[\frac{q_d}{S'(y^*(\theta))} - \frac{q_u}{S'(a^*(\theta))} + \int_{a^*(\theta)}^{y^*(\theta)}\pi(z,\theta)\,m(dz) \bigg] \\
%     &= \frac{(\Psi_\lambda \pi_u(\cdot, \theta))(a^*(\theta))}{(\Psi_\lambda 1)(a^*(\theta))}
%     = \frac{(\Phi_\lambda \pi_d(\cdot, \theta))(y^*(\theta))}{(\Phi_\lambda 1)(y^*(\theta))}.
% \end{aligned}
% \end{equation}
We prove \eqref{eq:FBP} in each region separately and then prove the regularity.

\smallskip
{\it Region $a^*(\theta)<x<y^*(\theta)$.}
By \eqref{Eq:wprime} and $m(dz)=m'(z)\,dz$, it follows that
\[
\frac{d}{dx}\frac{w'(x)}{S'(x)}
=
\big(\kappa^*(\theta)-\pi(x,\theta)\big)m'(x).
\]
Therefore, combining this with \eqref{Eq:GeneratorCanonical}, we get that 
\[
\mathcal{G}w(x) = \frac{1}{m'(x)} \frac{d}{dx}\frac{w'(x)}{S'(x)} %= \kappa^*(\theta)-\pi(x,\theta)
\quad \Rightarrow \quad 
\mathcal{G}w(x)+\pi(x,\theta)-\kappa^*(\theta)=0,
\quad a^*(\theta)<x<y^*(\theta).
\]
Moreover, we have by definition \eqref{Eq:wprime} that $w'(a^*(\theta)+) = q_u$, whereas using the expression of $\kappa^*(\theta)$ from Theorem~\ref{Theorem: Expected long-run average payoff} gives
\[
w'(y^*(\theta)-) 
= S'(y^*(\theta)) \Big( \frac{q_u}{S'(a^*(\theta))} + \int_{a^*(\theta)}^{y^*(\theta)} \big(\kappa^*(\theta)-\pi(z,\theta)\big)m(dz) \Big)
= q_d.
\]
Thus, the derivative of $w$ from the interior region matches the prescribed slopes $q_u$ and $q_d$ at the two free boundaries.

% \smallskip
% {\it Region $a^*(\theta)<x<y^*(\theta)$.}
% By \eqref{Eq:wprime} we have
% \begin{equation*}
% \frac{d}{dx}\frac{w'(x)}{S'(x)} = \big(\kappa^*(\theta)-\pi(x,\theta)\big) m'(x),
% \end{equation*}
% so that $\mathcal{G}w(x) + \pi(x,\theta) - \kappa^*(\theta) = 0$ by \eqref{Eq:GeneratorCanonical}.
% Moreover, $w'(a^*(\theta)+) = q_u$, and the first equality in \eqref{Eq:kappa-star} gives $w'(y^*(\theta)-) = q_d$.

\smallskip
{\it Region $x > y^*(\theta)$.}
Integrating the expression for $w'$ in \eqref{Eq:wprime} from
$y^*(\theta)$ to $x$ and using
$(\mathcal{G}-\lambda)\varphi_\lambda=0$ and
$(\mathcal{G}-\lambda)
R_\lambda\widetilde{\pi}_{q_d}(\cdot,\theta) = - \widetilde{\pi}_{q_d}(\cdot,\theta)$, where
$\widetilde{\pi}_{q_d}(\cdot,\theta)\in\mathcal{L}^1(\mathcal{I})$
by Assumptions~\ref{Ass:X}.(iii) and~\ref{Ass:Payoff}.(iv) and the fact that $\text{id} \in \mathcal{L}^1(\mathcal{I})$, %\bl{[Let's decide if we want to include Lemma~\ref{piL1}][HS: Usually we have just assumed that $\text{id(x) = x}$ belong to $\mathcal{L}^1(\mathcal{I})$ so that these functions are integrable.]}, 
a direct
calculation yields
\begin{equation*}
\begin{aligned}
&\mathcal{G}w(x) + \pi(x,\theta) - \kappa^*(\theta) - \lambda\big(w(x)-w(y^*(\theta)) - q_d \, (x-y^*(\theta)) \big) \\
&\quad = \lambda \, (R_\lambda \widetilde{\pi}_{q_d}(\cdot,\theta))(y^*(\theta)) - \lambda \, q_d \, y^*(\theta) + \lambda \, A_d(\theta) \, \varphi_\lambda(y^*(\theta)) - \kappa^*(\theta), 
\quad x > y^*(\theta).
\end{aligned}
\end{equation*}
This implies that \eqref{eq:FBP} holds true for $x>y^*(\theta)$ if and only if 
\begin{equation} \label{Eq:beta_up1}
\kappa^*(\theta) = \lambda \, (R_\lambda \widetilde{\pi}_{q_d}(\cdot,\theta))(y^*(\theta)) - \lambda \, q_d \, y^*(\theta) + \lambda \, A_d(\theta) \, \varphi_\lambda(y^*(\theta)).
\end{equation} 
In order to verify \eqref{Eq:beta_up1}, we notice that \eqref{Eq:Resolvent1} and \eqref{Eq:GreenKernel1} imply  
\begin{equation*}
(R_\lambda \widetilde{\pi}_{q_d}(\cdot, \theta))(y) \, \varphi_\lambda'(y) - (R_\lambda \widetilde{\pi}_{q_d}(\cdot, \theta))'(y) \, \varphi_\lambda(y) = - S'(y) \, (\Phi_\lambda \widetilde{\pi}_{q_d}(\cdot, \theta))(y), 
\quad y \in \mathcal{I}.
\end{equation*}
Using this together with the definition of $A_d(\theta)$ and the identity $\varphi_\lambda'=-\lambda \, S' \, (\Phi_\lambda 1)$, the right-hand side of \eqref{Eq:beta_up1} equals
\begin{equation*}
\begin{aligned}
&\frac{(\Phi_\lambda \widetilde{\pi}_{q_d}(\cdot, \theta))(y^*(\theta)) - \lambda \, q_d \, y^*(\theta)\,(\Phi_\lambda 1)(y^*(\theta)) - q_d\,\varphi_\lambda(y^*(\theta))/S'(y^*(\theta))}{(\Phi_\lambda 1)(y^*(\theta))}  
= \frac{(\Phi_\lambda \pi_d(\cdot, \theta))(y^*(\theta))}{(\Phi_\lambda 1)(y^*(\theta))},
\end{aligned}
\end{equation*}
where the equality follows from \eqref{Eq:zyInt} and Assumption~\ref{Ass:Payoff}.(iii). 
Then \eqref{Eq:beta_up1} follows from the expression of $\kappa^*(\theta)$ in Theorem~\ref{Theorem: Expected long-run average payoff}.

\smallskip
{\it Region $x < a^*(\theta)$.}
The argument is analogous to the case $x>y^*(\theta)$, using
$(\mathcal{G}-\lambda)\psi_\lambda=0$, the identity \eqref{Eq:zaInt}, the definition
of $A_u(\theta)$, and the expression for $\kappa^*(\theta)$ from
Theorem~\ref{Theorem: Expected long-run average payoff}.

\smallskip
{\it Regularity.}
We first observe that $w'\in C(\mathcal{I})$ thanks to the definitions of $A_u(\theta)$ and $A_d(\theta)$, which imply 
$w'(a^*(\theta)-) %= (R_\lambda\widetilde{\pi}_{q_u}(\cdot,\theta))'(a^*(\theta)) + A_u(\theta)\psi_\lambda'(a^*(\theta)) 
= q_u$ and $w'(y^*(\theta)+) = q_d$. 
Hence $w\in C^1(\mathcal{I})$ and what remains to show is that $w \in C^2(\mathcal{I})$. 
By construction, $w\in C^2(\mathcal{I}\setminus
\{a^*(\theta),y^*(\theta)\})$, 
so it suffices to show that $w''$ is continuous at
$a^*(\theta)$ and $y^*(\theta)$.

By solving \eqref{eq:FBP} for $w''$ and using the continuity of
$w$, $w'$, $b$, $\sigma$, and $\pi(\cdot,\theta)$, together with
\begin{align*}
&w(x) - w(a^*(\theta)) - q_u \, (x-a^*(\theta)) \rightarrow 0 \quad \text{as } x \to a^*(\theta),\\
&w(x) - w(y^*(\theta)) - q_d \, (x-y^*(\theta)) \rightarrow 0 \quad \text{as } x \to y^*(\theta),
\end{align*}
we see that the one-sided limits of $w''$ at
$a^*(\theta)$ and $y^*(\theta)$ coincide. 
Hence
$w\in C^2(\mathcal{I})$, and \eqref{eq:FBP} holds on all of
$\mathcal{I}$.
\end{proof}

We now rewrite the free-boundary problem \eqref{eq:FBP} satisfied by the potential function $w$ in Lemma~\ref{Lemma: FBP} in the form
\begin{align} \label{eq:FBPK} 
\mathcal{G} w(x) + \pi(x, \theta) - \kappa^*(\theta) + \lambda \big( \mathcal{K}_y(x) + \mathcal{K}_a(x) \big) = 0, 
\end{align}
where we define
\begin{align} \label{Kya} 
\begin{split}
\mathcal{K}_y(x) &:= 
\big( (w(y^*(\theta)) - q_d \, y^*(\theta)) - (w(x) - q_d \, x) \big) \mathbbm{1}_{\{x \geq y^*(\theta)\}}, \\
\mathcal{K}_a(x) &:= 
\big( (w(a^*(\theta)) - q_u \, a^*(\theta)) - (w(x) - q_u \, x) \big) \mathbbm{1}_{\{x \leq a^*(\theta)\}}.
\end{split}
\end{align}

\begin{lemma} \label{Lemma: wx_max}
The potential function $w$ in Lemma~\ref{Lemma: FBP} satisfies
\[
w(x) - q_d \, x  \; \text{ is } \begin{cases}
\text{increasing}, & x \in (\mathfrak{l}, y^*(\theta)), \\
\text{decreasing}, & x \in (y^*(\theta), \mathfrak{r}), 
\end{cases}
\quad 
w(x) - q_u \, x  \; \text{ is } \begin{cases}
\text{increasing}, & x \in (\mathfrak{l}, a^*(\theta)), \\
\text{decreasing}, & x \in (a^*(\theta), \mathfrak{r}), 
\end{cases}
\]
so that %the mappings $x \mapsto w(x) - q_d \, x$ and $x \mapsto w(x) - q_u \, x$ attain their unique global maxima at $y^*(\theta)$ and $a^*(\theta)$, so that 
$\mathcal{K}_y$ and $\mathcal{K}_a$ from \eqref{Kya} satisfy
\begin{align*}
\mathcal{K}_y(x) = \sup_{z \leq x} \big\{ (w(z) - q_d \, z) - (w(x) - q_d \, x) \big\}, 
\quad 
\mathcal{K}_a(x) = \sup_{z \geq x} \big\{ (w(z) - q_u \, z) - (w(x) - q_u \, x) \big\}.
\end{align*}
Furthermore, there exists $C > 0$ such that $w$ from Lemma~\ref{Lemma: FBP} satisfies 
\[
|w(x)| \leq  C(1+|x|), \quad x \in \mathcal{I}.
\]
\end{lemma}

\begin{proof}
Recall from Theorem~\ref{thm:barriers}.(i) that $a^*(\theta) < \hat z_u$ and $y^*(\theta) > \hat z_d$. Moreover, \eqref{Eq:Resolvent1}, \eqref{Eq:zyInt}, \eqref{Eq:zaInt} and $\lambda R_\lambda 1=1$ imply that 
\[
R_\lambda(\lambda \, \mathrm{id} - b)(x) = \mathrm{id}(x), 
\quad \text{where} \quad  \mathrm{id}(x):=x. 
\] 
Consequently, the function $\widetilde{\pi}_{q}$ defined by \eqref{tilpi} for $q \in \{q_d, q_u\}$ satisfies 
\begin{equation} \label{Eq:Rpiq}
(R_\lambda \widetilde{\pi}_{q_d}(\cdot,\theta))' - q_d = (R_\lambda \pi_d(\cdot,\theta))'
\quad \text{and} \quad
(R_\lambda \widetilde{\pi}_{q_u}(\cdot,\theta))'-q_u = (R_\lambda \pi_u(\cdot,\theta))'.
\end{equation}

\smallskip
{\it Step 1: $w'(x) < q_d$ for $x \in (y^*(\theta), \mathfrak{r})$.} 
From \eqref{Eq:wprime}, the definition of $A_d(\theta)$ and \eqref{Eq:Rpiq}, we find
\begin{equation*}
\frac{w'(x)-q_d}{\varphi_\lambda '(x)} = L_d(x) - L_d(y^*(\theta)), 
\quad x \geq y^*(\theta), 
\quad \text{where} \quad
L_d(x):= \frac{(R_\lambda \pi_d(\cdot,\theta))'(x)}{\varphi_\lambda '(x)}.
\end{equation*}
Given that $\varphi_\lambda'<0$, it suffices to show that $L_d(\cdot)$ is increasing on $[y^*(\theta), \mathfrak{r})$ to prove that $w'(x) < q_d$ for $x > y^*(\theta)$.
To that end, we differentiate $L_d$, use \eqref{Eq:Resolvent1} and $\varphi_\lambda' = - \lambda \, S' \, (\Phi_\lambda 1)$ to get 
\begin{equation*}
L_d'(x) = \frac{m'(x)}{\lambda(\Phi_\lambda 1)(x)}\bigg(\pi_d(x,\theta) - \frac{(\Phi_\lambda  \pi_d(\cdot,\theta))(x)}{(\Phi_\lambda  1)(x)}\bigg),
\end{equation*}
which is positive on $[y^*(\theta),\mathfrak{r}) \subset (\hat z_d,\mathfrak{r})$ by Lemma~\ref{Lemma: FractionProperties}.(i).

\smallskip
{\it Step 2: $w'(x) > q_d$ for $x \in [a^*(\theta), y^*(\theta))$.}
Define
\begin{equation*}
U(x) := \frac{q_d-w'(x)}{S'(x)} = \int_x^{y^*(\theta)} (\kappa^*(\theta)-\pi_d(z,\theta)) \, m(dz),
\end{equation*}
where the equality follows from \eqref{Eq:wprime}, $w'(y^*(\theta))=q_d$ and \eqref{Eq:DriftToScale}. 
% Consequently, using $y^*(\theta)>\hat z_d$ from Theorem~\ref{thm:barriers}.(i), the expression of $\kappa^*(\theta)$ combined with \eqref{eq:FOC1} in Theorem~\ref{Theorem: Expected long-run average payoff} and Lemma~\ref{Lemma: FractionProperties}.(i), we obtain  
% $$
% U'(x) = (\pi_d(x,\theta)-\kappa^*(\theta)) \, m'(x) \begin{cases} 
% > 0, \quad x \in (a^*(\theta) \vee \bar z_d, y^*(\theta)) \\
% < 0, \quad x \in (a^*(\theta), a^*(\theta) \vee \bar z_d)
% \end{cases}
% $$
% where $\bar z_d \leq \hat z_d$ denotes the unique solution, whenever it
% exists, of $\pi_d(\bar z_d,\theta)=\kappa^*(\theta)$.
Consequently,
\[
U'(x) = \bigl(\pi_d(x,\theta)-\kappa^*(\theta)\bigr) \, m'(x).
\]
Since $\pi_d(\cdot,\theta)$ is strictly increasing on
$(\mathfrak{l},\hat x_d)$ and strictly decreasing on
$(\hat x_d,\mathfrak r)$, while
$\pi_d(y^*(\theta),\theta) > H_d(y^*(\theta)) = \kappa^*(\theta)$ 
thanks to $y^*(\theta)>\hat z_d$ and Lemma~\ref{Lemma: FractionProperties}.(i), we have $\pi_d(x,\theta) > \kappa^*(\theta)$ for all $x \in [\hat x_d, y^*(\theta)]$. Hence the function
$\pi_d(\cdot,\theta)-\kappa^*(\theta)$ changes sign at most once on
$(a^*(\theta),y^*(\theta))$ and, if it does, the sign changes from
negative to positive.
Hence $U$ is either always increasing or first decreasing and then increasing on
$[a^*(\theta),y^*(\theta)]$.
% Since $\pi_d(\cdot,\theta)$ is unimodal and $\pi_d(y^*(\theta), \theta) > H_d(y^*(\theta)) = \kappa^*(\theta)$ by Lemma~\ref{Lemma: FractionProperties} and \eqref{Eq:kappa-star}, $U'$ changes sign at most once on $[a^*(\theta),y^*(\theta)]$, from negative to positive. Hence $U(x) < \max\{U(a^*(\theta)), U(y^*(\theta))\}$ for $x \in (a^*(\theta), y^*(\theta))$. 
Combining this with  
$U(a^*(\theta)) = (q_d - q_u) / {S'(a^*(\theta))} < 0 = U(y^*(\theta))$, implies that $U < 0$ and thus $w'(x) > q_d$ for all $x \in [a^*(\theta), y^*(\theta))$. 

\smallskip
{\it Step 3: $w'(x) > q_u > q_d$ for $x \in (\mathfrak{l}, a^*(\theta))$}. 
As above, from \eqref{Eq:wprime}, the definition of $A_u(\theta)$ and \eqref{Eq:Rpiq}, we find
\begin{equation*}
\frac{w'(x)-q_u}{\psi_\lambda '(x)} = L_u(x) - L_u(a^*(\theta)), 
\quad x \leq a^*(\theta), 
\quad \text{where} \quad
L_u(x):= \frac{(R_\lambda \pi_u(\cdot,\theta))'(x)}{\psi_\lambda '(x)}.
\end{equation*}
Given that $\psi_\lambda'>0$, it suffices to show that $L_u(\cdot)$ is decreasing on $(\mathfrak{l},a^*(\theta)]$ to prove that $w'(x) > q_u$ for $x < a^*(\theta)$.
To that end, we differentiate $L_u$, use \eqref{Eq:Resolvent1} and $\psi_\lambda' = \lambda \, S' \, (\Psi_\lambda 1)$ to get 
\begin{equation*}
L_u'(x) = \frac{m'(x)}{\lambda(\Psi_\lambda 1)(x)}\bigg(\frac{(\Psi_\lambda  \pi_u(\cdot,\theta))(x)}{(\Psi_\lambda  1)(x)} - \pi_u(x,\theta)\bigg),
\end{equation*}
which is negative on $(\mathfrak{l}, a^*(\theta)] \subset (\mathfrak{l},\hat z_u)$ by Lemma~\ref{Lemma: FractionProperties}.(ii). 

\smallskip
{\it Conclusion.}
Combining Steps 1--3, we get that $w(x) - q_d \, x$ is increasing for $x \in (\mathfrak{l}, y^*(\theta))$ and decreasing for $x \in (y^*(\theta), \mathfrak{r})$, which implies that its unique global maximum occurs at $y^*(\theta)$ (using similar arguments we also obtain the claim for $x \mapsto w(x) - q_u x$ and its unique global maximum at $a^*(\theta)$) and  
\begin{equation} \label{w<}
w(x) \leq w(y^*(\theta)) + q_d \, (x-y^*(\theta)), \quad x \in \mathcal{I}.
\end{equation}
Conversely, integrating the expression for $w'$ in \eqref{Eq:wprime} and using $R_\lambda\widetilde{\pi}_{q}(\cdot,\theta) \geq 0$ for $q \in \{q_u,q_d\}$, which is implied by $\widetilde{\pi}_{q}(\cdot,\theta) \geq 0$ for $x \in \mathcal{I}$ and $\lambda>0$ by its definition \eqref{tilpi}, together with the boundedness of $\psi_\lambda$ on $(\mathfrak{l}, a^*(\theta)]$ and of $\varphi_\lambda$ on $[y^*(\theta), \mathfrak{r})$, shows that $w$ is bounded from below on these two intervals. 
In addition, given that $w$ is bounded also on $[a^*(\theta), y^*(\theta)]$ by continuity, we conclude that it is bounded from below on the whole $\mathcal{I}$.
Combining this with \eqref{w<} yields that 
$|w(x)| \leq C(1+|x|)$ for all $x \in \mathcal{I}$, and completes the proof.
\end{proof}

The Poissonian singular control of threshold type $\nu(a^*(\theta),y^*(\theta))$ belongs to $\mathcal{A}$, and by Theorem~\ref{Theorem: Expected long-run average payoff} its long-run average payoff equals $\kappa^*(\theta)$. Hence $\kappa^*(\theta) \leq \beta(x;\theta)$, where $\beta(x;\theta)$ is defined in \eqref{Eq:BetaLambda}. The following verification theorem establishes the reverse inequality, so that $\beta(x;\theta) = \kappa^*(\theta)$ and the threshold policy is optimal.

\begin{theorem}[Verification] \label{thm:verification}
Recall the thresholds $(a^*(\theta),y^*(\theta))$ from Theorem~\ref{thm:barriers} and their associated value $\kappa^*(\theta)$ defined by \eqref{Eq:PayoffForAandY}, whose expression is given by Theorem~\ref{Theorem: Expected long-run average payoff}, % with the thresholds from Theorem~\ref{thm:barriers}, 
and $\beta(x;\theta)$ defined by \eqref{Eq:BetaLambda}. 
Then, $\kappa^*(\theta) \geq \beta(x;\theta)$.
\end{theorem}

\begin{proof} 
Recall $w$ defined via \eqref{Eq:wprime} with the thresholds $(a^*(\theta),y^*(\theta))$ from Theorem~\ref{thm:barriers}, and satisfying the properties in Lemma \ref{Lemma: FBP}. 
For any fixed $T>0$ and $\nu \in \mathcal{A}$ from \eqref{admissible controls}, we define the stopping time $\overline{\tau}_{\rho} := \inf \{ t \geq 0 \mid X^\nu_t \not\in (\mathfrak{l} + \frac1\rho, \mathfrak{r} - \frac1\rho) \}$, and apply the generalised It\^o's lemma to $w(X^{\nu})$, to get
\begin{equation*}  
\begin{aligned}
w(X_{T \wedge \overline{\tau}_{\rho}}^{\nu}) & = w(x) + \int_0^{T \wedge \overline{\tau}_{\rho}} \mathcal{G}w(X_s^{\nu})ds+ \int_0^{T \wedge \overline{\tau}_{\rho}}  \sigma(X^{\nu}_s)w'(X_s^{\nu})dW_s + \sum_{s \leq T \wedge \overline{\tau}_{\rho}} 
\hspace{-1mm}(w(X^{\nu}_s)-w(X^{\nu}_{s-})), 
\end{aligned}
\end{equation*}
% Reorganising the above gives \begin{equation}  \label{Eq:RHSneeded}  \begin{aligned} & w(X_{T \wedge \overline{\tau}_{\rho}}^{\nu}) - w(x) - \int_0^{T \wedge \overline{\tau}_{\rho}}  \sigma(X^{\nu}_s)w'(X_s^{\nu})dW_s \\ & = \int_0^{T \wedge \overline{\tau}_{\rho}} \mathcal{G}w(X_s^{\nu})ds + \sum_{s \leq T \wedge \overline{\tau}_{\rho}} [w(X^{\nu}_s)-w(X^{\nu}_{s-})]. \end{aligned} \end{equation}
which in light of \eqref{eq:FBPK} takes the form
\begin{equation} \label{Eq:RHSneeded}  
\begin{aligned}
w(X_{T \wedge \overline{\tau}_{\rho}}^{\nu}) & = w(x) - \int_0^{T \wedge \overline{\tau}_{\rho}} \big( \pi(X_s^{\nu}, \theta) - \kappa^*(\theta) + \lambda \big( \mathcal{K}_y(X_s^{\nu}) + \mathcal{K}_a(X_s^{\nu}) \big) \big) ds \\ 
&\quad +\int_0^{T \wedge \overline{\tau}_{\rho}}  \sigma(X^{\nu}_s)w'(X_s^{\nu})dW_s + \sum_{s \leq T \wedge \overline{\tau}_{\rho}} 
\hspace{-1mm}(w(X^{\nu}_s)-w(X^{\nu}_{s-})), 
\end{aligned}
\end{equation}

Because the control can jump only at Poisson jump times, we conclude from Lemma~\ref{Lemma: wx_max} that 
\begin{align*}
w(X^{\nu}_s)-w(X^{\nu}_{s-}) + q_d \, \Delta \nu_s^- -q_u \,  \Delta \nu_s^+ 
\leq \big( \mathcal{K}_y(X^{\nu}_{s-})+\mathcal{K}_a(X^{\nu}_{s-}) \big) \Delta N_s
\end{align*}
Using this in \eqref{Eq:RHSneeded} thus yields     \begin{equation} \label{Eq:w2verification}
\begin{aligned}
&w(X_{T \wedge \overline{\tau}_{\rho}}^{\nu}) 
\leq w(x) + \kappa^*(\theta) ({T \wedge \overline{\tau}_{\rho}}) - \int_0^{T \wedge \overline{\tau}_{\rho}} \big(\pi(X_s^{\nu},\theta) \, ds + q_d \, d\nu_s^- - q_u d \, \nu_s^+ \big) \\
&+\int_0^{T \wedge \overline{\tau}_{\rho}} \hspace{-3mm}  \sigma(X^{\nu}_s)w'(X_s^{\nu})dW_s 
+ \int_0^{T \wedge \overline{\tau}_{\rho}}  \hspace{-3mm}\mathcal{K}_y(X_{s-}^{\nu}) \big(dN_s - \lambda ds \big) 
+ \int_0^{T \wedge \overline{\tau}_{\rho}}  \hspace{-3mm}\mathcal{K}_a(X_{s-}^{\nu})  \big(dN_s - \lambda ds \big), 
\end{aligned}
\end{equation}
Then, define the compensated Poisson process $\tilde{N}_t := N_t - \lambda t$ and notice that 
\begin{align*}
&M_t := \int_0^{t \wedge \overline{\tau}_{\rho}} \sigma(X^{\nu}_s)w'(X_s^{\nu})dW_s 
\quad \text{and} \quad 
Z_t := \int_0^{t \wedge \overline{\tau}_{\rho}} \big(\mathcal{K}_y(X_{{s-}}^{\nu})+ \mathcal{K}_a(X_{{s-}}^{\nu}) \big) d \tilde{N}_s,
\end{align*}
are martingales since their integrands are locally bounded. 
Hence, taking expectations in \eqref{Eq:w2verification} and using that $\mathbb{E}_x[Z_{T \wedge \overline{\tau}_{\rho}} + M_{T \wedge \overline{\tau}_{\rho}}] =0$, we get
\begin{equation*}  
\begin{aligned}
&\mathbb{E}_x \Big[w(X_{T \wedge \overline{\tau}_{\rho}}^{\nu}) + \int_0^{T \wedge \overline{\tau}_{\rho}} \big(\pi(X_s^{\nu}, \theta)ds + q_d d\nu_s^- - q_u d \nu_s^+ \big) \Big]  \leq w(x) + \kappa^*(\theta) T. \\
\end{aligned}
\end{equation*}
Taking first the limit as $\rho \to \infty$, then dividing by $T$, and subsequently taking the limit as $T \to \infty$, yields 
\begin{equation} \label{limw} 
\begin{aligned} 
\liminf_{T\to \infty} \frac{1}{T} \, \mathbb{E}_x \Big[ w(X_T^{\nu}) + \int_0^{T}  (\pi(X_s^{\nu},\theta)ds + q_d d\nu_s^- - q_u d \nu_s^+) \Big] \leq \kappa^*(\theta).
\end{aligned}
\end{equation}
Lastly, we know from Lemma~\ref{Lemma: wx_max} that there exists $C>0$ such that $|w(x)| \leq  C(1+|x|)$ for $x \in \mathcal{I}$, thus for every $\nu \in \mathcal{A}$, given that $X^\nu \in \mathcal{I}\subseteq\mathbb{R}_+$, we get
$$
\limsup_{T \to \infty} \frac{1}{T} \, \mathbb{E}_x[|w(X_T^{\nu})|] 
\leq \limsup_{T \to \infty} \Big(\frac{C}{T} + \frac{C}{T} \, \mathbb{E}_x[X_T^{\nu}] \Big)
= C \, \limsup_{T \to \infty} \frac{1}{T} \, \mathbb{E}_x[X_T^{\nu}] 
= 0.
$$
Combining this with \eqref{limw} and taking the supremum over all $\nu \in \mathcal{A}$ on the left-hand side, we obtain that $\beta(x;\theta) \leq \kappa^*(\theta)$ and complete the proof.
\end{proof}

\section{Existence and Uniqueness of the Constrained Ergodic MFG Equilibrium} \label{sec:MFG}

In this section, we finally prove the existence and uniqueness of the mean-field equilibrium as in Definition~\ref{def:Equilibrium}. 
To achieve this, we first prove that the optimal boundaries characterising the solution to the constrained ergodic singular control problem in Section \ref{sec:barriers} are continuous in the mean-field parameter $\theta$ (see Proposition \ref{Prop:continuity}).
Then, we prove the continuity of the stationary distribution of the optimally controlled process with respect to the mean-field parameter $\theta$ as well. 
Using these two results, we finally establish the existence of an equilibrium via a fixed-point theorem and its uniqueness via suitable monotonicity arguments (see Theorem~\ref{thm:ExistenceMFGE}).

We begin our analysis by recalling that the potential function $w(\cdot;\theta)$, together with the optimal long-run average $\kappa^*(\theta)$ and the optimal boundaries $a^*(\theta)$ and $y(\theta)$, solve the free-boundary problem in Lemma \ref{Lemma: FBP}, which also establishes that $w(\cdot;\theta) \in C^2(\mathcal{I})$. 
Hence, setting $u(x;\theta) = w'(x;\theta)$, so that $u(\cdot;\theta)\in C^1(\mathcal{I})$ and $u_{xx}(\cdot;\theta)\in L^\infty_{\mathrm{loc}}(\mathcal{I})$, standard differentiation of the ordinary differential equation (ODE) \eqref{eq:FBP} yields the following ODE system:
\begin{equation} \label{eq:u-ode}
\widehat{\mathcal G}u(x;\theta)+b'(x)u(x;\theta)+\pi_x(x,\theta)
= \begin{cases}
\lambda\big(u(x;\theta)-q_d\big), & x \geq y^*(\theta),\\
0, & a^*(\theta)<x<y^*(\theta),\\
\lambda\big(u(x;\theta)-q_u\big), & x \leq a^*(\theta),
\end{cases}
\end{equation}
where 
$\widehat{\mathcal G}\phi(x)
:= \frac{1}{2}\sigma^2(x)\phi''(x)
+\big(b(x)+\sigma(x)\sigma'(x)\big)\phi'(x)$.
Moreover, $u$ satisfies the inequalities
\begin{equation} \label{eq:u-ineq}
\begin{cases}
u(x;\theta)\geq q_u, & x \leq a^*(\theta),\\
q_d<u(x;\theta)<q_u, & a^*(\theta)<x<y^*(\theta),\\
u(x;\theta)\leq q_d, & x \geq y^*(\theta).
\end{cases}
\end{equation}

%Together with \eqref{eq:u-ode}, these imply the differential inequalities
%\begin{equation} \label{eq:u-diff-ineq}
%\begin{cases}
%\widehat{\mathcal G}u(x;\theta)+b'(x)u(x;\theta)+\pi_x(x,\theta)\geq0,
%& x \leq a^*(\theta),\\
%\widehat{\mathcal G}u(x;\theta)+b'(x)u(x;\theta)+\pi_x(x,\theta)=0,
%& a^*(\theta)<x<y^*(\theta),\\
%\widehat{\mathcal G}u(x;\theta)+b'(x)u(x;\theta)+\pi_x(x,\theta)\leq0,
%& x \geq y^*(\theta).
%\end{cases}
%\end{equation}

Let then $(\widehat{X}_{t})_{t\geq 0}$ be the unique strong solution to the SDE
\begin{equation} \label{eq:hatX}
d\widehat{X}_{t} = \big( b(\widehat{X}_{t}) + \sigma (\widehat{X}_{t}) \sigma'(\widehat{X}_{t}) \big) \, dt + \sigma(\widehat{X}_{t}) d\widehat{W}_{t}, 
\quad \widehat{X}_{0} = x \in \mathcal{I},
\end{equation}
for an $\mathbb{F}$-adapted Brownian motion $\widehat{W}$, 
which follows from Assumption~\ref{Ass:X} and the fact that $\widehat{X}$ is a Girsanov transformation of $X$ with kernel $\sigma\sigma'$.
This further implies that the boundaries $\mathfrak{l}$ and $\mathfrak{r}$ are natural for $\widehat{X}$.
We then define 
%\bl{[HS: Do we need that the boundaries are natural to $\hat{X}$ as well? This doesn't automatically transfer. Also do we need that $\sigma \sigma'$ grows linearly?] GIORGIO: The fact that the boundaries of $\widehat{X}$ are natural follow from the fact that $\widehat{X}$ is just $X$ under a new probability measure (via Girsanov) with kernel $\sigma\sigma'$ (for that we need to assume that $\E[e^{\int_0^T \frac{1}{2}\sigma\sigma'({X}_t)^2 dt}]<\infty$ for all $T>0$ (Novikov condition). Now $\sigma\sigma'$ is locally locally Lip, hence either we say that the solution has a solution up to a possible explosion time or we assume that $\sigma\sigma'$ has sublinear growth. Take your favourite one! [HS: I'm fine with either.]}
\begin{equation}
    \tau^*(\theta):=\inf\{T^{(n)}, n\geq 1: \widehat{X}_{T^{(n)}} \leq a^*(\theta)\}, \quad \eta^*(\theta):=\inf\{T^{(n)}, n\geq 1: \widehat{X}_{T^{(n)}} \geq y^*(\theta)\}, 
\end{equation}
which belong to the set of admissible Poisson stopping times
\[
\mathcal{T}_p^\lambda
:=\big\{\tau:\tau \text{ is an } \mathbb F\text{-stopping time and }
\tau\in\{T^{(n)}:n\in\mathbb N\}\cup\{\infty\}\ \mathbb P\text{-a.s.}\big\}.
\]
Applying It\^o's formula (up to a standard approximation with smooth functions) to the processes
\begin{align*}
&t\mapsto e^{\int_0^{t\wedge \tau^*(\theta) \wedge \eta}  b'(\widehat X_s)ds}u(\widehat X_{t\wedge \tau^*(\theta) \wedge \eta};\theta), \quad t\geq 0, \quad \eta \in \mathcal{T}_p^\lambda, \\ 
&t\mapsto e^{\int_0^{t\wedge \tau \wedge \eta^*(\theta)}  b'(\widehat X_s)ds}u(\widehat X_{t\wedge \tau \wedge \eta^*(\theta)};\theta), \quad t\geq 0, \quad \tau \in \mathcal{T}_p^\lambda
\end{align*}
and using \eqref{eq:u-ode}--\eqref{eq:u-ineq},
it is easy to show that the function $u$ is the value of the auxiliary constrained Dynkin game
\begin{align} \label{eq:u}
\begin{split}
u(x;\theta) = \inf_{\tau \in \mathcal{T}_p^\lambda} \sup_{\eta \in \mathcal{T}_p^\lambda} \widehat{\mathbb{E}}_{x}\Big[ &\int_{0}^{\tau \wedge \eta} e^{\int_{0}^{t} b'(\widehat{X}_{s})ds} \pi_{x}(\widehat{X}_{t}, \theta) \, dt \\ 
&+ q_u e^{\int_{0}^{\tau} b'(\widehat{X}_{s}) ds} \boldsymbol{1}_{\{\tau<\eta\}}
+ q_d \, e^{\int_{0}^{\eta} b'(\widehat{X}_{s}) ds} \boldsymbol{1}_{\{\eta<\tau\}} \Big],
\end{split}
\end{align}
where $\widehat{\mathbb{E}}_{x}[\,\cdot\,] := \mathbb{E}^{\widehat{\mathbb{P}}}[\,\cdot\, |\widehat{X}_{0}=x]$ for $\widehat{\mathbb{P}}_{x}(\,\cdot\,) := \mathbb{P} (\,\cdot\, | \widehat{X}_{0}=x)$.

% \begin{proposition} \label{prop:u}
% Consider the value $u$ of the Dynkin game in \eqref{eq:u}. Then, 
% \begin{enumerate}[\rm (i)]
% \item $x \mapsto u(x;\theta)$ is decreasing for every $\theta \in \mathbb{R}_{+}$;
% \item $\theta \mapsto u(x;\theta)$ is decreasing for every $x \in \mathcal{I}$.
% \end{enumerate}
% \end{proposition}
% \begin{proof}

% Fix $\theta\in\mathbb{R}_{+}$. By Assumption~\ref{Ass:Payoff}.(ii), the map
% $x\mapsto\pi_x(x,\theta)$ is non-increasing. Moreover, the stochastic flow
% $x\mapsto \widehat X_t^x$ associated with \eqref{eq:hatX} is increasing a.s.,
% because the coefficients are $C^1$ and pathwise uniqueness holds. Hence, for
% $x_1\leq x_2$ and for every pair $(\tau,\eta)\in
% \mathcal T_p^\lambda\times\mathcal T_p^\lambda$, the running payoff in
% \eqref{eq:u} started from $x_1$ is not smaller than the corresponding running
% payoff started from $x_2$, while the terminal rewards $q_u$ and $q_d$ are
% state-independent \bl{[HS: The discounting however is not state-independent. Does this matter?]}. 

% Similarly, by Assumption~\ref{Ass:Payoff}.(v), the map
% $\theta\mapsto \pi_x(x,\theta)$ is decreasing for every fixed
% $x\in\mathcal{I}$. Since the remaining terms in \eqref{eq:u} do not depend on
% $\theta$, the same comparison of payoffs gives
% $u(x;\theta_1)\geq u(x;\theta_2)$ whenever $\theta_1\leq\theta_2$.
% \end{proof}

\begin{lemma} \label{prop:u}
Recall the value $u$ of the Dynkin game in \eqref{eq:u}. Then, for every $x \in \mathcal{I}$, we have that $\theta \mapsto u(x;\theta)$ is decreasing.
\end{lemma}
\begin{proof}
Fix $x \in \mathbb{R}_{+}$ and $0 \leq \theta_1 \leq \theta_2$. 
Then by Assumption~\ref{Ass:Payoff}.(v), the map $\theta\mapsto \pi_x(x,\theta)$ is decreasing. Hence, pathwise for all $t \geq 0$, $\pi_x(\widehat X_t^x,\theta_1) \geq \pi_x(\widehat X_t^x,\theta_2)$. 
Since the remaining terms in the expectation in \eqref{eq:u} do not depend on $\theta$, taking first the supremum over $\eta$ and then the infimum over $\tau$ preserves the ordering, and therefore $u(x;\theta_1) \geq u(x_2;\theta_2)$.
\end{proof}

% {\color{red}{GIORGIO: Good catch about the monotonicity, but actually we did not need that proposition. The fact is that we have already identified $w'$ as $u$ and therefore the stopping and continuation regions defined in the control problem coincide with those of he Dynkin game. HS: I deleted the first part of the Proposition since only the second part was used. I also downgraded it to a Lemma}}

Using that $w'=u$, it thus follows from Lemma~\ref{Lemma: FBP}, Lemma~\ref{Lemma: wx_max} and Theorem~\ref{thm:verification} that the continuation region of \eqref{eq:u} takes the form 
\begin{align*}
\mathcal{C}^{\theta} 
= \{ x \in \mathcal{I} \, | \, q_d < u(x;\theta) < q_u \} 
&= \{ x \in \mathcal{I} \, | \, q_d < w'(x;\theta) < q_u \} %\\& 
= \{ x \in \mathcal{I} \, | \, a^*(\theta) < x < y^*(\theta) \}
\end{align*}
and the stopping region of \eqref{eq:u} is the union of the following disconnected sets
\begin{align*}
\mathcal{S}^{\theta}_{\inf} 
&= \{x\in\mathcal{I} \; | \; u(x;\theta) \geq q_u \} 
= \{x\in\mathcal{I} \; | \; w'(x;\theta) \geq q_u \} 
= \{x\in\mathcal{I} \; | \; x \leq a^*(\theta) \} 
\\
\mathcal{S}^{\theta}_{\sup} 
&= \{ x \in \mathcal{I} \; | \; u(x;\theta) \leq q_d \}
= \{ x \in \mathcal{I} \; | \; w'(x;\theta) \leq q_d \}
= \{x\in\mathcal{I} \; | \; x \geq y^*(\theta) \}, 
\end{align*}
where the boundaries $a^*$ and $y^*$ are obtained in Theorem \ref{thm:barriers}.
Given that the stopping regions are non-empty by Theorem~\ref{thm:barriers}, we can write for any $\theta \in \mathbb{R}_{+}$, 
\begin{equation} \label{eq:ay}
a^*(\theta) = \max \{ x \in \mathcal{I} \;|\; u(x;\theta) \geq q_u\} 
\quad \text{and} \quad
y^*(\theta) = \min \{x \in \mathcal{I} \;|\; u(x;\theta) \leq q_d \}.
\end{equation}

\begin{lemma} \label{lem:boundary-decreasing}
The maps $\theta \mapsto a^*(\theta)$ and $\theta \mapsto y^*(\theta)$ are non-increasing.
\end{lemma}
\begin{proof}
Take arbitrary $\theta_{1},\theta_{2} \in \mathbb{R}_{+}$ such that $\theta_{1} \leq \theta_{2}$. Then, using that the map $\theta \mapsto u(x;\theta)$ is decreasing for any $x \in \mathcal{I}$ from Proposition \ref{prop:u}, we obtain
\[
\{x \in \mathcal{I} \;|\; u(x;\theta_{1}) \leq q_d \} \subseteq \{x \in \mathcal{I} \;|\; u(x;\theta_{2}) \leq q_d \},
\]
which, by \eqref{eq:ay}, implies that
\[
y(\theta_2) = \min \{x \in \mathcal{I} \;|\; u(x;\theta_{2}) \leq q_d \} \leq \min \{x \in \mathcal{I} \;|\; u(x;\theta_{1}) \leq q_d \} 
= y(\theta_1),
\] 
and proves the desired monotonicity. 

The monotonicity of $\theta \mapsto a^*(\theta)$ can be obtained using similar arguments.
\end{proof}

\begin{proposition}
\label{Prop:continuity}
Recall the pair of boundaries $(a^*,y^*)$ obtained in Theorem \ref{thm:barriers}. Then, the map $\theta \mapsto (a^*(\theta), y^*(\theta))$ is continuous.
\end{proposition}
\begin{proof}
For fixed $\theta\in\mathbb{R}_{+}$, the threhsolds $a^*=a^*(\theta)$,
$y^*=y^*(\theta)$ from Theorem \ref{thm:barriers}, are characterised in Theorem~\ref{Theorem: Expected long-run average payoff}
as the solution of the nonlinear
algebraic system \eqref{eq:FOC1}. 
Equivalently, $(a^*(\theta),y^*(\theta))$ is the unique admissible pair satisfying
\[
\mathcal H(a^*(\theta),y^*(\theta),\theta)=0,
\qquad
\mathcal H:=(\mathcal H_1,\mathcal H_2).
\]
where we define admissible pairs $(A,B)$ satisfying $A<B$, and 
\[
\begin{aligned}
\mathcal H_1(A,B,\theta)
&:=
\frac{1}{m((A,B))}
\bigg[\int_A^B \pi_d(z,\theta)m(dz)+\frac{q_d-q_u}{S'(A)}\bigg]
-\frac{(\Phi_\lambda  \pi_d(\cdot,\theta))(B)}{(\Phi_\lambda  1)(B)},\\
\mathcal H_2(A,B,\theta)
&:=
\frac{1}{m((A,B))}
\bigg[\int_A^B \pi_u(z,\theta)m(dz)+\frac{q_d-q_u}{S'(B)}\bigg]
-\frac{(\Psi_\lambda  \pi_u(\cdot,\theta))(A)}{(\Psi_\lambda  1)(A)}.
\end{aligned}
\]
The regularity assumptions on $\pi$, $b$, $\sigma$, $\pi_d$, $\pi_u$ in Assumptions~\ref{Ass:X} and \ref{Ass:Payoff}, together with the smoothness of the scale function and speed density, imply that $\mathcal H$ is continuous in $(A,B,\theta)$ and continuously differentiable in the variables $(A,B)$. 
Moreover, the construction in Theorem~\ref{thm:barriers}
gives the required nondegeneracy at the optimal pair, namely
\[
\det D_{(A,B)}\mathcal H(a^*(\theta),y^*(\theta),\theta)\neq0.
\]
By the parameter-dependent implicit function theorem, there exist a
neighbourhood $U$ of $\theta$ and a continuous map
$\vartheta\mapsto(\widetilde a(\vartheta),\widetilde y(\vartheta))$ such that
\[
\mathcal H(\widetilde a(\vartheta),\widetilde y(\vartheta),\vartheta)=0,
\quad \vartheta\in U 
\quad \text{and} \quad 
(\widetilde a(\theta),\widetilde y(\theta))=(a^*(\theta),y^*(\theta)).
\]
Shrinking $U$ if necessary, the pair $(\widetilde a(\vartheta),\widetilde y(\vartheta))$ remains admissible for every $\vartheta\in U$. 
For every such $\vartheta$, Theorem~\ref{thm:barriers} gives a unique admissible solution of \eqref{eq:FOC1}. 
The local branch produced by the implicit function theorem is one admissible solution of this system, and hence it must coincide with the optimal pair 
$(\widetilde a(\vartheta),\widetilde y(\vartheta))
=(a^*(\vartheta), y^*(\vartheta))$.
Therefore $(a^*(\cdot), y^*(\cdot))$ is continuous at the arbitrary point
$\theta$, and the claim follows.
\end{proof}

\begin{proposition}
\label{prop:mu-continuous}
Recall the stationary distribution $\mu^{(a^*(\theta),y^*(\theta))}$ of $X^{\nu(a^*(\theta),y^*(\theta))}$ obtained in Theorem~\ref{Theorem:TwoSidedDensity} for $\nu(a^*(\theta),y^*(\theta))$ defined by \eqref{nuay}--\eqref{nudownup} with $a=a^*(\theta)$ and $y=y^*(\theta)$ from Theorem~\ref{thm:barriers}.
Let $p^\theta:=d\mu^{(a^*(\theta),y^*(\theta))}/dm$ denote its density with respect to the speed measure. 
Then, for every $x\in\mathcal{I}$, the map $\theta \mapsto p^{\theta}(x)$ is continuous.
\end{proposition}

\begin{proof}
Recall that the density with respect to the speed measure is given in Theorem~\ref{Theorem:TwoSidedDensity} by
\[
p^\theta(x)
=\frac{1}{Z(\theta)}
\begin{cases}
\frac{\varphi_\lambda(x)}{\varphi_\lambda (y^*(\theta))}, & x\geq y^*(\theta),\\
1, & a^*(\theta)<x<y^*(\theta),\\
\frac{\psi_\lambda(x)}{\psi_\lambda (a^*(\theta))}, & x\leq a^*(\theta),
\end{cases}
\]
with
\[
Z(\theta):=
m((a^*(\theta),y^*(\theta)))
+\frac{(\Psi_\lambda 1)(a^*(\theta))}{\psi_\lambda (a^*(\theta))}
+\frac{(\Phi_\lambda 1)(y^*(\theta))}{\varphi_\lambda (y^*(\theta))} > 0.
\] 
Given that $\theta_n\to\theta$ implies $a^*(\theta_n) \to a^*(\theta)$ and $y^*(\theta_n) \to y^*(\theta)$ as $n \to \infty$ by Proposition~\ref{Prop:continuity},  that $\psi_\lambda$, $\varphi_\lambda$, $\Phi_\lambda 1$, $\Psi_\lambda 1$ are continuous, and the continuity of the speed measure of
intervals, we conclude that each of the three terms defining $Z(\theta_n)$ converges to the corresponding term in $Z(\theta)$. 
Therefore $\theta \mapsto Z(\theta)$ is continuous. 
Then, for each fixed $x \in \mathcal{I}$, the corresponding expression within each of the three regions is continuous in $\theta$. 
Moreover, the pieces match continuously at $x=a^*(\theta)$ since
$\psi_\lambda(x)/\psi_\lambda(x)=1$, and at $x=y^*(\theta)$ since $\varphi_\lambda(x)/\varphi_\lambda(x)=1$. 
Thus, we finally  conclude that $p^{\theta_n}(x)\to p^\theta(x)$ as $n \to \infty$, for every $x\in\mathcal{I}$.
\end{proof}

Define the operator $\mathcal{T}:\mathbb{R}_{+} \to \mathbb{R}_{+}$, as
\begin{equation}
\label{def:T}
\mathcal{T} \theta := F\bigg(\int_{\mathcal{I}} f(x) \, \mu^{(a^*(\theta),y^*(\theta))}(dx) \bigg) 
=: F\Big( \langle f, \mu^{(a^*(\theta),y^*(\theta))} \rangle \Big),\quad \theta \in \mathbb{R}_{+}.
\end{equation}
% where $\langle f, \mu^{(a^*(\theta),y^*(\theta))} \rangle := \int_{\mathcal{I}} f(x) \, \mu^{(a^*(\theta),y^*(\theta))}(dx)$.

\begin{theorem}
\label{thm:ExistenceMFGE}
There exists a unique MFG equilibrium $\theta^{*}$, i.e.~a unique $\theta^{*} \in \mathbb{R}_{+}$ such that 
$$
\mathcal{T} \theta^{*} = \theta^{*}.
$$
\end{theorem}
\begin{proof}
We split the proof into the following five steps.

\vspace{1mm}
{\it Step 1.~Lower bound for the fixed-point operator $\mathcal{T}$.} 
Using the definition \eqref{eq:u} of $u(\cdot;\theta)$, together with Assumptions \ref{Ass:Payoff}.(v)--(vi), which imply that $\pi_{x}(x,\theta) \geq k(x)$ for all $x \in \mathcal{I}$ and any $\theta \in \mathbb{R}_{+}$, 
we can conclude that, for any $\theta \in \mathbb{R}_{+}$, we have  
\begin{align*} %\label{aux Dynkin game}
\begin{split}
u(x;\theta) \geq \underline{u}(x) 
:= \inf_{\tau \in \mathcal{T}_p^\lambda} \sup_{\eta \in \mathcal{T}_p^\lambda} \widehat{\mathbb{E}}_{x} &\bigg[ \int_{0}^{\tau\wedge \eta} e^{\int_{0}^{t} b'(\widehat{X}_{s}) ds} \kappa(\widehat{X}_{s}) ds \\
&+ q_u \, e^{\int_{0}^{\tau} b'(\widehat{X}_{s}) ds} \, \boldsymbol{1}_{\{ \tau < \eta \}} 
+ q_d \, e^{\int_{0}^{\eta} b'(\widehat{X}_{s}) ds} \, \boldsymbol{1}_{\{\eta<\tau\}} \bigg], 
\quad x \in \mathcal{I}.
\end{split}
\end{align*}
The function $\underline{u}$ is thus the value of a constrained Dynkin game, where players are restricted to make decisions only at the jump times of the Poisson process $N$. 
It can be shown by standard techniques that this game admits a saddle point given by the Poisson jump times
$$
\underline{\tau}^{*}(x) := \inf\{ T^{(i)} > 0 \,|\, \widehat{X}_{T^{(i)}} \in \underline{\mathcal{S}}_{\inf} \} 
\quad \text{and} \quad \underline{\eta}^{*}(x) := \inf\{T^{(i)} > 0 \,|\, \widehat{X}_{T^{(i)}} \in \underline{\mathcal{S}}_{\sup} \}
$$
where $\underline{\mathcal{S}}_{\inf}$ and $\underline{\mathcal{S}}_{\sup}$ are the stopping regions of each player, defined by 
$$ 
\underline{\mathcal{S}}_{\inf} := \{x \in \mathcal{I} \,|\, \underline{u}(x) \geq q_u\} 
\quad \text{and} \quad \underline{\mathcal{S}}_{\sup} := \{ x \in \mathcal{I} \,|\, \underline{u}(x) \leq q_d \}.
$$
In particular, thanks to Assumption \ref{Ass:Payoff}.(vi), standard analysis for Poisson constrained Dynkin games for one-dimensional diffusions (see e.g. \citep{LiangSun2019, LempaSaarinen2023}) imply existence of free boundaries $\mathfrak{l} < \underline{a} < \underline{b} < \mathfrak{r}$ such that: 
$$
\underline{a} := \sup \{ x \in \mathcal{I} \,|\, \underline{u}(x) \geq q_u\}
\quad \text{and} \quad 
\underline{y} := \inf\{ x \in \mathcal{I} \,|\, \underline{u}(x) \leq q_d \}
$$
Then, for any $\theta \in \mathbb{R}_{+}$, we have
\begin{equation} \label{ula}
a^*(\theta) 
= \sup\{x\in\mathcal{I} \,|\, u(x;\theta) \geq q_u \} 
\geq \sup \{x \in \mathcal{I} \,|\, \underline{u}(x) \geq q_u \} 
= \underline{a}.
\end{equation}
Similarly, we can conclude that 
\begin{equation} \label{uly}
y^*(\theta) 
\geq \underline{y}, \quad \text{for all } \theta \in \mathbb{R}_{+}.
\end{equation}

In view of \eqref{ula}--\eqref{uly}, and by the comparison principle for the order-preserving stochastic flow of the one-dimensional diffusion, combined with the monotonicity of the projection map $x\mapsto \min\{\max\{x,a\},y\}$ in the threshold pair $(a,y)$, we have 
$$
X^{\nu(a^*(\theta), y^*(\theta))} 
\equiv X^{x, \nu(a^*(\theta), y^*(\theta))} 
\geq X^{x, \nu(\underline{a}, \underline{y})}, \quad \mathbb{P}\text{-a.s.},
$$
where we use the additional superscript $x$ to stress the starting point of the process (whenever needed).
Since $f$ is increasing by Assumption \ref{Ass:Ff}, we obtain
\begin{equation*}
\mathbb{E}_{x}\big[ f\big(X^{\nu(\underline{a}, \underline{y})}_{t} \big) \big] 
\leq \mathbb{E}_{x}\big[ f\big( X^{\nu(a^*(\theta), y^*(\theta))}_{t} \big) \big], \quad \text{for any } t\geq 0, 
\end{equation*}
which implies, using also that $F$ is increasing, that 
\begin{equation*}
\underline{\theta} := F \Big(\mathbb{E}_{x}\big[ f\big(X^{\nu(\underline{a}, \underline{y})}_{\infty} \big) \big] \Big) 
\leq F \Big(\mathbb{E}_{x}\big[ f\big( X^{\nu(a^*(\theta), y^*(\theta))}_{\infty} \big) \big] \Big)
= F \Big( \int_{\mathcal{I}} f(x) \, \mu^{(a^*(\theta), y^*(\theta))}(dx) \Big)
= \mathcal{T} \theta,
\end{equation*}
where $\underline{\theta}$ is well-defined as 
\begin{equation} \label{ultheta}
\underline{\theta} 
:= F \Big( \mathbb{E}_{x} \big[ f\big(X^{\nu(\underline{a}, \underline{y})}_{\infty} \big) \big] \Big) 
= F \Big( \int_{\mathcal{I}} f(x) \, \underline{\mu}(dx) \Big) ,
\end{equation}
since there exists a unique stationary distribution $\underline{\mu}$ of the process $(X^{x, \nu(\underline{a}, \underline{y})}_{t})_{t\geq 0}$, thanks to the analysis in Section \ref{sec:density}. 

\vspace{1mm}
{\it Step 2.~Upper bound for the fixed-point operator $\mathcal{T}$.} 
Consider now any $\theta \geq \underline{\theta}$, where $\underline{\theta}$ is defined by \eqref{ultheta} in Step 1. 
Thanks to the monotonicity of $a^*$ and $y^*$ in Lemma \ref{lem:boundary-decreasing}, we have
\begin{equation} \label{olay}
a^*(\theta) \leq a^*(\underline{\theta}) 
=: \overline{a} 
\quad \text{and} \quad
y^*(\theta) \leq y^*(\underline{\theta}) 
=: \overline{y}, 
\quad \text{for all } \theta \geq \underline\theta.
\end{equation}
Analogous considerations as in Step 1 then yield that
\begin{align} \label{oltheta}
\begin{split}
\mathcal{T} \theta
= F \Big( \int_{\mathcal{I}} f(x) \, \mu^{(a^*(\theta), y^*(\theta))}(dx) \Big)
&= F \Big( \mathbb{E}_{x} \big[ f\big(  X^{\nu(a^*(\theta), y^*(\theta))}_{\infty} \big) \big] \Big) \\
&\leq F \Big( \mathbb{E}_{x} \big[ f\big( X^{\nu(\overline{a}, \overline{y})}_{\infty} \big) \big] \Big) 
= F \Big( \int_{\mathcal{I}} f(x) \, \overline{\mu}(dx) \Big)
=: \overline{\theta} 
\end{split}
\end{align} 
where $\overline{\mu}$ is the unique stationary distribution of the process $(X^{x, \nu(\overline{a}, \overline{y})}_{t})_{t\geq 0}$, which exists according to the analysis in Section \ref{sec:density}.    
    
\vspace{1mm}
{\it Step 3.~Invariant set for the fixed-point operator $\mathcal{T}$.}   
Combining the bounds obtained in Steps 1 and 2, namely \eqref{ultheta} and \eqref{oltheta}, we conclude that $\mathcal{T}$ maps the convex, compact set
\begin{equation} \label{eq:compactset}
K := [\underline{\theta},\overline{\theta}] \subseteq \mathbb{R}_{+}
\quad \text{into itself}.
\end{equation}
    
\vspace{1mm}
{\it Step 4.~Continuity of the fixed-point operator $\mathcal{T}$.} 
We begin by recalling the bounds of $a^*(\theta)$ and $y^*(\theta)$ obtained in \eqref{ula}--\eqref{uly} and \eqref{olay} for all $\theta \in K$.
Using these bounds we observe that $a^*(\theta)$ and $y^*(\theta)$ belong to a compact $\widehat{K}$, defined in terms of the compact set $K$ by
\begin{equation} \label{compactset} 
a^*(\theta), y^*(\theta) \in \big[ \, \underline{a}, \overline{y} \, \big]
=: \widehat{K}, \quad \text{for all } \theta \in K.
\end{equation}
Combining this with the fact that for every fixed $x\in\mathcal{I}$ the speed-density $p^\theta(x)$ is
continuous in $\theta$ (due to Proposition~\ref{prop:mu-continuous}), and its explicit expression in Theorem~\ref{Theorem:TwoSidedDensity}, as well as the growth condition on $f$ in Assumption~\ref{Ass:Ff}, we obtain a common
integrable bound for $x\mapsto f(x) \, p^\theta(x)$ with respect to $m(dx)$, $\theta\in K$. Hence,
\[
\theta \mapsto  \int_{\mathcal{I}} f(x) \, \mu^{(a^*(\theta), y^*(\theta))}(dx) \quad \text{is continuous}.
\]
Since $F$ is continuous, we conclude that
\[
\mathcal T:K\to K
\quad\text{is continuous on the convex compact set }K.
\]

\vspace{1mm}
{\it Step 5.~Existence and uniqueness of the fixed-point.} 
Given the compact and convex set $K$ from Step 3 and the continuity of the map $\theta \mapsto \mathcal{T} \theta$ from Step 4, we can apply the Schauder--Tychonoff fixed-point theorem (see, e.g.~\citep[Cor.~17.56]{AliprantisBorder1999}), to conclude that there exists $\theta^{*} \in K$ such that $\mathcal{T}\theta^{*}=\theta^{*}$.

For uniqueness, we assume (aiming for a contradiction) that there exists another fixed point $\theta'\in K$ of $\mathcal{T}$, such that $\theta^{*} \neq \theta'$. 
Without loss of generality we can assume that $\theta^{*}<\theta'$ (the opposite inequality can be treated similarly). 
Given the monotonicity of $a^*$ and $y^*$ from Lemma~\ref{lem:boundary-decreasing}, we have
$$
a^*(\theta^{*}) \geq a^*(\theta') 
\quad \text{and} \quad 
y^*(\theta^{*}) \geq y^*(\theta'), 
$$
so that, by following the same arguments as in Step 1, we obtain that $\theta'=\mathcal{T}\theta'\leq \mathcal{T}\theta^{*}=\theta^{*}$, which leads to a contradiction.
\end{proof}

\section{Case Study}
\label{sec:illustration}

In this section we illustrate our theoretical results in the case study of Remark~\ref{Rem:Benchmark}. 
In particular, the uncontrolled state process evolves as a geometric Brownian motion, so that  
\begin{equation}
\label{eq:gbm}
dX^\nu_t = - \rho X_t \, dt + \sigma  X_t\,dW_t + d\nu_t, 
, \quad t \ge 0, \quad X^\nu_0 = x \in \mathcal{I}, \quad \rho,\sigma>0 ,
\end{equation}
and the running reward and mean-field primitives are given by \eqref{case}.
% We select the payoff and the mean-field primitives to be
% \begin{equation*} \pi(x,\theta)=A_\theta x^\delta, \quad A_\theta:=\theta^{-(1+\delta)}+k^*,\quad F(m)=m^{1/\delta}, \quad f(x)=x^\delta, \end{equation*}
% where $k^*>0$ and $\delta \in(0,1)$. Hence, \begin{equation*} \pi_i(x,\theta)=A_\theta x^\delta-q_i\rho x, \qquad i\in\{u,d\}. \end{equation*}
It is straightforward to verify that Assumptions \ref{Ass:X}, \ref{Ass:Payoff}, and \ref{Ass:Ff} are satisfied for this specification.

For the uncontrolled diffusion $X^0$ from  \eqref{eq:gbm}, we have 
\begin{equation*}
S'(x)=x^\eta, 
\quad m'(x)=\frac{2}{\sigma^2}x^{-2-\eta}, 
\quad x>0, 
\quad \eta := \frac{2\rho}{\sigma^2},
\end{equation*} 
so that the increasing and decreasing fundamental solutions of $(\mathcal G-\lambda)u=0$ can be written as $\psi_\lambda(x)=x^{r_+(\lambda)}$ and $\varphi_\lambda(x)=x^{r_-(\lambda)}$, respectively, where
\begin{equation*}
r_\pm(\lambda) = \frac{1+\eta \pm \sqrt{(1+\eta)^2+8\lambda/\sigma^2}}{2}.
\end{equation*}

\subsection{Solution to the Constrained Ergodic Singular Control Problem in \eqref{Eq:BetaLambda}}
\label{sec:Egcontrol}

For fixed $\theta \in \mathbb{R}_+$, we know from Theorem~\ref{Theorem: Expected long-run average payoff} that the optimal thresholds $(a^*,y^*)$ solve the system
\begin{equation}
\label{eq:gbm-poisson-system}
K(a^*,y^*; \theta)
= H_u(a^*; \theta, \lambda) 
= H_d(y^*; \theta, \lambda),
\end{equation}
where we define via \eqref{eq:FOC1} and using $d_+(\lambda) := r_+(\lambda)-1 - \eta$ and $d_-(\lambda):=1+\eta-r_-(\lambda)$,
\begin{align*}
K(a,y;\theta) &:=\frac{q_d \, y^{-\eta} - q_u \, a^{-\eta} + \frac{2}{\sigma^2} \, (\theta^{-(1+\delta)}+k^*) \, \frac{y^{\delta-1-\eta}-a^{\delta-1-\eta}}{\delta-1-\eta}}{\frac{2}{\sigma^2}\frac{a^{-1-\eta}-y^{-1-\eta}}{1+\eta}} , \\
H_u(a;\theta,\lambda) 
&:= d_+(\lambda) \bigg(\frac{(\theta^{-(1+\delta)}+k^*) \, a^\delta}{d_+(\lambda) + \delta} - \frac{q_u \, \rho \, a}{d_+(\lambda)+1}\bigg), \\
H_d(y;\theta,\lambda) 
&:= d_-(\lambda) \bigg(\frac{(\theta^{-(1+\delta)}+k^*) \,y^\delta}{d_-(\lambda)-\delta}-\frac{q_d \, \rho \, y}{d_-(\lambda)-1}\bigg).
\end{align*}
Equations \eqref{eq:gbm-poisson-system} are highly non-linear but we know from Theorem~\ref{thm:barriers} that this system admits a unique solution $(a^*,y^*)$. 
We can therefore easily obtain this solution numerically; see Figure~\ref{fig:gbm-threshold-convergence} for the optimal thresholds $(a^*(\theta), y^*(\theta)) \equiv (a^*_\lambda(\theta), y^*_\lambda(\theta))$ for different values of $\lambda>0$.
We further observe that when $\lambda \to \infty$, 
we get $\lim_{\lambda \to \infty} H_u(a;\theta,\lambda) = \pi_u(a;\theta)$ and $\lim_{\lambda \to \infty} H_d(a;\theta,\lambda) = \pi_d(a;\theta)$, so that the system of equations solved by the optimal thresholds $(a^*_\infty(\theta), y^*_\infty(\theta))$ takes the form 
\begin{equation} \label{eq:gbm-reflected-system}     
K(a_\infty,y_\infty;\theta) 
= \pi_u(a_\infty;\theta)
= \pi_d(y_\infty;\theta).
\end{equation}
Notably, the system \eqref{eq:gbm-reflected-system} corresponds to the one resulting from the classical singular control problem when interventions are allowed continuously, where the thresholds $(a_\infty,y_\infty)$ are reflecting boundaries; see \citep{Alvarez2018}. 
The optimal reflecting thresholds $(a^*_\infty(\theta), y^*_\infty(\theta))$ are also illustrated in Figure~\ref{fig:gbm-threshold-convergence}.
% Furthermore, we observe that the optimal thresholds $(a^*_\lambda(\theta), y^*_\lambda(\theta))$ converge to the optimal reflecting thresholds $(a^*_\infty(\theta), y^*_\infty(\theta))$ as $\lambda \to \infty$.

\begin{figure}[ht]
\centering
\includegraphics[width=0.60\textwidth]{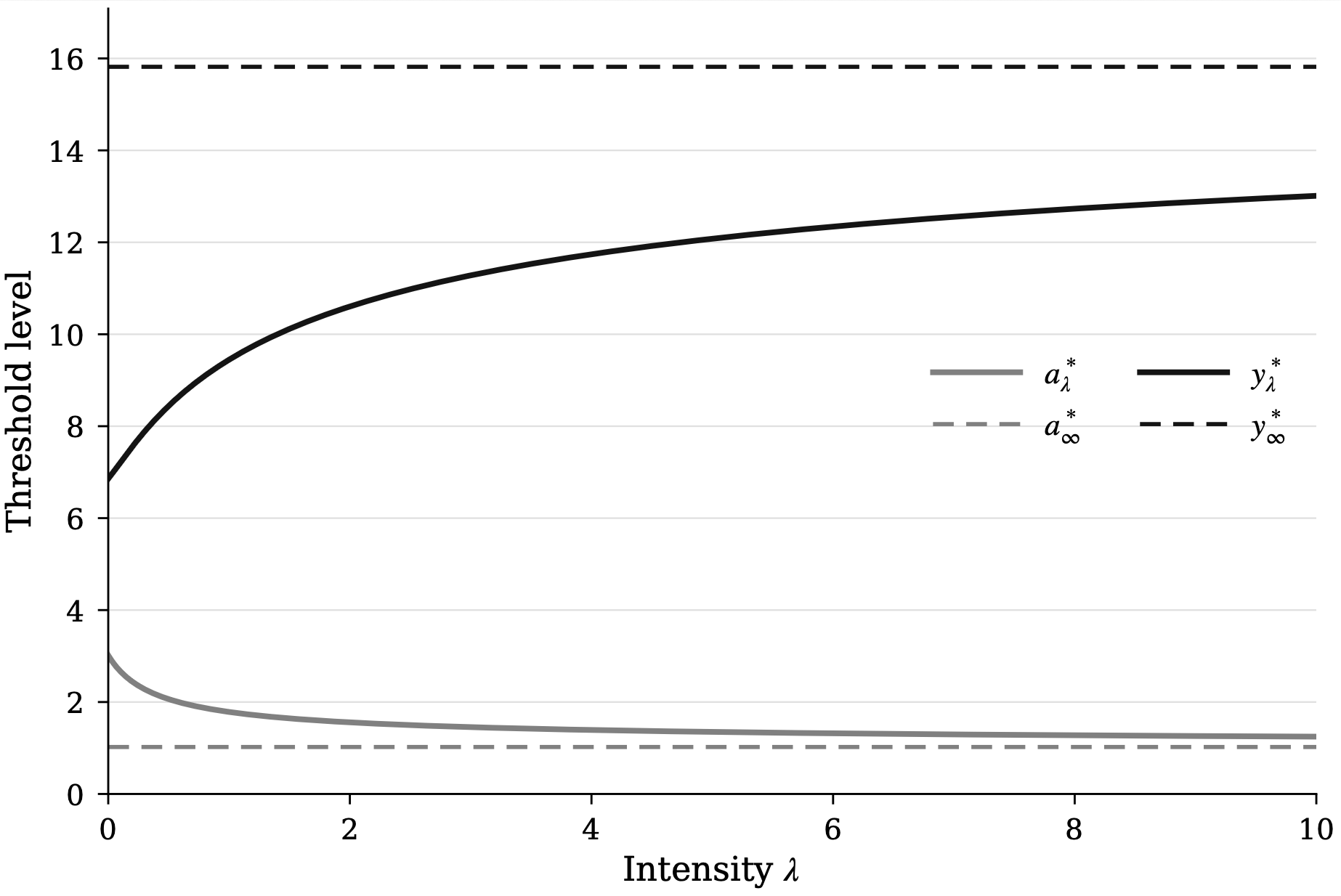}
\caption{
Numerical illustration of optimal thresholds $(a^*_\lambda(\theta), y^*_\lambda(\theta))$ solving \eqref{eq:gbm-poisson-system} for $\lambda>0$ and the optimal reflecting boundaries $(a^*_\infty(\theta), y^*_\infty(\theta))$ in the classical singular control problem with continuous interventions. 
Model parameters:
$\delta= 0.65, \rho= 0.30, \sigma= 0.90, k^* = 0.30, q_u= 2.00, q_d= 1.50, \theta = 1$.
}
\label{fig:gbm-threshold-convergence}
\end{figure}

\subsection{The Constrained Ergodic MFG Equilibrium}

Given the optimal thresholds $a^*_\lambda(\theta)$ and $y^*_\lambda(\theta))$ solving \eqref{eq:gbm-poisson-system} for any $\lambda>0$ in Section~\ref{sec:Egcontrol}, we find that the equilibrium of the constrained ergodic MFG in Definition~\ref{def:Equilibrium} is given via Theorem~\ref{Theorem:TwoSidedDensity} by the fixed point $\theta^*$ of 
\begin{equation}
\label{eq:gbm-poisson-equilibrium} 
\theta 
= \Big( \int_0^\infty x^\delta \, \mu^{(a^*_\lambda(\theta),y^*_\lambda(\theta))}_{\lambda}(dx) \Big)^{1/\delta}
= \bigg( \frac{M(a^*_\lambda(\theta),y^*_\lambda(\theta);\lambda)}{Z(a^*_\lambda(\theta),y^*_\lambda(\theta);\lambda)}\bigg)^{1/\delta},
\end{equation}
where we define 
\begin{align*}
M(a,y;\lambda)
&:=\frac{a^{\delta-1-\eta}}{d_+(\lambda)+\delta}
+ \frac{a^{\delta-1-\eta}-y^{\delta-1-\eta}}{1+\eta-\delta}
+ \frac{y^{\delta-1-\eta}}{d_-(\lambda)-\delta}, \\
Z(a,y;\lambda)
&:= \frac{a^{-1-\eta}}{d_+(\lambda)}
+ \frac{a^{-1-\eta}-y^{-1-\eta}}{1+\eta}
+ \frac{y^{-1-\eta}}{d_-(\lambda)}.
\end{align*}
Given that there exists a unique fixed point of \eqref{eq:gbm-poisson-equilibrium} thanks to Theorem~\ref{thm:ExistenceMFGE}, we can easily obtain $\theta^* \equiv \theta^*_\lambda$ numerically; see Figure~\ref{fig:gbm-equilibrium-convergence} for different values of $\lambda>0$.
We can also obtain the equilibrium $\theta^*_\infty$ when $\lambda \to \infty$ in \eqref{eq:gbm-poisson-equilibrium}, where the latter takes the form 
\[
\theta = \bigg( \lim_{\lambda \to \infty} \frac{M(a^*_\infty(\theta),y^*_\infty(\theta);\lambda)}{Z(a^*_\infty(\theta),y^*_\infty(\theta);\lambda)}\bigg)^{1/\delta}
= \bigg( \frac{1+\eta}{1+\eta-\delta}\bigg)^{1/\delta} \bigg( \frac{(a^*_\infty(\theta))^{\delta-1-\eta}-(y^*_\infty(\theta))^{\delta-1-\eta}}{(a^*_\infty(\theta))^{-1-\eta}-(y^*_\infty(\theta))^{-1-\eta}} \bigg)^{1/\delta},
\]
which is also illustrated in Figure~\ref{fig:gbm-equilibrium-convergence}.

\begin{figure}[ht]
\centering
\includegraphics[width=0.60\textwidth]{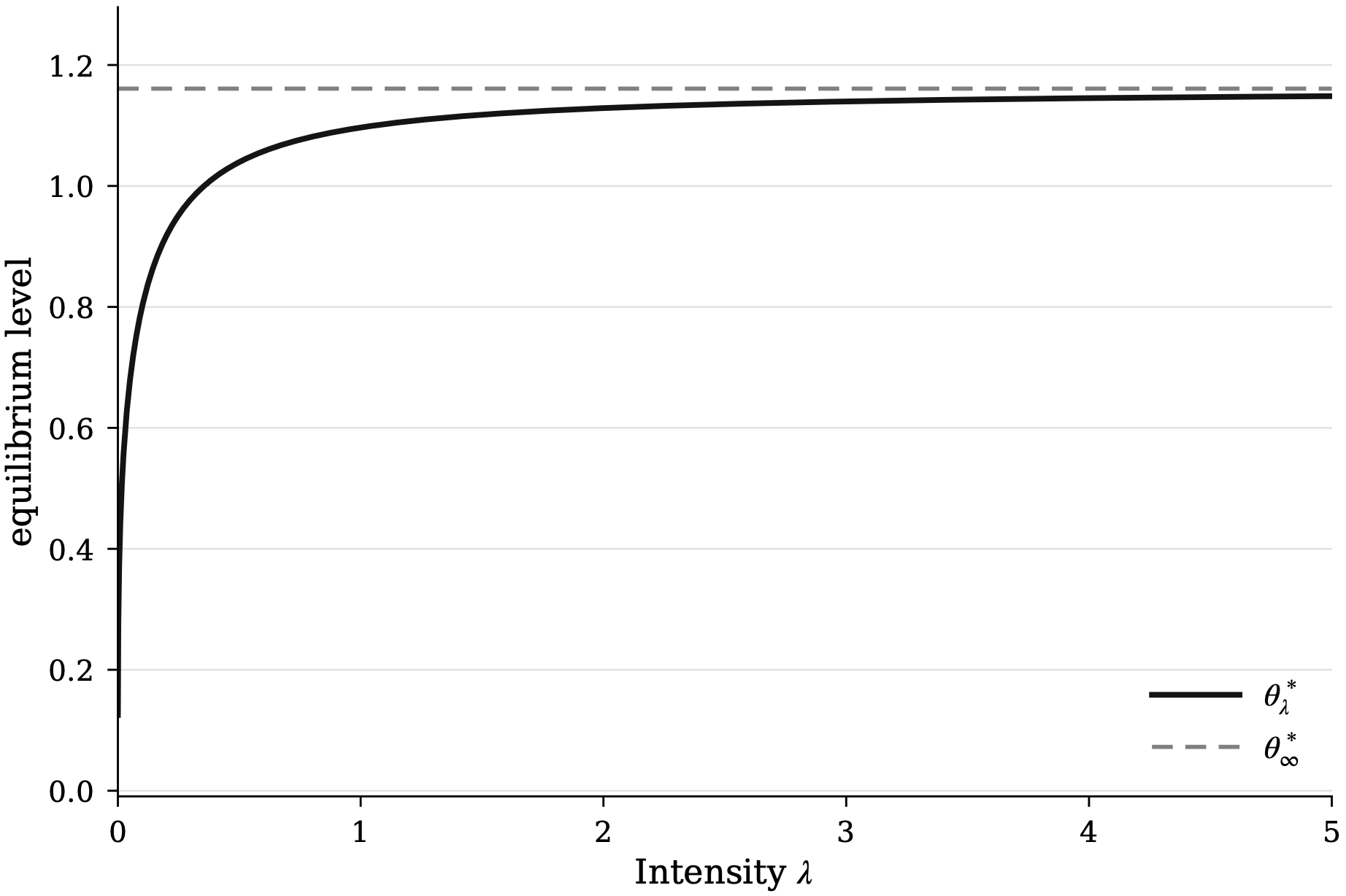}
\caption{Numerical illustration of the MFG equilibrium $\theta^*_\lambda$ solving \eqref{eq:gbm-poisson-equilibrium} for $\lambda>0$.  
Model parameters:
$\delta=0.65, \rho=0.30, \sigma=0.90, k^*=0.30, q_u=2.00, q_d=1.50$.}
\label{fig:gbm-equilibrium-convergence}
\end{figure}

\section{Stationary distributions for one-sided Poissonian singularly controlled processes of threshold type}
\label{sec:density1}

In this section, we aim at deriving the stationary distribution for one-sided Poissonian singularly controlled processes of threshold type $X^\downarrow$ from \eqref{Xdown} and $X^\uparrow$ from \eqref{Xup}.

We firstly derive expressions for the hitting times in \eqref{ulambda}, which simplify considerably when the underlying process $X$ satisfies suitable ergodicity conditions, as shown in the following result.

\begin{corollary}
\label{Corr:u_lambda}
Suppose that Assumptions~\ref{Ass:X}.(i),(ii),(iv),(v) hold true, and consider $\lambda > 0$ and $y \in \mathcal{I}$. 
Recall the first Poisson-sampled hitting times $T_y$ and $T^y$ from \eqref{Ty} and their expected values $u_{\lambda}^{\downarrow}(\cdot;y)$ and $u_{\lambda}^{\uparrow}(\cdot;y)$, respectively, from \eqref{ulambda}. 
If $S(\mathfrak{l+}) = -\infty$ and $m((\mathfrak{l},y)) < \infty$, then
\begin{align*}   
u_{\lambda}^{\downarrow}(x;y) 
%= \mathbb{E}_x[T_y] 
&= \int_\mathfrak{l}^y G_0^y(x,z) \, m(dz) + \frac{m((\mathfrak{l},y))}{ \lambda(\Phi_\lambda  1)(y)} \, \varphi_\lambda(y) + \frac{1}{\lambda}, \quad x \leq y, 
\end{align*}
and if $S(\mathfrak{r}-) = \infty$ and $m((y,\mathfrak{r})) < \infty$, then 
\begin{align*}
u_{\lambda}^{\uparrow}(x;y)
%= \mathbb{E}_x[T^y] 
&= \int_y^\mathfrak{r} G_0^y(x,z) \, m(dz) + \frac{m((y,\mathfrak{r}))}{ \lambda(\Psi_\lambda  1)(y)} \, \psi_\lambda(y) + \frac{1}{\lambda}, \quad x \geq y.
\end{align*}
%where $G_0^y, m, \varphi_\lambda , \Phi_\lambda , \psi_\lambda , \Psi_\lambda $ are defined in Section~\ref{sec:Diffusions}.
\end{corollary}

\begin{proof}
We only prove the expression of $u_{\lambda}^{\downarrow}$, since the expression of $u_{\lambda}^{\uparrow}$ can be obtained using completely analogous arguments.

We firstly notice that $u_{\lambda}^{\downarrow}(x;y) = V^{\downarrow}(x;y,\mathcal{I})$ defined by \eqref{OT}. 
Thus, by the strong Markov property and the memoryless property of the exponential distribution, we obtain that 
$$
u_{\lambda}^{\downarrow}(x;y) 
= u(x;y,\mathcal{I}) + h(x;y) \, V^{\downarrow}(y;y,\mathcal{I})
= \int_\mathfrak{l}^y G_0^y(x,z) \, m(dz) + h(x;y) \, V^{\downarrow}(y;y,\mathcal{I})
, \quad x \leq y,
$$
where the latter equality follows from the Green kernel representation of $u(x;y,\mathcal{I})=u(x;y)$ for $x \leq y$, and $h(x,y)$ is defined by \eqref{Eq:hitprob}.
Then, substituting $x=y$ in the expression for $V^{\downarrow}$ obtained in Theorem~\ref{Theorem: PoissonOccupation}, and integrating over $\mathcal{I}$, yields (cf.~\eqref{PsiPhi})
\[
V^{\downarrow}(y; y, \mathcal{I}) 
= \frac{\varphi_\lambda(y) \int_{\mathfrak{l}}^y h(z;y) \,  m(dz) + \int_y^{\mathfrak{r}} \varphi_\lambda(z) \, m(dz)}{\lambda(\Phi_\lambda 1)(y) + \frac{\varphi_{\lambda}(y)}{S(y)-S(\mathfrak{l}+)}} 
= \frac{\varphi_\lambda(y) \int_{\mathfrak{l}}^y h(z;y) \,  m(dz) + (\Phi_\lambda 1)(y)}{\lambda(\Phi_\lambda 1)(y) + \frac{\varphi_{\lambda}(y)}{S(y)-S(\mathfrak{l}+)}}. 
\]
Finally, using the limit $S(\mathfrak{l+}) = -\infty$, which further implies via \eqref{Eq:hitprob} that $h(z;y)=1$ for all $z\leq y$, and the finiteness of $m((\mathfrak{l},y))$, we obtain the desired expression for $u_{\lambda}^{\downarrow}(x;y)$ and complete the proof.
%%%%%%%%%%%%%%%%%%%%%%%%%%%%%%
% Taking $A = \mathcal{I}$ in \eqref{Eq:V_up_markov} gives $V^{\downarrow}(x; y, \mathcal{I}) 
% = u(x; y, \mathcal{I})+h(x;y)V^{\downarrow}(y; y, \mathcal{I})$ for $x \leq y$, and $u(x; y, \mathcal{I})  = \int_\mathfrak{l}^y G_0^y(x,z) m(dz)$. Then integrating the expressions \eqref{Eq:V_down_theorem} in Theorem~\ref{Theorem: PoissonOccupation} over $\mathcal{I}$ and using the limit $S(\mathfrak{l+}) = -\infty$ gives the result for $u_{\lambda}^{\downarrow}(x;y)$. The upward case is analogous.
\end{proof}

The preceding Corollary~\ref{Corr:u_lambda}, together with the explicit densities derived in Theorem~\ref{Theorem: PoissonOccupation} for the occupation measures in \eqref{OT}, are used in the following result to derive explicit representations for the stationary measures associated with the one-sided downward and upward Poissonian singularly controlled processes $X^\downarrow$ and $X^\uparrow$ at the threshold $y \in \mathcal{I}$.

\begin{theorem}
\label{Thm: PoissonOneSidedDistributions}
Suppose that Assumptions~\ref{Ass:X}.(i),(ii),(iv),(v) hold true, and consider $\lambda > 0$ and $y \in \mathcal{I}$. 
If $S(\mathfrak{l+}) = -\infty$ and $m((\mathfrak{l},y)) < \infty$, the stationary distribution $\mu^{\downarrow}_{\lambda}$ of $X^\downarrow$ is given by 
\begin{equation*}
\mu^{\downarrow}_{\lambda}(dx) = \frac{m(dx) }{ (\Phi_\lambda 1)(y) + \varphi_\lambda(y)m((\mathfrak{l},y))} \cdot
\begin{cases}
\varphi_\lambda(y), & x \leq y, \\
\varphi_\lambda(x), & x \geq y,
\end{cases}
\end{equation*}
whereas if $S(\mathfrak{r-}) = \infty$ and $m((y, \mathfrak{r}))<\infty$, the stationary distribution $\mu^{\uparrow}_{\lambda}$ of $X^\uparrow$ is given by 
\begin{equation*}
\mu^{\uparrow}_{\lambda}(dx) = \frac{m(dx) }{ (\Psi_\lambda 1)(y) + \psi_\lambda(y)m((y,\mathfrak{r}))} \cdot
\begin{cases}
\psi_\lambda(x), & x \leq y, \\
\psi_\lambda(y), & x \geq y.
\end{cases}
\end{equation*}
\end{theorem} 

\begin{proof}
We only prove the expression of $\mu^{\downarrow}_{\lambda}$, since the expression of $\mu^{\uparrow}_{\lambda}$ is obtained similarly.

Consider the path of $X^\downarrow_t$ being built from cycles $C_0, C_1, C_2, C_3, \ldots$, such that the initial segment $C_0$ starts from $x$ and ends after the process is controlled down to $y$ at a Poisson jump time, whereas all subsequent ones are instead selected to start from $y$ and end after the process is controlled back down to $y$ at a Poisson jump time. 
In other words, since, by the strong Markov property and the memoryless property of the exponential distribution, the cycle lengths $C_1, C_2, \ldots$ are i.i.d. and each $C_i$ has the same distribution as the first Poisson-sampled hitting time $T_y$ of $X^\downarrow_t$ started from $y$, we have (cf.~\eqref{ulambda}, Corollary~\ref{Corr:u_lambda})
\begin{equation*}
C_i \overset{d}{=} T_y 
\quad \text{and} \quad 
\mathbb{E}_y[C_i] 
= \mathbb{E}_y[T_y]
= u^{\downarrow}_{\lambda}(y;y), 
\quad i =1,2, \ldots.
\end{equation*}
Following the arguments from renewal theory \citep{Harrison1985, KarlinTaylor1981}, we can express the stationary measure, independently of the initial state $x$, as
\begin{equation}  \label{mudown}
\mu^{\downarrow}_{\lambda}(A) 
= \frac{\mathbb{E}_y \big[ \int_0^{C_i} \mathbf{1}(X^\downarrow_t \in A) \, dt \big]}{\mathbb{E}_y[C_i]}
= \frac{V^{\downarrow}(y;y,A)}{u^{\downarrow}_{\lambda}(y;y)}
, \quad \text{for a Borel set } A \subset \mathcal{I}.
\end{equation}
We then evaluate the expression of $V^{\downarrow}$ in Theorem~\ref{Theorem: PoissonOccupation} at $x=y$ and integrate over  of $A \subset \mathcal{I}$ to obtain thanks to $S(\mathfrak{l}+)=-\infty$, that
\begin{equation} \label{Eq:Vrecurrent}
V^{\downarrow}(y; y, dx) = \frac{m(dx)}{\lambda(\Phi_{\lambda}1)(y)} \cdot
\begin{cases}
\varphi_{\lambda}(y), & x \leq y,\\
\varphi_{\lambda}(x), & x \geq y.
\end{cases}
\end{equation}
Finally, dividing by $u^{\downarrow}_{\lambda}(y;y)$ from Corollary~\ref{Corr:u_lambda} gives $\mu^{\downarrow}_{\lambda}(dx)$ and completes the proof.
\end{proof}

\section*{Acknowledgements}

The work of Harto Saarinen was supported by the OP Research Foundation Grant 20240114. 

\bibliographystyle{plainnat}
\bibliography{bibliography.bib}

@book{AliprantisBorder1999,
  author = {Aliprantis, C. and Border, K.},
  year = {1999},
  title = {Infinite dimensional analysis: A hitchhiker’s guide},
  publisher = {Springer},
  address = {},
  edition = {}
}

@article{calvia2025optimal,
  title={Optimal Policy Characterization for a Class of Multi-Dimensional Ergodic Singular Stochastic Control Problems},
  author={Calvia, Alessandro and Cannerozzi, Federico and Ferrari, Giorgio},
  journal={Forthcoming on The Annals of Applied Probability. ArXiv preprint arXiv:2510.11158},
  year={2025}
}

@article{dianetti2026ergodic,
  title={Ergodic Mean Field Games of Singular Control with Regime-Switching},
  author={Dianetti, Jodi and Ferrari, Giorgio and Tzouanas, Ioannis},
  journal={SIAM Journal on Control and Optimization},
  volume={64},
  number={3},
  pages={1830--1858},
  year={2026},
  publisher={SIAM}
}

@article{cannerozzi2026stationary,
  title={Stationary Mean-Field singular control of an Ornstein-Uhlenbeck process},
  author={Cannerozzi, Federico},
  journal={arXiv preprint arXiv:2601.23036},
  year={2026}
}

@article{christensen2025two,
  title={Two sided ergodic singular control and mean-field game for diffusions},
  author={Christensen, S{\"o}ren and Mordecki, Ernesto and Oli{\'u}, Facundo},
  journal={Decisions in Economics and Finance},
  volume={48},
  number={1},
  pages={241--267},
  year={2025},
  publisher={Springer}
}

@article{aid2025stationary,
  title={A stationary mean-field equilibrium model of irreversible investment in a two-regime economy},
  author={A{\"\i}d, Ren{\'e} and Basei, Matteo and Ferrari, Giorgio},
  journal={Operations Research},
  volume={73},
  number={5},
  pages={2351--2374},
  year={2025},
  publisher={INFORMS}
}

@article{cao2023stationary,
  title={Stationary discounted and ergodic mean field games with singular controls},
  author={Cao, Haoyang and Dianetti, Jodi and Ferrari, Giorgio},
  journal={Mathematics of Operations Research},
  volume={48},
  number={4},
  pages={1871--1898},
  year={2023},
  publisher={INFORMS}
}

@article{cannerozzi2026cooperation,
  title={Cooperation, Correlation, and Competition in Ergodic N-Player Games and Mean-Field Games of Singular Controls: A Case Study},
  author={Cannerozzi, Federico and Ferrari, Giorgio},
  journal={Mathematics of Operations Research},
  volume={},
  number={},
  pages={},
  year={2026},
  publisher={INFORMS}
}

@article{cao2022mfgs,
  title={MFGs for partially reversible investment},
  author={Cao, Haoyang and Guo, Xin},
  journal={Stochastic Processes and their Applications},
  volume={150},
  pages={995--1014},
  year={2022},
  publisher={Elsevier}
}

@article{ferrari2025stationary,
  title={Stationary Mean-Field Games of Singular Control under Knightian Uncertainty},
  author={Ferrari, Giorgio and Tzouanas, Ioannis},
  journal={Forthcoming on Mathematics of Operations Research. ArXiv preprint arXiv:2505.08317},
  year={2025}
}

@article{CalviaFedericoFerrariGozzi2024,
  title={Existence and uniqueness results for a mean-field game of optimal investment},
  author={Calvia, A. and Federico, S. and Ferrari, G. and Gozzi, F.},
  journal={Applied Mathematics and Optimization},
  volume={94},
  pages={},
  year={2026}
}

@article{AchdouBueraLasryLionsMoll2014,
  title={Partial differential equations models in macroeconomics},
  author={Achdou, Y. and Buera, F.J. and Lasry, J.-M. and Lions, J.-L. and Moll, B.},
  journal={Philosophical Transactions of the Royal Society A},
  volume={372},
  pages={2028},
  year={2014}
}

@article{Alvarez2024,
title = {Solutions for Poissonian stopping problems of linear diffusions via extremal processes},
journal = {Stochastic Processes and their Applications},
volume = {172},
pages = {104351},
year = {2024},
author = {Luis H.R. {Alvarez E.} and Jukka Lempa and Harto Saarinen and Wiljami Sillanpää}
}

@book{Harrison1985,
  author = {Harrison, J. M.},
  year = {1985},
  title = {Brownian Motion and Stochastic Flow Systems},
  publisher = {Wiley, New York},
  address = {},
  edition = {}
}

@book{KarlinTaylor1981,
  author = {Karlin, S. and Taylor, H.M.},
  year = {1981},
  title = {A Second Course in Stochastic Processes},
  publisher = {Academic Press, New York},
  address = {},
  edition = {}
}

@article{Saarinen2023,
  title={Two-sided {P}oisson control of linear diffusions},
  author={Saarinen, H.},
  journal={Stochastics},
  volume={},
  number={},
  pages={1--23},
  year={2023},
  publisher={}
}

@article{Albrecher2011a,
  title={Randomized observation periods for the compound {P}oisson risk model: Dividends},
  author={Albrecher, H. and Cheung, E. and Thonhauser, S.},
  journal={Astin Bulletin},
  volume={41},
  number={},
  pages={645--672},
  year={2011},
  publisher={}
}

@article{Albrecher2011b,
  title={The optimal dividend barrier in the Gamma–Omega model},
  author={Albrecher, H. and Gerber, H. and Shiu, E.},
  journal={European Actuarial Journal},
  volume={1},
  number={},
  pages={43--55},
  year={2011},
  publisher={}
}

@article{Albrecher2011c,
  title={Optimal dividend-payout in random discrete time},
  author={Albrecher, H. and Bäuerle, N. and Thonhauser, S.},
  journal={Statistics and Risk Modelling},
  volume={28},
  number={},
  pages={251--276},
  year={2011},
  publisher={}
}

@article{Albrecher2013,
  title={Randomized observation times for the compound {P}oisson risk model: The discounted penalty function},
  author={Albrecher, H. and Cheung, E. and Thonhauser, S.},
  journal={Scandinavian Actuarial Journal},
  volume={6},
  number={},
  pages={424--452},
  year={2013},
  publisher={}
}

@article{Albrecher2016,
  title={Exit identities for {L}évy processes observed at {P}oisson arrival times},
  author={Albrecher, H. and Ivanovs, J. and Zhou, X.},
  journal={Bernoulli},
  volume={22},
  number={},
  pages={1364--1382},
  year={2016},
  publisher={}
}

@article{Albrecher2017,
  title={Strikingly simple identities relating exit problems for {L}évy processes under continuous and {P}oisson observations},
  author={Albrecher, H. and Ivanovs, J.},
  journal={Stochastic Processes and their Applications},
  volume={127},
  number={},
  pages={643--656},
  year={2017},
  publisher={}
}

@article{Alvarez2018,
  title={A Class of Solvable Stationary Singular Stochastic Control Problems},
  author={Alvarez, L. H. R.},
  journal={arXiv:1803.03464},
  volume={},
  number={},
  pages={},
  year={2018},
  publisher={}
}

@article{Avanzi2013,
  title={On a periodic dividend barrier strategy in the dual model with continuous monitoring of solvency},
  author={Avanzi, B. and Cheung, E. and Wong, B. and Woo, J.-K.},
  journal={Insurance: Mathematics and Economics},
  volume={52},
  number={},
  pages={98--113},
  year={2013},
  publisher={}
}

@book{BorodinSalminen,
  author = {Borodin, A. N. and Salminen, P.},
  year = {2015},
  title = {Handbook of {B}rownian Motion - Facts and Formulae},
  publisher = {Birkhäuser, Basel},
  address = {},
  edition = {2}
}

@article{DupuisWang2002,
  title={Optimal Stopping with Random Intervention Times},
  author={Dupuis, P. and Wang, H.},
  journal={Advances in Applied Probability},
  volume={34},
  number={},
  pages={141--157},
  year={2002},
  publisher={}
}

@article{LempaSaarinen2023,
  author  = {Lempa, Jukka and Saarinen, Harto},
  title   = {A Zero-Sum {P}oisson Stopping Game with Asymmetric Signal Rates},
  journal = {Applied Mathematics and Optimization},
  volume  = {87},
  pages   = {35},
  year    = {2023},
}

@article{Guo2005,
  title={Stopping at the maximum of geometric {B}rownian motion when signals are received},
  author={Guo, X. and Liu, J.},
  journal={Journal of Applied Probability},
  volume={42},
  number={},
  pages={826--838},
  year={2005},
  publisher={}
}

@article{Hobson2021,
  title={The shape of the value function under {P}oisson optimal stopping},
  author={Hobson, D.},
  journal={Stochastic Processes and their Applications},
  volume={133},
  number={},
  pages={229--246},
  year={2021},
  publisher={}
}

@article{JackZervos2006,
  title={A Singular Control Problem with An Expected and Pathwise Ergodic Performance Criterion},
  author={Jack, A. and Zervos, M.},
  journal={Journal of Applied Mathematics and Stochastic Analysis},
  volume={(82538)},
  number={},
  pages={1--19},
  year={2006},
  publisher={}
}

@article{Kunwai2022,
  title={On an Ergodic Two-Sided Singular Control Problem},
  author={Kunwai, K. and Xi, F. and  Yin, G. and Zhu, C.},
  journal={Applied Mathematics and Optimization},
  volume={86},
  number={},
  pages={26},
  year={2022},
  publisher={}
}

@article{Lange2020,
  title={Real-option valuation in multiple dimensions using {P}oisson optional stopping times},
  author={Lange, R.-J. and Ralph, D. and Støre, K.},
  journal={Journal of Financial and Quantitative Analysis},
  volume={55},
  number={},
  pages={653--677},
  year={2020},
  publisher={}
}

@article{Lempa2012,
  title={Optimal Stopping with Information Constraint},
  author={Lempa, J.},
  journal={Applied Mathematics and Optimization},
  volume={66},
  number={},
  pages={147--173},
  year={2012},
  publisher={}
}

@article{Lempa2014,
  title={Bounded Variation Control of Itô Diffusion with Exogenously Restricted Intervention Times},
  author={Lempa, J.},
  journal={Advances in Applied Probability},
  volume={46},
  number={},
  pages={102--120},
  year={2014},
  publisher={}
}

@article{LempaSaarinen2021,
  title={Ergodic Control of Diffusions with Random Intervention Times},
  author={Lempa, J. and Saarinen, H.},
  journal={Journal of Applied Probability},
  volume={58},
  number={},
  pages={1--21},
  year={2021},
  publisher={}
}

@article{Liang2015b,
  title={Funding liquidity, debt tenor structure, and creditor’s belief: An exogenous dynamic debt run model},
  author={Liang, G. and Lütkebohmert, E. and Wei, W.},
  journal={Mathematics and Financial Economics},
  volume={9},
  number={},
  pages={271--302},
  year={2015},
  publisher={}
}

@article{LiangSun2019,
  title={{D}ynkin games with {P}oisson random Intervention Times},
  author={Liang, G. and Sun, H.},
  journal={SIAM Journal on Control and Optimization},
  volume={57},
  number={},
  pages={2962--2991},
  year={2019},
  publisher={}
}

@article{LiangSun2020,
  title={Risk-Sensitive {D}ynkin Games with Heterogeneous {P}oisson Random Intervention Times},
  author={Liang, G. and Sun, H.},
  journal={arXiv:2008.01787},
  volume={},
  number={},
  pages={},
  year={2020},
  publisher={}
}

@article{Matsumoto2006,
  title={Optimal Portfolio of Low Liquid Assets with a Log-utility Function},
  author={Matsumoto, K.},
  journal={Finance and Stochastics},
  volume={10},
  number={},
  pages={121--145},
  year={2006},
  publisher={}
}

@article{Matsumoto2007,
  title={Portfolio insurance with liquidity risk},
  author={Matsumoto, K.},
  journal={Asia-Pacific Financial Markets},
  volume={14},
  number={},
  pages={363--386},
  year={2007},
  publisher={}
}

@article{Perez2021,
  title={Nonzero-sum optimal stopping game with continuous vs. periodic exercise opportunities},
  author={Pérez, J.L. and Rodosthenous, N. and Yamazaki, K.},
  journal={Mathematics of Operations Research},
  volume={50},
  number={},
  pages={1832--1867},
  year={2021},
  publisher={}
}

@article{Pham2008,
  title={A model of optimal consumption under liquidity risk with random trading times},
  author={Pham, H. and Tankov, P.},
  journal={Mathematical Finance},
  volume={18},
  number={},
  pages={613--627},
  year={2008},
  publisher={}
}

@article{Pham2009,
  title={Investment/consumption choice in illiquid markets with random trading times},
  author={Pham, H.},
  journal={Radon Series on Computational and Applied Mathematics},
  volume={8},
  number={},
  pages={1--17},
  year={2009},
  publisher={}
}

@article{Pham2009b,
  title={A coupled system of integrodifferential equations arising in liquidity risk model},
  author={Pham, H. and Tankov, P.},
  journal={Applied Mathematics and Optimization},
  volume={59},
  number={},
  pages={147--173},
  year={2009},
  publisher={}
}

@article{RogersZane2002,
  title={A Simple Model of Liquidity Effects},
  author={Rogers, L. C. G. and Zane, O.},
  journal={Advances in Finance and Stochastics},
  volume={},
  number={},
  pages={161--176},
  year={2002},
  publisher={Springer-Verlag, Berlin Heidelberg}
}

@article{Wang2001,
  title={Some Control Problems with Random Intervention Times},
  author={Wang, H.},
  journal={Advances in Applied Probability},
  volume={33},
  number={},
  pages={404--422},
  year={2001},
  publisher={}
}

@article{Yu2021,
  title={Randomized observation periods for compound {P}oisson risk model with capital injection and barrier dividend},
  author={Yu, W. and Guo, P. and Wang, Q. and Guan, G. and Huang, Y. and Yu X.},
  journal={Advances in Difference Equations},
  volume={2021},
  number={},
  pages={220},
  year={2021},
  publisher={}
}

@article{Zhao2017,
  title={Optimal periodic dividend and capital injection problem for spectrally positive {L}évy processes},
  author={Zhao, Y. and Chen, P. and Yang, H.},
  journal={Insurance: Mathematics and Economics},
  volume={74},
  number={},
  pages={135--146},
  year={2017},
  publisher={}
}

@book{kyprianou2014,
  author = {Kyprianou, Andreas E.},
  title = {Fluctuations of {L{\'e}vy} Processes with Applications},
  edition = {2},
  series = {Universitext},
  publisher = {Springer},
  address = {Berlin, Heidelberg},
  year = {2014},
}

\end{document}